\documentclass[a4paper]{amsart}

\usepackage{txfonts, amsmath,amstext,amsthm,amscd,amsopn,verbatim,amssymb, amsfonts,enumerate,graphics, todonotes}
\usepackage{tikz-cd}

\DeclareMathAlphabet\mathbb{U}{msb}{m}{n}

\usepackage[bbgreekl]{mathbbol}

\usepackage{tikz}
\usetikzlibrary{matrix}
\usetikzlibrary{shapes}
\usetikzlibrary{arrows}
\usetikzlibrary{calc,3d}
\usetikzlibrary{decorations,decorations.pathmorphing}
\usetikzlibrary{through}
\tikzset{ext/.style={circle, draw,inner sep=1pt},int/.style={circle,draw,fill,inner sep=1pt},nil/.style={inner sep=1pt}}
\tikzset{exte/.style={circle, draw,inner sep=3pt},inte/.style={circle,draw,fill,inner sep=3pt}}
\tikzset{diagram/.style={matrix of math nodes, row sep=3em, column sep=2.5em, text height=1.5ex, text depth=0.25ex}}
\tikzset{diagram2/.style={matrix of math nodes, row sep=0.5em, column sep=0.5em, text height=1.5ex, text depth=0.25ex}}

\theoremstyle{plain}
  \newtheorem{thm}{Theorem}
  \newtheorem{defi}[thm]{Definition}
\newtheorem{conj}[thm]{Conjecture}
  \newtheorem{prop}[thm]{Proposition}
  \newtheorem{defprop}[thm]{Definition/Proposition}
  \newtheorem{cor}[thm]{Corollary}
  \newtheorem{convention}[thm]{Convention}
  \newtheorem{lemma}[thm]{Lemma}
\theoremstyle{definition}
  
  \newtheorem{rem}[thm]{Remark}

\newcommand{\aor}{\begin{tikzpicture}
\draw[->] (0,0) arc (190:190+340:0.15cm);
\end{tikzpicture}}
\newcommand{\aol}{\begin{tikzpicture}
\draw[->] (0,0) arc (-10:-360:0.15cm);
\end{tikzpicture}}
\newcommand{\cutaol}{\begin{tikzpicture}[scale=0.7]
\draw[->] (0,0) arc (-20:-350:0.15cm);
\draw (-0.4,0.25) -- (0.09,-0.18);
\end{tikzpicture}}
\newcommand{\alg}[1]{\mathfrak{{#1}}}

\newcommand{\ad}{{\text{ad}}}

\newcommand{\p}{\partial}

\newcommand{\C}{{\mathbb{C}}}
\newcommand{\R}{{\mathbb{R}}}

\newcommand{\Graphs}{{\mathsf{Graphs}}}
\newcommand{\pdGraphs}{{{}^*\mathsf{Graphs}}}
\newcommand{\BVGraphs}{{\mathsf{BVGraphs}}}
\newcommand{\BVGra}{{\mathsf{BVGra}}}

\newcommand{\Gra}{{\mathsf{Gra}}}
\newcommand{\pdGra}{{^*\mathsf{Gra}}}

\newcommand{\tgraphs}{{\mathsf{^*graphs}}}
\newcommand{\Pois}{{\mathsf{Pois}}}

\newcommand{\mF}{\mathcal{F}}

\newcommand{\op}{\mathcal}

\newcommand{\Lie}{\mathsf{Lie}}
\newcommand{\Com}{\mathsf{C}}
\newcommand{\La}{\Lambda}

\newcommand{\hoLie}{\mathsf{L}_\infty}

\newcommand{\SO}{\mathrm{SO}(2)}

\newcommand{\coCom}{\mathsf{coComm}}

\newcommand{\FM}{\mathsf{FM}}
\newcommand{\FFM}{\mathsf{FFM}}

\newcommand{\bpm}{\begin{pmatrix}}
\newcommand{\epm}{\end{pmatrix}}

\newcommand{\GC}{\mathrm{GC}}
\newcommand{\fGC}{\mathrm{fGC}}
\newcommand{\BG}{\mathrm{BG}}

\newcommand{\tBG}{{\widetilde {\mathrm{BG}}}}

\newcommand{\sBG}{\mathrm{sBG}}

\DeclareMathOperator{\gr}{gr}
\DeclareMathOperator{\Tw}{Tw}
\DeclareMathOperator{\Sh}{Sh}
\DeclareMathOperator{\coker}{coker}
\DeclareMathOperator{\Hom}{Hom}
\DeclareMathOperator{\Harr}{Harr}
\DeclareMathOperator{\Conf}{Conf}

\newcommand{\cP}{\mathcal{P}}
\newcommand{\stG}{{}^*\Graphs}

\newcommand{\id}{\mathrm{id}}

\newcommand{\notadp}
{{
\begin{tikzpicture}[baseline=-.55ex,scale=.2, every loop/.style={}]
 \node[circle,draw,fill,inner sep=.5pt] (a) at (0,0) {};
 \draw (a) edge[loop] (a);
 \draw (-.2,-.2) -- (.2,.5);
\end{tikzpicture}}}

\renewcommand{\cutaol}{\notadp}

\begin{document}
\title{A model for configuration spaces of points}

\author{Ricardo~Campos}
\address{Ricardo~Campos:  IMAG, Univ. Montpellier, CNRS, Montpellier, France}
\email{ricardo.campos@umontpellier.fr}

\author{Thomas~Willwacher}
\address{Thomas~Willwacher: Department of Mathematics, ETH Zurich, Zurich, Switzerland}
\email{thomas.willwacher@math.ethz.ch}

\keywords{}

\begin{abstract}
We construct a real combinatorial model for the configuration spaces of points of compact smooth oriented manifolds without boundary. We use these models to show that the real homotopy type of configuration spaces of a simply connected such manifold only depends on the real homotopy type of the manifold. 

Moreover, we show that for framed $D$-dimensional manifolds these models capture a natural right homotopy action of the little $D$-disks operad. 

\end{abstract}
\maketitle
\setcounter{tocdepth}{1}
\tableofcontents

\section{Introduction}

Given a smooth manifold $M$, we study the configuration space of $n$ non-overlapping points on $M$
$$\Conf_n(M) = \{ (m_1,\dots,m_n)\in M^n \mid  m_i\neq m_j \text{ for } i\neq j \}.$$
These spaces are classical objects in topology, whose homological and homotopical properties have been subject to intensive study over the decades. One of the first important results dates back to 1978 when Cohen and Taylor \cite{CT} constructed a spectral sequence converging to the cohomology $H^\bullet(\Conf_n(M))$. A different spectral sequence was constructed by Bendersky and Gitler \cite{BT} and both spectral sequences have been shown to coincide from the $E^2$ term on by Felix and Thomas \cite{FT}. 
In the particular case of smooth compact projective complex manifolds it was shown by Totaro \cite{To} that the Cohen-Taylor spectral sequence collapses after the second page and Kriz \cite{Kr} showed that for those manifolds the $E^2$ page is actually a model of $\Conf_n(M)$ in the sense of rational homotopy theory.

In this paper, we aim to understand the rational homotopy type of configuration spaces. Classical rational homotopy theory \`a la Sullivan \cite{Su}  states that we can understand topological spaces via algebraic \textit{models} which are differential graded commutative $\mathbb K$-algebras (dgca), where $\mathbb K$ is a field of characteristic zero. This roughly amounts to capturing the non-torsion part of the homotopy groups of such spaces. Usually, the field $\mathbb K$ is taken to be the rational numbers, but due to the transcendental methods we use, we take the base field $\mathbb K=\mathbb R$ to be the real numbers and we will therefore refer to the \emph{real }homotopy type of configuration spaces.

Our first main result is the construction of a differential graded commutative $\mathbb R$-algebra model $\stG_M$ for $\Conf_{\bullet}(M)$, in the case when $M$ is a $D$-dimensional compact smooth oriented manifold without boundary, with $D\geq 2$. Our model depends on $M$ only through the following data:
\begin{itemize}
	\item The cohomology $V=H^\bullet(M)$ as a vector space with a non-degenerate pairing of degree $D=\dim (M)$.
	\item The partition function $Z_M$ of the ``universal'' perturbative AKSZ topological field theory on $M$. This is a Maurer-Cartan element in a certain graph complex only depending on $V$.
\end{itemize}
In particular, this shows that the latter perturbative invariants $Z_M$ (special cases of which have been studied in the literature \cite{BCM}) contain at least as much information as the real homotopy type of $\Conf_\bullet(M)$.
Furthermore, the real homotopy type of $M$ is encoded in the tree-level components of $Z_M$. The higher loop order pieces of $Z_M$ ``indicate'' (in a vague sense) the failure of the homotopy type of $\Conf_\bullet(M)$ to depend only on $M$.
Finally, the real cohomology of $\Conf_\bullet(M)$ can be computed just from the tree level knowledge, see section \ref{sec:Lambrechts and Stanley}. 

Now suppose that $M$ is furthermore framed, i.e., the frame bundle of $M$ is trivialized.
Then the totality of spaces $\Conf_\bullet(M)$ carries additional algebraic structure, in that it can be endowed with a homotopy right action of the little $D$-disks operad as follows.
First we consider the natural compactification $\FM_M(n)$ of $\Conf_n(M)$ introduced by Axelrod and Singer \cite{AS}, cf. also \cite{Si}.
It is naturally acted upon from the right by the Fulton-MacPherson-Axelrod-Singer variant of the little disks operad $\FM_D$ introduced in \cite{GJ} by ``insertion'' of configurations of points.

The right $E_D$-module structure on configuration spaces has been receiving much interest in the last decade, since it has been realized that the homotopy theory of these right modules captures much of the homotopy theory of the underlying manifolds. For example, by the Goodwillie-Weiss embedding calculus the derived mapping spaces (``Ext's'') of those right $E_D$ modules capture (under good technical conditions) the homotopy type of the embedding spaces of the underlying manifolds \cite{GW,BW, BW2}. Dually, the factorization homology (``Tor's'') of $E_D$-algebras has been widely studied and captures interesting properties of both the manifold and the $E_D$ algebra \cite{AF}.
However, in order to use these tools in concrete situations it is important to have models for $\Conf_\bullet(M)$ (as a right Hopf $E_D$-module) that are computationally accessible, i.e., combinatorial.
In this paper we provide such models.

Concretely, our second main result is that our model $\stG_M$ above combinatorially captures this action of the little $D$-disks operad as well, in the sense that it is a right Hopf operadic comodule over the Kontsevich Hopf cooperad $\stG_D$, modeling the topological little $D$-disks operad, and the combinatorially defined action models the topological action of $E_D$ on $\Conf_\bullet(M)$.

In fact, one can consider the following ``hierarchy" of invariants of a manifold $M$.
\begin{enumerate}
	\item \label{hoi:it1} The real (or rational) homotopy type of $M$.
	\item \label{hoi:it2} The real (or rational) homotopy types of $\FM_M(m)$ for $m=1,2,\dots$
	\item \label{hoi:it3} The real (or rational) homotopy type of $\FM_M$ considered as right $\FM_D$-module, for parallelized $M$. (For non-parallelizable $M$ one may consider similarly the homotopy type of the $\FM$-module of framed configuration spaces of points $\FFM_M$.)
\end{enumerate} 
 
 The relative strength of this invariants has been unknown. In particular it is a long standing open problem if for simply connected $M$ the rational homotopy type of $\Conf_\bullet(M)$ depends only on the rational homotopy type of $M$ \cite[Problem 8, p. 518]{FHT} (cf. also \cite{Le} for a stronger conjecture disproved in \cite{LoS}). 

In our model the above hierarchy is nicely encoded in the loop order filtration on a certain graph complex $\GC_M$, in which item \ref{hoi:it1} is encoded by the tree level piece of $Z_M$ along with the cohomology of item \ref{hoi:it2}, while the full $Z_M$ encodes item \ref{hoi:it3}.

Our third main result states that for a simply connected smooth closed framed manifold $M$, these invariants are of equal strenght. We show furthermore that without the framed assumption item \ref{hoi:it1} is still equally strong as item \ref{hoi:it2}, thus establishing \cite[Problem 8, p. 518]{FHT} under the assumption of smoothness.

Finally, if we consider a non-parallelized manifold there is still a way to make sense of the insertion of points at the boundary, but the price to pay is that one has to consider configurations of \textit{framed points} in $M$. 
The resulting framed configuration spaces $\Conf^{fr}_\bullet(M)$ then come equipped with a natural right action of the framed little disks operad $E_D^{fr}$. 
In Section \ref{sec:framed case} we present $\BVGraphs_M$, a natural modification of $\Graphs_M$ that encompasses the data of the frames and we show that if we consider $\Sigma$ a two dimensional orientable manifold, $\BVGraphs_\Sigma$ models this additional right action.
In the framed case we restrict ourselves to the 2-dimensional setting.

\subsection*{Outline and statement of the main result}
Let us summarize the construction and state the main result here.
First recall from \cite{K2} the Kontsevich dg cooperad $\stG_D$. Elements of $\stG_D(r)$ consist of linear combinations of graphs with $r$ numbered and an arbitrary number of unidentifiable vertices, like the following 
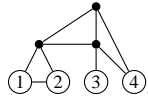
\[
\begin{tikzpicture}
\node[ext] (v1) at (0,0) {$\scriptstyle 1$};
\node[ext] (v2) at (.5,0) {$\scriptstyle 2$};
\node[ext] (v3) at (1,0) {$\scriptstyle 3$};
\node[ext] (v4) at (1.5,0) {$\scriptstyle 4$};
\node[int] (w1) at (.25,.5) {};
\node[int] (w2) at (1,1) {};
\node[int] (w3) at (1,.5) {};
\draw (v1) edge (v2) edge (w1) (w1) edge (w2) edge (w3) edge (v2) (v3) edge (w3) (v4) edge (w3) edge (w2) (w2) edge (w3);
\end{tikzpicture}
.
\]
The precise definition of $\stG_D$ will be recalled in section \ref{sec:graphs} below.
The graphs contributing to $\stG_D$ may be interpreted as the non-vaccuum Feynman diagrams of the perturbative AKSZ $\sigma$-models on $\R^D$ \cite{AKSZ}.

Kontsevich constructs an explicit map $\stG_D\to \Omega_{PA}(\FM_D)$ to the dgca of PA forms on the compactified configuration spaces $\FM_D$.
This map is compatible with the (co-)operadic compositions, in the sense described in section \ref{sec:graphs} below.

Now fix a smooth compact manifold $M$ of dimension $D$, of which we pick an algebraic realization, so that we can talk about PA forms $\Omega_{PA}(M)$.
Then we consider a collection of dg commutative algebras $\stG_M(r)$. Elements of $\stG_M(r)$ are linear combinations of graphs, but with additional decorations of each vertex in the symmetric algebra $S(\tilde H(M))$ generated by the reduced cohomology $\tilde H(M)$. The following graph is an example, where we fixed some basis $\{\omega_j\}$ of $\tilde H(M)$.
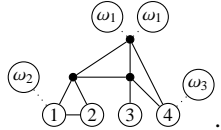
\[
\begin{tikzpicture}
\node[ext] (v1) at (0,0) {$\scriptstyle 1$};
\node[ext] (v2) at (.5,0) {$\scriptstyle 2$};
\node[ext] (v3) at (1,0) {$\scriptstyle 3$};
\node[ext] (v4) at (1.5,0) {$\scriptstyle 4$};
\node[int] (w1) at (.25,.5) {};
\node[int] (w2) at (1,1) {};
\node[int] (w3) at (1,.5) {};
\node[ext] (i1) at (.7,1.3) {$\scriptstyle \omega_1$};
\node[ext] (i2) at (1.3,1.3) {$\scriptstyle \omega_1$};
\node[ext] (i3) at (-.4,.5) {$\scriptstyle \omega_2$};
\node[ext] (i4) at (1.9,.4) {$\scriptstyle \omega_3$};
\draw (v1) edge (v2) edge (w1) (w1) edge (w2) edge (w3) edge (v2) (v3) edge (w3) (v4) edge (w3) edge (w2) (w2) edge (w3);
\draw[dotted] (v1) edge (i3) (w2) edge (i2) edge (i1) (v4) edge (i4);
\end{tikzpicture}
.
\]
These graphs may be interpreted as the non-vaccuum Feynman diagrams of the perturbative AKSZ $\sigma$-model on $M$.
We equip the spaces $\stG_M(r)$ with a non-trivial differential built using the partition function $Z_M$ of those field theories.
This partition function can be considered as a special Maurer-Cartan element of a certain graph complex $\GC_M$.
Algebraically, the spaces $\stG_M(r)$ assemble into a right dg Hopf cooperadic comodule over the Hopf cooperad $\stG_D$.

By mimicking the Kontsevich construction, we construct, for a parallelized manifold $M$, a map of dg Hopf collections\footnote{A (dg) Hopf collection $\op C$ for us is a sequence $\op C(r)$ of dg commutative algebras, with actions of the symmetric groups $S_r$. A (dg) Hopf cooperad is a cooperad in dg commutative algebras.}
\[
\stG_M\to \Omega_{PA}(\FM_M),
\]
compatible with the (co)operadic (co)module structure, where we consider $\FM_M$ as equipped with the right $\FM_D$-action.
If $M$ is not parallelized, we do not have an $\FM_D$-action on $\FM_M$. Nevertheless we may consider a (quasi-isomorphic) dg Hopf collection
\[
 \stG_M^\notadp \subset \stG_M
\]
that still comes with a map of dg Hopf collections
\[
\stG_M^\notadp \to \Omega_{PA}(\FM_M).
\]

Our first main result is the following.
\begin{thm}\label{thm-intro:main}
The map $\stG_M^\notadp \to \Omega_{PA}(\FM_M)$ is a quasi-isomorphism of dg Hopf collections.
In the parallelized case the map $\stG_M\to \Omega_{PA}(\FM_M)$ is a quasi-isomorphism of dg Hopf collections, compatible with the (co)operadic (co)module structures.
\end{thm}

This result provides us with explicit combinatorial dgca models for configuration spaces of points, compatible with the right $E_D$ action on these configuration spaces in the parallelizable setting. An extension to the non-parallelized case is provided in section \ref{sec:framed case}, albeit only in dimension $D=2$.

We note that our model $\stG_M$ depends on $M$ only through the partition function $Z_M\in \GC_M$. The tree part of this partition function encodes the real homotopy type of $M$. The loop parts encode invariants of $M$.
Now, simple degree counting arguments may be used to severely restrict the possible graphs occuring in $M$.
In particular, one finds that if $H^1(M,\R)$ vanishes, then for $D\geq 4$ there are no contributions to $Z_M$ of positive loop order, and one hence arrives at the following result.

\begin{cor}[Theorem \ref{thm:ht only dep on htM} below]\label{cor-intro}
 Let $M$ be an orientable compact manifold without boundary of dimension $D\geq 4$, such that $H^1(M,\R)=0$. Then the (naive\footnote{We call the naive real homotopy type the quasi-isomorphism type of the dg commutative algebra of (PL or smooth) forms. Note that in the non-simply connected case this definition is not the correct one, one should rather consider the real homotopy type of the universal cover with the action of the fundamental group. We do not consider this better notion here, and in this paper ``real homotopy type'' shall always refer to the naive real homotopy type. }) real homotopy type of $\Conf_\bullet(M)$ depends only on the (naive) real homotopy type of $M$.
\end{cor}

For $D=2$ the analogous statement is empty, as there is only one connected manifold satisfying the assumption.
If we replace the condition $H^1(M,\R)=0$ by the stronger condition of simple connectivity, the statement is also true in dimension 3, but for the trivial reason that by the Poincar\'e conjecture there is only one simply connected manifold $M$ in dimension 3.
Hence the above result also solves the real version of the long standing question in algebraic topology of whether for simply connected $M$ the rational homotopy type of the configuration space of points on $M$ is determined by the rational homotopy type of $M$, cf. \cite[Problem 8, p.518]{FHT}

\begin{rem}
Our result also shows that the ``perturbative AKSZ''-invariant $Z_M$ is at least as strong as the invariant of $M$ given by the totality of the real homotopy types of the configuration spaces of $M$, considered as right $E_D$-modules. The latter ``invariant'' is the data entering the factorization or ``manifoldic'' homology \cite{AF, MT} and the Goodwillie-Weiss calculus \cite{GW} (over the reals). 
Conversely, from the fact that the models $\stG_M$ encode the real homotopy type of configuration spaces one may see that the expectation values of the perturbative AKSZ theories on $M$ may be expressed through the factorization homology of $M$.
However, we will leave the physical interpretation to forthcoming work and focus here on the algebraic-topological goal of providing models for configuration spaces.
\end{rem}

\begin{rem}
After the first version of this article appeared on the arXiv, Idrissi \cite{I} obtained results very similar to ours by showing that for simply connected closed oriented manifolds the Lambrechts-Stanley dg model \cite{LS} is actually a real model of $\Conf_n(M)$. 
We sketch in Appendix \ref{sec:LS model and ours} how this latter statement can also be obtained as a consequence of our main results.
\end{rem}

\subsection*{Plan of the paper}

In Section \ref{sec:FM_M} we introduce the spaces $\FM_M$, the compactifications of configuration spaces of points on a smooth manifold $M$ ($D= \dim M$)  and its semi-algebraic realizations and adapt results in the literature to construct the propagator.

Starting with the framed case, in Section \ref{sec:graphs} we construct the first graph complex $\pdGra_M$ (a Hopf $\pdGra_D$-comodule) and construct the map into $\Omega_{PA}(\FM_M)$ which is not yet a quasi-isomorphism.
In Section \ref{sec: twisting} we use operadic twisting to obtain the graph complex $\stG_M$ and in Section \ref{sec:main proof} we show that $\stG_M$ is indeed a model for the real homotopy type of $\FM_M$ as a right $\FM_D$-module.

In Section \ref{sec:non-framed} we construct a no-tadpole variant of the graph complex to deal with the case where $M$ is not parallelized and show that it models the real homotopy type of the collection of topological spaces $\FM_M$, concluding the proof of Theorem \ref{thm-intro:main}.

The next goal is to study the dependence of the homotopy type of the configuration spaces on the base manifold. In Sections \ref{sec:Lambrechts and Stanley} and \ref{sec:proof of conjecture} we study the partition function $Z_M$ that gives rise to the differential in $\stG_M$ and we show that it is gauge equivalent to one vanishing on graphs containing $\leq 2$-valent vertices. We conclude that in good conditions the real homotopy type of $M$ can be recovered from the tree piece of the graph complex, thus proving Corollary \ref{cor-intro}.

Finally, in the last section we construct a graphical model of configuration spaces of framed points in $2$-dimensions, together with the action of the framed little disks operad.

\subsection{Notations and conventions.}

Throughout the text all algebraic objects (vector spaces, algebras, operads, etc) are differential graded (or just dg) and are defined over the field $\mathbb R$.

We use cohomological conventions, i.e. all differentials have degree $+1$. We use the language of operads and follow mostly the conventions of the textbook \cite{lodayval}.
One notable exception is that we denote the $k$-fold operadic (de-)suspension of an operad $\op P$ by $\La^k\op P$.


\subsection{Acknowledgments}
We would like to thank Pascal Lambrechts for useful remarks and references and Najib Idrissi, Nils Prigge and Victor Turchin for valuable discussions and for pointing out some mistakes in the original version.
Both authors have been supported by the Swiss National Science Foundation, grant 200021\_150012, by the NCCR SwissMAP funded by the Swiss National Science Foundation, and the ERC starting grant 678156 (GRAPHCPX).

\section{Compactified configuration spaces}\label{sec:FM_M}

\subsection{Semi-algebraic Manifolds}

Given a compact semi-algebraic set $X$ one can consider its dgca of piecewise semi-algebraic (PA) forms, $\Omega_{PA}(X)$, which is quasi-isomorphic to Sullivan's dgca of piecewise polynomial forms \cite{HLTV, KS}.

Dually, one can also consider its complex of semi-algebraic chains, that we denote by $Chains(X)$, which is also quasi-isomorphic to the usual complex of singular chains.

By the Nash--Tognoli Theorem \cite{Tog} (see also \cite[section 14]{BCR}), any smooth compact manifold is diffeomorphic to a component of a non-singular algebraic subset of $\mathbb R^N$ for some $N$. 
In particular, any such manifold can be realized as a smooth semi-algebraic (i.e., Nash-)submanifold of Euclidean space.
Throughout this paper whenever we consider a closed smooth manifold $M$ we will consider implicitly a chosen such realization of $M$ as a Nash submanifold of $\R^N$.

We refer to \cite{BCR} for an introduction to real algebraic geometry. An overview is also contained in the introductory sections of \cite{HLTV}.

	Even though all manifolds considered in this paper will be smooth, it is not sufficient for our purposes to consider the de Rham complex. The main reason for this is that we would like to consider fiber integration over non-smooth fiber bundles $E\to B$. 
	Nonetheless, the relevant bundles will be SA (semi-algebraic) bundles \cite{HLTV} and, for such bundles, there is a pushforward map $\Omega_{min}(E) \to \Omega_{PA}(B)$, where $\Omega_{min}(M) \subset \Omega_{PA}(M)$ is the (non-quasi-isomorphic) subalgebra of minimal forms. 
	
	While this pushforward cannot be naturally extended to the whole space of PA forms $\Omega_{PA}(E)$, as described in Appendix \ref{app:PA}, we can consider a subalgebra of trivial forms $\Omega_{triv}(E)$, sitting between $\Omega_{min}(E)$ and $\Omega_{PA}(E)$ and quasi-isomorphic to $\Omega_{PA}(E)$, such that the pushforward extends naturally to a map $\Omega_{triv}(E) \to \Omega_{PA}(B)$.

\subsection{Configuration spaces of points in $\mathbb R^D$}

Let $D$ be a positive integer. We will use the Fulton-MacPherson topological operad $\FM_D$ that was introduced by Getzler and Jones \cite{GJ}. Its $n$-ary space $\FM_D(n)$ is a suitable compactification of the quotient of the configuration space 
$
\Conf_n(\R^D)/(\mathbb{R}_{>0} \ltimes \mathbb{R}^D) $, with the Lie group $\mathbb{R}_{>0} \ltimes \mathbb{R}^D$ acting by scaling and translations. For $n>1$ the spaces $\FM_D(n)$ are $Dn-D-1$ dimensional manifolds with corners whose  boundary strata represent sets of points getting infinitely close.

The first few terms are
\begin{itemize}
	\item $\FM_D(0) = \{*\}$,\footnote{We work with the unital version of the Fulton-MacPherson operad.}
	\item $\FM_D(1) = \{*\}$,
	\item $\FM_D(2) = S^{D-1}$.
\end{itemize}
The operadic composition $\circ_i$ is given by inserting a configuration at the boundary stratum at the point labeled by $i$.
A thorough study of these operads can be found in \cite{LV}.

The operad $\FM_D$ can be related to a shifted version of the homotopy Lie operad via the operad morphism

\begin{equation}\label{eq:L_infty to FM_D} 
 \La^{D-1} L_\infty \to Chains(\FM_D),
\end{equation}
given by sending the generator $\mu_n \in \La^{D-1} L_\infty(n)$ to the fundamental chain of $\FM_D(n)$, i.e. the semi-algebraic chain corresponding to $\FM_D(n)$ as a submanifold of itself.\footnote{Recall that due to our cohomological conventions these spaces live in non-positive degree. In particular, the generator $\mu_n\in L_\infty$ has degree $2-n$.}

\subsection{Configuration spaces of points on a manifold}
Let $M$ be a closed smooth oriented manifold of dimension $D$. We denote by $\Conf_n(M)$, the configuration space of $n$ points in $M$. Concretely, $\Conf_n(M) = M^n	 -	\Delta$, where $\Delta$ is the fat (or long) diagonal $\Delta = \{ (m_1,\dots,m_n)\in M^n \mid  \exists i\neq j: m_i= m_j  \}$.

The Fulton-MacPherson-Axelrod-Singer compactification of $\Conf_n(M)$ is a smooth manifold with corners $\FM_M(n)$ whose boundary strata correspond to nested groups of points becoming ``infinitely close", cf. \cite{Si} for more details and a precise definition. Since the inclusion $\Conf_n(M) \hookrightarrow \FM_M(n)$ is a homotopy equivalence we work preferably with $\FM_M(n)$ as these spaces have a richer structure.

\begin{convention}[Semi-algebraicity of $\FM_M(n)$ ]
The choice of semi-algebraic structure on $\FM_M(n)$ is done in a way compatible with the one from $M$ as follows: 
Let us consider the chosen semi-algebraic realization of the manifold $M$ in $\mathbb R^N$ for some $N$. 

For $1\leq i\ne j\leq n$, let $\theta_{i,j} \colon \Conf_n(M) \to S^{N-1}$ sending $(x_1,\dots,x_n)$ to $\frac{x_i-x_j}{\|x_i - x_j\|_{\mathbb R^N}}$.

For $1\leq i\ne j\ne k\leq n$ we define $d_{i,j,k} \colon \Conf_n(M) \to (0,+\infty)$ by $d_{i,j,k}((x_1,\dots,x_n))=\frac{\|x_i - x_j\|}{\|x_i - x_k\|}$.

Considering all possibilities of $i$, $j$ and $k$, we have defined a natural embedding $$\iota\colon \Conf_n(M)\to M^n \times (S^{N-1})^{n(n-1)} \times [0,+\infty]^{n(n-1)(n-2)}.$$  

We define $\FM_M(n)$ as the closure $\iota(\Conf_n(M))$ inheriting thus a semi-algebraic structure.
\end{convention}

\begin{rem}[SA bundles]

For every $m>n$ there are various projection maps $\FM_M(m) \to \FM_M(n)$ corresponding to forgetting $m-n$ of the points. These maps are not smooth fiber bundles, but they are SA (semi-algebraic) bundles \cite{HLTV}, which allows us to consider pushforwards (fiber integration) of forms along these maps. 

The proof of this fact is a straightforward adaptation of the proof of the same fact for $\FM_D$ done in \cite[Section 5.9]{LV}.
 In this case one starts instead by associating to a configuration in $\FM_M(n)$ a configuration of nested disks in $M$.

\end{rem}

\begin{convention}\label{convention}
From here onward, we fix representatives of the cohomology of $M$, i.e., we fix an embedding
\begin{equation}\label{equ:idef}
\iota:
H^\bullet(M) \hookrightarrow \Omega_{triv}^\bullet(M)
\end{equation}
that yields a right inverse of the projection from closed forms to cohomology.
\end{convention}

\subsubsection{The diagonal class}

 Since $M$ is compact and oriented, the pairing $\int \colon H^\bullet(M)\otimes H^\bullet(M)\to \mathbb R$, $(\omega,\nu) \mapsto \int_M\omega\wedge\nu$ given by Poincar\'e duality is non-degenerate. 
 We shall also consider a version of this pairing which is antisymmetric for odd $D$,
 \[
\langle \omega,\nu\rangle=
(-1)^{D\, \deg(\nu)} \int_M\omega\wedge\nu\,.
 \]

The diagonal map $\Delta \colon M\to M\times M$ defines an element in $H_\bullet (M\times M)$ and its dual under Poincar\'e duality is called the diagonal class, which is also denoted by $\Delta \in H^\bullet(M\times M)= H^\bullet(M)\otimes H^\bullet(M)$.

If we pick a homogeneous basis $e_1,\dots, e_k$ of $H^\bullet(M)$, we have $\Delta= \sum_{i,j}g^{ij} e_i \otimes e_j$, where $(g^{ij})$ is the matrix inverse to the Poincar\'e duality pairing $\langle-,-\rangle$. Alternatively, this can also be written as $\Delta= \sum_{i}(-1)^{\deg(e_i)} e_i \otimes e_i^*$, where $\{e_i^*\}$ is the dual basis of $\{e_i\}$ with respect to $(-,-)$. 

In $\FM_M(2)$, if we consider the case in which the two points come infinitely close to one another, we obtain a map $\partial \FM_M(2)\to M \cong \Delta \subset M\times M$ which is a sphere bundle (with $S^{D-1}$ fibers). Notice that $\partial \FM_M(2)$ can be identified with $ST(M)$, the sphere tangent bundle of $M$.

The following proposition can essentially be found in the literature (c.f. for instance \cite[Section 3]{BC}, \cite{CR} and \cite[Lemma 2]{CM}), we only have to apply minor modifications in order to work in the semi-algebraic setting.
\begin{prop}\label{prop:angular form}
Let $p_1\colon \FM_M(2) \to M$ (respectively $p_2 \colon \FM_M(2) \to M$) be the map that forgets the point labeled by $2$ (resp. $1$) from a configuration.
There is a form $\phi_{12}\in \Omega_{triv}^{D-1}(\FM_M(2))$ 
 satisfying the following properties:
\begin{enumerate}[(i)]

\item $d\phi_{12} = p_1^* \wedge p_2^* (\Delta) = \sum_{i,j}g^{ij}p_1^*(e_i) \wedge p_2^*(e_j) \in \Omega_{triv}^D(\FM_M(2))$,

\item The fiber integral of the restriction of $\phi_{12}$ to $\partial \FM_M(2)$ is equal to $(-1)^D$. (We then say that this restriction is a global angular form.)  Additionally, if $D=2$, the restriction of $\phi_{12}$ to every fiber of the circle bundle yields a round volume form of that circle, with respect to some metric.

\item The form $\phi_{12}$ is symmetric with respect to the $\mathbb Z_2$ action induced by swapping points $1$ and $2$ for $D$ even and antisymmetric for $D$ odd.

\item For any $\alpha\in H^\bullet(M)$,
$$
\int_2  \phi_{12} p_2^* \iota(\alpha) =0
$$
where $\iota$ is as in \eqref{equ:idef} and the integral is along the fiber of $p_1$, i.e., one integrates out the second coordinate.

\end{enumerate}

\end{prop}

Notice that the form $\phi_{12} \in \Omega_{triv}^{D-1}(\FM_M(2))$ is also called the propagator in the literature. More generally, we consider the forms $\phi_{ij}\in \Omega_{triv}^{D-1} (\FM_M(n))$ to be $p_{ij}^*(\phi_{12})$, where $p_{ij} \colon \FM_M(n) \to \FM_M(2)$ is the projection map that remembers only the points labeled by $i$ and $j$.

\begin{proof}
Let $\psi  \in \Omega_{triv}^{D-1}(\partial \FM_M(2))$ be a global angular form of the sphere bundle. Such a form always exists, see for example \cite{BT} where such construction is made in the smooth case, but the argument can be adapted to the semi-algebraic case. It is also shown in \cite{BT} that for a circle bundle the global angular form can be chosen to restrict to the standard volume form on each fiber. Moreover, the differential of such a form is basic (it is the pullback of the Euler class of the sphere bundle). Let $E$ be a collar neighborhood of $\partial \FM_M(2)$ inside $\FM_M(2)$. (See \cite[Lemma VI.1.6]{Shiota} for the existence of a semi-algebraic (even Nash) collar.) Let us extend the form $\psi$ to $E$ by pulling it back along the projection $E \to \partial \FM_M(2)$. 
We can consider a semi-algebraic cutoff function $\rho \colon \FM_M(2) \to \mathbb R$ such that $\rho$ is constant equal to zero outside of $E$ and is constant equal to $1$ in some open set $U$ such that $\partial \FM_M(2) \subset U \subset E$. We can therefore consider the well defined form $\rho \psi \in \Omega_{triv}^{D-1}( \FM_M(2))$.

Since $d(\rho\psi)\big|_{\partial \FM_M(2)}= d\psi$ is basic, the form $d(\rho \psi)\in \Omega_{triv}^{D}( \FM_M(2))$ induces a form in $\Omega_{triv}^D(M\times M)$, still denoted by $d(\rho \psi)$. This form is clearly closed, but not necessarily exact, as $\rho \psi$ itself might not extend to the boundary. 

Let $\omega \in H^\bullet(M\times M) \subset \Omega_{triv}(M\times M) $. Then, we have
\begin{align}
\int_{M\times M} \omega \ d(\rho \psi) 
= \int_{\FM_M(2)}  \omega\   d(\rho \psi)
= (-1)^D  \int_{\partial \FM_M(2)}  \omega \rho \psi = \int_{\Delta\cong M} \omega .
\end{align}

It follows that the cohomology class of $d(\rho \psi)$ is the Poincar\'e dual of the diagonal $\Delta\cong M \subset M\times M$. Therefore $p_1^* \wedge p_2^*(\Delta)$ and $d(\rho \psi)$ are cohomologous in $\Omega_{triv}^{D}(M\times M)$. It follows that there exists a form $\beta \in \Omega_{triv}^{D-1}(M\times M)$ such that $d\beta = p_1^* \wedge p_2^*(\Delta) - d(\rho \psi)$. We define the form $\phi_{12} \in \Omega_{triv}^{D-1}(\FM_M(2))$ to be $\pi^*\beta + \rho \psi$, where $\pi \colon \FM_M(2)\to M\times M$ is the projection. It is clear that $\phi_{12}$ satisfies property \textit{(i)} and since the restriction of $\pi^*\beta$ to the boundary is a basic form and properties \textit{(ii)
} is preserved.

To ensure \textit{(iv)} one can replace the $\phi_{12}$ constructed so far by
\[
\phi_{12} - \int_3 \phi_{13}p_{23}^*\Delta - \int_3 \phi_{23}p_{13}^*\Delta +
\int_{3,4} \phi_{34}(p_{13}^*\Delta )(p_{24}^*\Delta)
\]
where $p_{ij}$ is the forgetful map, forgetting all but points $i$ and $j$ from a configuration of points. We refer the reader to \cite{CM} where more details can be found. (The reference contains a construction of the propagator in the smooth setting, but the trick to ensure \textit{(iv)} is verbatim identical in our semi-algebraic setup.)

Finally, we can (anti)symmetrise $\phi_{12}$ to ensure it satisfies property \textit{(iii)}, while preserving the other properties.
\end{proof}

\begin{rem}\label{rem:constant angular form}
For $M$ parallelizable, we can (and will) require a stronger version of property \textit{(ii)}.
A parallelization is a choice of a trivialization $\partial \FM_M(2) \simeq M \times S^{D-1}$ and given such parallelization, in the proof of the previous Proposition we can take $\psi = \pi^*(\omega_{S^{D-1}})\in \Omega_{triv}^{D-1}(M \times S^{D-1})$, the pullback of the standard volume form of $S^{D-1}$ via the projection $\pi\colon M \times S^{D-1} \to S^{D-1}$. 
By construction of $\phi_{12}$ the restriction of $\phi_{12}$ to $\partial\FM_M(2)$ has the form
\begin{equation}
 \label{equ:psietadef}
\phi_{12}\mid_{\partial\FM_M(2)} = \psi + p^*\eta
\end{equation}
where $p:\partial\FM_M(2)\to M$ is the projection to the base and $\eta\in \Omega_{triv}(M)$ is some form on the base.
Note in particular that from the closedness of $\psi$ and condition \textit{(i)} above it follows that 
\begin{equation}
 \label{equ:deta}
d\eta = \Delta_M,
\end{equation}
where $\Delta_M\in\Omega_{triv}^{D}(M)$ denotes the pullback of $\Delta$ along the diagonal map (i.e. the wedge product of its components). 
\end{rem}

\section{The Cattaneo-Felder-Mnev graph complex and operad}\label{sec:graphs}

Let $n$, $N$ and $D$ be positive integers and let $V$ be an $N$-dimensional graded vector space with a non-degenerate  pairing of degree $-D$; $\langle\ ,\rangle \colon V\otimes V\to \mathbb R[-D]$.
We require that for all homogeneous $x,y\in V$ of degrees $k$ and $l$ the pairing satisfies the (anti-)symmetry condition $\langle x, y\rangle = (-1)^{kl+D} \langle y,x\rangle$. Moreover, we assume $V$ to be ``augmented" in the sense that we are given also a canonical decomposition $V = \mathbb R \oplus \overline{V}$.
One should keep in mind the example of the Poincar\'e pairing on the cohomology of a connected $N$-manifold.

Let $e_2,e_3\dots, e_N$ be a graded basis of $\overline{V}$ and for convenience of notation we denote $e_1=1\in \mathbb R$. We consider the free graded commutative algebra generated by symbols $s^{ij}$ of degree $D-1$, where $1\leq i, j \leq n$, $s^{ij}=(-1)^D s^{ji}$, and symbols $e_1^j,\dots, e_N^j$, $j=1,\dots n$ of the same degrees as the elements of the basis $e_1,\dots, e_N$. We define a  differential on it by the following rules:
\begin{align*}
d e_\alpha^j &= 0 \\
d s^{ij} &= \sum_{\alpha,\beta} g^{\alpha\beta} e_\alpha^i e_\beta^j
\end{align*}
where $g^{kl}$ is the inverse of the matrix describing the pairing on $V$. (So $\sum_{\alpha,\beta} g^{\alpha\beta} e_\alpha^i  e_\beta^j$ is the ``diagonal class''.)

 We define the dgca $\pdGra_V(n)$ as the quotient of this algebra by the sub-dgca generated by elements of the form $e^j_1 -1$.
Notice that there is a natural right action of the symmetric group $
\mathbb S_n$ on $\pdGra_V(n)$ by permuting the superscript indices (the $i$ and $j$ above) running from $1$ to $n$. 

\begin{rem}
All definitions are independent of the choice of graded basis of $V$ and can be given in a basis-free way.
\end{rem}

\begin{rem}
The notation $\pdGra_V(n)$ stands for ``pre-dual graphs" as one may represent elements of $\pdGra_V(n)$ as linear combinations of decorated directed graphs with $n$ vertices and an ordering of the edges. The decorations are elements of $V$ that may be attached to vertices, see Figure \ref{fig:pdGra}. Each such graph corresponds to monomial in $\pdGra_V(n)$, an edge between vertices $i$ and $j$ corresponds to one occurrence of $s^{ij}$ and a decoration by an 
element $e_\alpha \in V$ at vertex $j$ corresponds to one occurrence of $e_\alpha^j$. Directions of the edges and their ordering might be ignored, keeping in mind that then a graph is only well defined up to a $\pm 1$ pre-factor. 
Notice that while both tadpoles and double edges are allowed, for (anti-)symmetry reasons one has that $s^{ii}=0$ if $D$ is odd and that $s^{ij}s^{ij} =0$ if $D$ is even.
\end{rem}

\begin{figure}[h] 
 \begin{tikzpicture}[scale=1,every loop/.style={}]

\node (1) at (1.4,1) [ext] {$1$};
\node (2) at (2.2,0.6) [ext] {$2$};
\node (3) at (2.8,2) [ext] {$3$};
\node (4) at (0.2,1.4) [ext] {$4$};

\node (w) at (2,1.6) [ext] {$e_2$};
\node (v) at (3,1.4) [ext] {$e_4$};
\node (q) at (1.3,0.2) [ext] {$e_4$};

\draw  (2)--(3);

\draw (4) --  (1);
\path (3) edge[loop] (3);

\draw[dotted]  (3)--(w);
\draw[dotted]  (v)--(3);
\draw[dotted]  (q)--(1);


\node at (4.8,1) {$=\pm s^{41}e_4^1s^{23}s^{33}e_2^3e_4^3$};
\end{tikzpicture}
\caption{An example of a graph describing an element in $\pdGra_V(4)$.}\label{fig:pdGra}
\end{figure}
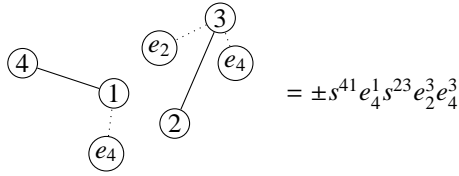

\subsection{Cooperadic comodule structure}

 \begin{defi}
Let $D$ be a positive integer. For $n\geq 2$, the space $\pdGra_D(n)$ is defined to be the free graded commutative algebra generated by symbols $s^{ij}$ in degree $D-1$, for $i\ne j$, quotiented by the relations $s^{ij} = (-1)^D s^{ji}.$
We set $\pdGra_D(0)=\pdGra_D(1)=\mathbb R$.
\end{defi}

As before, the spaces $\pdGra_D(n)$ can be seen as the span of undecorated graphs such that every edge has degree $D-1$.

\begin{prop}\label{prop:graphs cooperad}
The spaces $\pdGra_D(n)$ form a cooperad in dg commutative algebras. The cooperadic structure is given by removal (contraction) of subgraphs, i.e., for $\Gamma \in \pdGra(n)$, the component of $\Delta(\Gamma)$ in $\pdGra_D(k)\otimes \pdGra_D(i_1)\otimes \dots \otimes \pdGra_D(i_k)$ is 

\begin{equation}\label{eq:graphs coaction}
\sum \pm \Gamma'\otimes \Gamma_1\otimes \dots \otimes \Gamma_k,
\end{equation}
where the sum runs over all $k+1$-tuples $(\Gamma',\Gamma_1, \dots ,\Gamma_k)$ such that when each graph $\Gamma_i$ is inserted at the vertex $i$ of $\Gamma'$, there is a way of reconnecting the loose edges such that one obtains $\Gamma$.
\end{prop}

To obtain the appropriate signs one has to consider the full data of graphs with an ordering of oriented edges. In this situation the orientation of the edges of $\Gamma$ is preserved and one uses the symmetry relations on $\Gamma$ in such a way that for all $i=1,\dots,k$, the labels of the edges of the subgraph $\Gamma_{i}$ come before the labels of the edges of the subgraph $\Gamma_{i+1}$ and all of those come before the labels of the edges of the subgraphs $\Gamma'$. 
Notice that if one of the $i_j=0$, the cooperadic cocomposition is given by adding a disconnected vertex to $\Gamma'$ \cite[Section 2.2.1]{fressewillwacher}.
The cooperad axioms are a straightforward verification.

\begin{prop}\label{prop:GraV comodule}
The dg commutative algebras $\pdGra_V(n)$ for $n=1,2,\dots$ assemble to form a cooperadic right $\pdGra_D$ comodule $\pdGra_V$ in dg commutative algebras. 
\end{prop}
\begin{proof}

The cooperadic coactions are defined through formulas similar to \eqref{eq:graphs coaction} and proof of the associativity axiom is formally the same as the proof of the previous Proposition.

To show that the differential respects the comodule structure it suffices to check this on generators of the commutative algebra. This is clear for decorations $e_\alpha^i$ and for tadpoles $s^{ii}$. For edges connecting different vertices let us do the verification for $s^{12}\in \pdGra_V(2)$ for simplicity of notation. 
The only non-trivial commutative diagram to check if the following:
\[
\hspace{-1cm}
\begin{tikzcd}
\begin{tikzpicture}[scale=0.5]
\node (1) at (-3,0) [ext]{$1$};
\node (2) at (-1.5,0) [ext]{$2$};
\draw (1)--(2);
\end{tikzpicture} \arrow[mapsto]{r}{\Delta} \arrow[mapsto]{d}{d}
&
\underbrace{\begin{tikzpicture}[scale=0.5]
\node (1) at (-3,0) [ext]{$1$};
\node (2) at (-1.5,0) [ext]{$2$};
\node at (-0.75,0) {$\otimes$};
\node at (0,0) [ext]{$1$};
\node at (0.75,0) {$\otimes$};
\node at (1.5,0) [ext]{$1$};
\draw (1)--(2);
\end{tikzpicture}}_{\pdGra_V(2) \otimes \pdGra_D(1)\otimes \pdGra_D(1)}
\quad
+
\quad
\underbrace{\begin{tikzpicture}[scale=0.5]
\node at (-4.5,0) [ext]{$1$};
\node at (-3.75,0) {$\otimes$};
\node (1) at (-3,0) [ext]{$1$};
\node (2) at (-1.5,0) [ext]{$2$};
\draw (1)--(2);
\end{tikzpicture}}_{\pdGra_V(1) \otimes \pdGra_D (2)}
\quad 
\arrow[mapsto]{d}{d}
+
\quad \underbrace{\hspace{-.4cm}\begin{tikzpicture}[scale=0.5, every loop/.style={}]
\node (0) at (-4.5,0) [ext]{$1$};
\node at (-3.75,0) {$\otimes$};
\node (1) at (-3,0) [ext]{$1$};
\node (2) at (-1.5,0) [ext]{$2$};
 \path (0) edge[loop] (0);
\end{tikzpicture}}_{\pdGra_V(1) \otimes \pdGra_D (2)}
\\
 \begin{tikzpicture}[scale=1]
\node at (-1,0) {$\displaystyle\sum_{\alpha,\beta} g^{\alpha,\beta}$};
\node (1) at (0,0) [ext] {$1$};
\node (2) at (1.5,0) [ext] {$2$};

\node (w) at (0.6,0.4) [ext] {$e_\alpha$};
\node (v) at (0.9,-0.4) [ext] {$e_\beta$};

\draw[dotted]  (1)--(w);
\draw[dotted]  (v)--(2);
\end{tikzpicture} \arrow[mapsto]{r}{\Delta}
&
\begin{tikzpicture}[scale=0.5]
\node at (-1,0) {$\displaystyle\sum_{\alpha,\beta} g^{\alpha,\beta}$};
\node (1) at (0.5,0) [ext] {$1$};
\node (2) at (2.5,0) [ext] {$2$};
\node (w) at (1.5,0.6) [ext] {$e_\alpha$};
\node (v) at (1.5,-0.6) [ext] {$e_\beta$};
\draw[dotted]  (1)--(w);
\draw[dotted]  (v)--(2);

\node at (3.25,0) {$\otimes$};
\node at (4,0) [ext]{$1$};
\node at (4.75,0) {$\otimes$};
\node at (5.5,0) [ext]{$1$};

\end{tikzpicture}
\quad
 +0
\quad 
+
\begin{tikzpicture}[scale=0.5, every loop/.style={}]
\node at (-6,0) {$\displaystyle\sum_{\alpha,\beta} g^{\alpha,\beta}$};
\node (0) at (-4.5,0) [ext]{$1$};
\node at (-3.75,0) {$\otimes$};
\node (1) at (-3,0) [ext]{$1$};
\node (2) at (-1.5,0) [ext]{$2$};

\node (w) at (-5.1,1.5) [ext] {$e_\alpha$};
\node (v) at (-4,1.5) [ext] {$e_\beta$};

\draw[dotted]  (0)--(w);
\draw[dotted]  (v)--(0);
\end{tikzpicture}
%
%
%
\end{tikzcd}\]
where the vertical arrows correspond to the differential and the horizontal ones to the coaction.
\end{proof}

\subsection{Forms on (closed) manifolds}
Let $M$ be a closed smooth framed connected manifold of dimension $D$ and let $\FM_M$ be the Fulton-MacPherson compactification of the spaces of configurations of points of $M$ as described in Section \ref{sec:FM_M}. It is naturally an operadic right  module over the operad $\FM_D$, where the $i$-th composition of $c\in \FM_D(k)$ in a configuration  $\overline c \in \FM_M(n)$ corresponds to the insertion of the configuration $c$ at the $i$-th point of $\overline{c}$. The parallelization of the manifold ensures that this insertion can be made in a consistent way.

It follows that $\Omega_{PA}(\FM_M)$ is naturally equipped with a right cooperadic coaction of the cooperad  (in dg commutative algebras) $\Omega_{PA}(\FM_D)$ (mind Remark \ref{rem:almost cooperad} below). The coaction is obtained from the restriction of forms to boundary strata where multiple points collide.

There is a map of (``almost'') cooperads in dg commutative algebras
\begin{equation}\label{map:pdGra to forms}
\pdGra_D \to \Omega_{PA}(\FM_D),
\end{equation}

given by associating to every edge the angle form relative to the two incident vertices \cite{K1,LV}. More explicitly, one considers the standard volume form $\phi_{12} \in \Omega_{PA}^{D-1}(S^{D-1}) = \Omega_{PA}^{D-1}(\FM_D(2))$, which plays the role of the propagator. The forms $\phi_{ij}\in \Omega_{PA}^{D-1}(\FM_D(2))$ are defined by pulling back $\phi_{12}$ by the appropriate projection map. 
Finally, the map \eqref{map:pdGra to forms} above is obtained by extending the assignment $s^{ij} \mapsto \phi_{ij}$ to a map of dgcas.

\begin{rem}\label{rem:almost cooperad}
The functor $\Omega_{PA}$ is not comonoidal since the canonical map $\Omega_{PA}(A)\otimes \Omega_{PA}(B) \to \Omega_{PA}(A\times B)$ goes ``in the wrong direction", therefore $\Omega_{PA}(\FM_D)$ is not a cooperad. Nevertheless, by abuse of language throughout this paper we will refer to maps such as map \eqref{map:pdGra to forms} as maps of cooperads (or cooperadic modules) if they satisfy a compatibility relation such as commutativity of the following diagram:
$$  \begin{tikzcd}
   ^*\Gra_D(n) \ar{r} \arrow{dd} & \Omega_{PA}(\FM_D(n)) \arrow{d} \\
   & \Omega_{PA}(\FM_D(n-k+1)\times \FM_D(k)) \\
   ^*\Gra_D(n-k+1)\otimes {} ^*\Gra_D(k) \ar{r} & \Omega_{PA}(\FM_D(n-k+1))\otimes \Omega_{PA}(\FM_D(k)). \ar{u}
  \end{tikzcd}$$

\end{rem}

\vspace{1cm}

Since $M$ is connected, its cohomology $H^\bullet(M)$ has a canonical augmentation given by the constant functions on $M$ and since $M$ is closed, Poincar\'e duality gives us a pairing on $H^\bullet
(M)$ of degree $-D$. We define, for any manifold $M$:
\[
\pdGra_M \coloneqq \pdGra_{H^\bullet (M)}.
\]
Let us denote by $\iota\colon H^\bullet(M)\hookrightarrow \Omega_{triv}(M)$ the embedding from Convention \ref{convention}, that is, for every $\omega\in H^\bullet(M) $, $\iota(\omega)$ is a representative of the class $\omega$. Following Cattaneo and Mnev \cite{CM} we can define a map of dg commutative algebras (which a priori depends on various pieces of data)
\begin{equation}\label{eq:Gra to forms}
\pdGra_M \to \Omega_{triv}(\FM_M)\subset \Omega_{PA}(\FM_M)
\end{equation}
 as follows: The map sends the generator $s^{ij}$ for $i\ne j$ to $\phi_{ij}$, where $\phi_{ij}$ is the form constructed in the discussion preceding the proof of Proposition \ref{prop:angular form} with the additional assumption from Remark \ref{rem:constant angular form}. The map sends the decoration by $\omega\in H^\bullet(M)$ on the $j$-th vertex $\omega^j\in \pdGra_D$ to $p_j^*(\iota(\omega))$, where $p_j\colon \FM_M \to M$ is the map that remembers only the point labeled by $j$.
Finally the map sends $s^{jj}$ to $p_j^*\eta$, where $\eta$ is as in \eqref{equ:psietadef}.

\begin{lemma}\label{Lemma:map of comodules}
The map $\pdGra_M \to \Omega_{PA}(\FM_M)$ is a map of dg Hopf collections, compatible with the cooperadic comodule structures along the map $\pdGra_D\to \Omega_{PA}(\FM_D)$, in the sense of Remark \ref{rem:almost cooperad}. 
In other words there is a map of 2-colored dg Hopf collections
\[
\pdGra_M\ \aol\ \pdGra_D \to \Omega_{PA}(\FM_M)\ \aor\ \Omega_{PA}(\FM_D)
\]
compatible with the (2-colored) cooperadic cocompositions. 
\end{lemma}
\begin{proof}The compatibility with the differentials is clear for every generator of $\pdGra_M$ except possibly $s^{jj}$, for which one uses \eqref{equ:deta}.
By definition the map consists of morphisms of commutative algebras, therefore it is enough to check the compatibility of the cocompositions on generators. For elements $e_\alpha^j$ this is clear.
For the other generators we will sketch the verification for the case of $s^{12}\in \pdGra_M(2)$ for simplicity of notation.

The composition map in $\left(\FM_M, \FM_D\right)$ is done by insertion at the boundary stratum. Since the cocomposition map $\Omega_{PA}(\FM_M) \to  \Omega_{PA}(\FM_M)\circ \Omega_{PA}(\FM_D) $ is given by the pullback of the composition map we get, using \eqref{equ:psietadef}\footnote{Notice that on the second summand $\phi_{12}$ refers to the volume form of $S^{D-1}= \FM_D(2)$. We are using Remark \ref{rem:constant angular form} to ensure that this term is indeed of that form.}
$$\phi_{12} \in \Omega_{PA}(\FM_M(2))  \mapsto \underbrace{\phi_{12}\otimes 1 \otimes 1}_{ \Omega_{PA}(\FM_M(2)) \otimes \Omega_{PA}(\FM_D(1)) \otimes \Omega_{PA}(\FM_D(1))}  +\underbrace{1\otimes \phi_{12} + \eta \otimes 1.}_{\Omega_{PA}(\FM_M(1)) \otimes \Omega_{PA}(\FM_D(2))}.$$

On the other hand, the corresponding cocomposition $\pdGra_M \to\pdGra_M \circ \pdGra_D$ given by de-insertion sends $s^{12}$ to 

$$\underbrace{\begin{tikzpicture}[scale=0.5]
\node (1) at (-3,0) [ext]{$1$};
\node (2) at (-1.5,0) [ext]{$2$};
\node at (-0.75,0) {$\otimes$};
\node at (0,0) [ext]{$1$};
\node at (0.75,0) {$\otimes$};
\node at (1.5,0) [ext]{$1$};
\draw (1)--(2);
\end{tikzpicture}}_{\pdGra_M(2) \otimes \pdGra_D(1)\otimes \pdGra_D(1)}
\quad
+
\quad 
\underbrace{\begin{tikzpicture}[scale=0.5]
\node at (-4.5,0) [ext]{$1$};
\node at (-3.75,0) {$\otimes$};
\node (1) at (-3,0) [ext]{$1$};
\node (2) at (-1.5,0) [ext]{$2$};
\draw (1)--(2);
\end{tikzpicture}}_{\pdGra_M(1) \otimes \pdGra_D (2)}
\quad +
\quad 
\underbrace{\hspace{-.4cm}\begin{tikzpicture}[scale=0.5, every loop/.style={}]
\node (0) at (-4.5,0) [ext]{$1$};
\node at (-3.75,0) {$\otimes$};
\node (1) at (-3,0) [ext]{$1$};
\node (2) at (-1.5,0) [ext]{$2$};
 \path (0) edge[loop] (0);
\end{tikzpicture}}_{\pdGra_M(1) \otimes \pdGra_D (2)},$$ 
therefore the cocomposition is respected by the map.

\end{proof}

\section{Twisting $\Gra_M$ and the comodule $^*\Graphs_M$}\label{sec: twisting}

Let $\Gra_D$ and $\Gra_V$ be the duals of $\pdGra_D$ and $\pdGra_V$, respectively. $\Gra_V$ is an operadic right $\Gra_D$ module in dg cocommutative coalgebras.

Recall that there is a map of operads $\La^{D-1} \Lie\to \Gra_D$ given by mapping the generator $\mu \in \La^{D-1} \Lie(2)$ to the single edge graph in $\Gra_D(2)$ \cite{W1}.
This extends to a map from the canonical operadic right module $$\La^{D-1} \Lie\ \aor\ \La^{D-1} \Lie \to \Gra_M\ \aor\ \Gra_D$$ sending the generator $\mu$ to $s^{12} \in \Gra_M(2)$.
One can then apply the right module twisting procedure described in Appendix I of \cite{W1} to $ \Gra_M \aor\ \Gra_D$, thus obtaining the bimodule  $\Tw \Gra_M \aor \Tw\Gra_D$.

$\Tw \Gra_M$ can be described via a different kind of graphs. The space $\Tw \Gra_M(n)$ is spanned by graphs with $n$ vertices labeled from $1$ to $n$, called ``external" vertices and $k$ indistinguishable ``internal" vertices. Both types of vertices can be decorated by elements of $(H^\bullet (M))^*$ (with $\bullet\geq 1$, see Remark \ref{rem:quotiente_1} below), that can be identified with $H^{|D|-\bullet}(M)$ via the canonical pairing. The degree of the internal vertices is $D$, the degree of edges is $1-D$ and the degree of the decorations is the correspondent degree in $(H^\bullet (M))^*$, even if there is an identification with the cohomology.
 The differential in $\Tw \Gra_M$ can be split into $3$ pieces $d= \Delta + d_{ex} + d_{in} = \Delta + d_s$, where $\Delta$ is the differential coming from $\Gra_M$, that connects decorations by making an edge, $d_{ex}$ splits an internal vertex out of every external vertex and reconnecting incident edges in all possible ways and $d_{in}$ splits similarly an internal vertex out of every internal vertex:

\begin{figure}[h]

\begin{tikzpicture}

\node at (-2,0) {$\Delta$};
\node at (0,0) {
\begin{tikzpicture}[scale=0.8]
\node (w1) at (-0.11,-0.7) [ext]{$\omega$};
\node (w2) at (0.1,0.2) [ext]{$\nu$};
\node[fill=gray!20] (a) at (-1,-0.5) [ext] {$a$};

\node[fill=gray!20] (b) at (1,-0.25) [ext] {$b$};

\draw[dotted]  (a)--(w1);
\draw[dotted] (b)--(w2);

\draw (a) -- (-1.5,0.5);
\draw (a) -- (-2,0);
\draw (a) -- (-2,-1);
\draw (a) -- (-1.5,-1.5);
\draw (b) -- (1.5,0.8);
\draw (b) -- (1.8,0.3);
\draw (b) -- (2,-0.5);
\draw (b) -- (1.5,-1);
\end{tikzpicture}};

\node at (2.5,0) {$= \langle \omega, \nu \rangle$};

\node at (4.3,0) {
\begin{tikzpicture}[scale=0.8]
\node[fill=gray!20] (a) at (-1,-0.5) [ext] {$a$};

\node[fill=gray!20] (b) at (0,-0.25) [ext] {$b$};

\draw (a)--(b);

\draw (a) -- (-1.5,0.5);
\draw (a) -- (-2,0);
\draw (a) -- (-2,-1);
\draw (a) -- (-1.5,-1.5);
\draw (b) -- (0.5,0.8);
\draw (b) -- (0.8,0.3);
\draw (b) -- (1,-0.5);
\draw (b) -- (0.5,-1);
\end{tikzpicture}};

\node at (6.5,0) {$, \quad \quad d_s$};
\node at (7.6,0.15) {
\begin{tikzpicture}[scale=0.8]
\node[fill=gray!20] (a) at (-1,-0.5) [ext] {$a$};

\draw (a) -- (-1.3,0.4);
\draw (a) -- (-1.8,-0.1);
\draw (a) -- (-1.7,-1.1);

\draw (a) -- (-0.5,-1);
\draw (a) -- (-0.4,0.3);
\end{tikzpicture}};

\node at (8.5,0) {$=$};

\node at (9.4,0) {\begin{tikzpicture}[scale=1]
\node (ro) at (1,1.7) {};

\node[fill=gray!20] (e1) at (1,0.2)[ext] {$a$};
\node (black) at (1,0.8) [int] {$*$};

\node (e2) at (0,-1) {};
\node (e3) at (.75,-1) {};
\node (ed) at (1.3,-1){} ;
\node (en) at (2,-1) {};

\draw [dashed] (1,0.45) circle (0.60);
\draw (black) -- (e1);

\draw (0.2,-0.4) -- (0.6,-0);
\draw (1.8,-0.3) -- (1.4 ,0.05);

\draw (0.85,1) -- (0.6,1.5);
\draw (0.5,0.82) -- (0.2,1.1);
\draw (1.3,0.95) -- (1.8,1.5);

\end{tikzpicture}};

\end{tikzpicture}
\caption{Internal vertices are depicted in black. Gray vertices are either internal or external vertices.}
\end{figure}
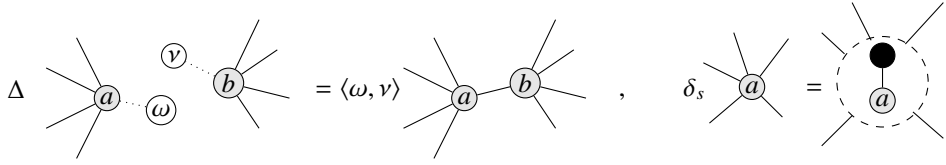

\begin{rem}\label{rem:quotiente_1}
Notice that due to $\pdGra_M$ being given by a quotient by $e_1^j-1$, if a certain vertex $v$ of $\Gamma\in \pdGra_M$ is decorated with the volume form on $M$, then we find as summands of $\Delta(\Gamma)$ all possibilities of connecting $v$ to every other vertex in $\Gamma$.
\end{rem}

The operad $\Tw \Gra_D$ is spanned by similar kinds of graphs, except that there are no decorations.
 We will therefore  also refer to the vertices of $\Tw \Gra_D$ as internal and external.  

The degrees of graphs in $\Tw \Gra_D$ are computed similarly, but the differential of $\Tw \Gra_D$ is different (since $\Gra_M$ is twisted as a $\Lie$-module whereas $\Gra_D$ is twisted as an operad under $\Lie$). Not only there is no $\Delta$ term, but also the splitting piece has an extra term subtracting all possible ways of adding a univalent internal vertex. 

 We are interested in a suboperad of $\Tw \Gra_D$, since $\Tw \Gra_D$ is in homologically ``too big". 
 The following result originates in \cite{K2,LV}.

\begin{defprop}[\cite{W1}] \label{defprop:Graphs_D}
The operad $\Tw \Gra_D$ has a suboperad that we call $\Graphs_D$ spanned by graphs $\Gamma$ such that:

\begin{itemize}
\item All internal vertices of $\Gamma$ are at least trivalent,
\item $\Gamma$ has no connected components consisting only of internal vertices.
\end{itemize}
Moreover there is a cooperadic quasi-isomorphism 
$$ ^*\Graphs_D \to \Omega_{PA}(\FM_D),$$
extending the map \eqref{map:pdGra to forms}.
\end{defprop}

This quasi-isomorphism is defined by integrating over all possible configurations of points corresponding to the internal vertices, a formula similar to the one from Lemma \ref{lemma: graphs to forms}.

We will from now on interpret  $\Tw \Gra_M$ as a right $\Graphs_D$-module.

Let $^*\! 	\Tw \Gra_M$ be the cooperadic right $\pdGraphs_D$ comodule that is (pre)dual to $\Tw \Gra_M$. 

The differential in ${}^*\!\Tw \Gra_M$ decomposes as $d = \delta_{cut} + \delta_{contr}$, where $\delta_{cut}$ is the piece originating from $\pdGra_M$ that splits edges into ``diagonal classes" and $\delta_{contr}$ contracts any edge adjacent to one or two internal vertices.

\begin{lemma}\label{lemma: graphs to forms}
For $M$ a closed compact framed connected manifold as above there is a natural map of right cooperadic comodules
\[
\omega_\bullet\colon {}^*\!\Tw \Gra_M \to \Omega_{PA}(\FM_M).
\]
extending the map $f\colon \pdGra_M\to \Omega_{PA}(\FM_M)$ from equation \eqref{eq:Gra to forms}.
\end{lemma}

\begin{proof}
Let $\Gamma$ be a graph in $^*\Gra_M(n+k)^{\mathbb S_k}\subset {}^*\!\Tw \Gra_M(n)$ i.e. $\Gamma$  has $n$ external and $k$ internal vertices. Let us consider $f(\Gamma)\in \Omega_{PA}(\FM_M(n+k))$, the image of $\Gamma$ under the map \eqref{eq:Gra to forms}. We define $\omega_\Gamma$ to be the integral of $f(\Gamma)$ over all configurations of the internal vertices. Concretely, if $\FM_M(n+k) \to\FM_M(n)$ denotes the map that forgets the last $k$ points, then  $\omega_\Gamma$ is given by the following fiber\footnote{Notice that here we make use of the fact that $f(\Gamma)$ is actually in $\Omega_{triv}(\FM_M(n+k))$.} integral $$\int\limits_{\FM_M(n+k) \to\FM_M(n)} f(\Gamma).$$  

The commutativity with the right cooperadic cocompositions is formally the same why ${^*}\! \Tw \Gra_D \to \Omega_{PA}(\FM_D)$ is a map of cooperads (see \cite[Section 9.5]{LV}) together with the fact that the propagator on $\FM_M$ on clusters of points is given by the corresponding propagator of $\FM_D$.
 It remains to check the compatibility of the differentials. 

Notice that ${}^*\!\Tw \Gra_M $ is a quasi-free dgca generated by internally connected graphs i.e. graphs that remain connected if we delete all external vertices. Since the map $\omega_\bullet$ is compatible with the products, it suffices to check the compatibility of the differentials on internally connected graphs. Let $\Gamma\in {}^*\!\Tw \Gra_M(n)$ be such a graph with $k$ internal vertices.

If we denote by $F$ the fiber of the map $\FM_M(n+k) \to\FM_M(n)$, we have, by Stokes Theorem

$$d\omega_\Gamma = \int\limits_{F} df(\Gamma)  \pm  \int\limits_{\partial F} f(\Gamma).$$

If we compute $d\Gamma = \delta_{cut}\Gamma + \delta_{contr}\Gamma$, we retrieve $$\omega_{\delta_{cut}\Gamma} = \int\limits_{F} f(d_{cut}\Gamma)= \int\limits_{F} df(\Gamma).$$ 

The boundary of the fiber decomposes into various pieces, namely $$\partial F = \bigcup_{n<i,j\leq n+k} \partial_{i,j}F \   \cup \  \bigcup_{\substack{a\leq n \\ n<i\leq n+k}} \partial_{a,i}F  \  \cup \  \partial_{\geq 3} F,$$ where $\partial_{i,j}F$ is the boundary piece where points $i$ and $j$ (corresponding to internal vertices) collided, $\partial_{a,i}F$ is the boundary piece where point $i$ (corresponding to an internal vertex) collided with point $a$ (corresponding to an external vertex) and $\partial_{\geq 3} F$ is the boundary piece in which at least $3$ points corresponding to internal vertices collided.

If in $\Gamma$ points $i$ and $j$ are not connected by an edge, then  $\int_{\partial_{i,j} F} f(\Gamma)=0$.
To see this, note that this integral has the form $\int_{\partial_{i,j} F} f(\Gamma) = \int_i \left. f(\Gamma)\right|_{i=j} \int_{S^{D-1}} 1 = 0$. Here the integral vanishes by degree reasons since there is no top degree component of the form on the factor $S^{D-1}$.
Here we used that the tangent bundle is trivialized. However, the same argument goes through without using this feature by using trivializations of the tangent bundle.

 If points $i$ and $j$ are connected by an edge, then by property \textit{(ii)} of Proposition \ref{prop:angular form} we have  $\int_{\partial_{i,j} F} f(\Gamma) = \omega_{\Gamma/e}$, where  $\Gamma/e$ is the graph $\Gamma$ with edge $e$ contracted. An analogous argument  for the boundary pieces $\partial_{a,i}F$ allows us to conclude that $\omega_{d\Gamma} = d\omega_{\Gamma}  \pm  \int_{\partial_{\geq 3} F} f(\Gamma)$.

The vanishing of $\int_{\partial_{\geq 3} F} f(\Gamma)$ results from Kontsevich's vanishing Lemmas. Concretely, suppose there are $3\leq l \leq k$ points colliding together. By integrating over the $l$ points first we obtain an integral of the form $\int_{\FM_D(l)}\nu$, where $\nu$ is a product of $\phi_{i,j}$. If the dimension $D$ is at least $3$, this integral vanishes as in \cite[Lemma 2.2]{K3}, using property \textit{(iii)} of Proposition \ref{prop:angular form}.

To factor the integral we used the trivialization of the tangent bundle in this step. For later use we shall however remark that this is not necessary. More precisely, let the full subgraph on the ``colliding'' vertices be $\gamma$. Then by the same argument as in the proof of \cite[Lemma 2.2]{K3}, using property \textit{(iii)} of Proposition \ref{prop:angular form}, we may assume that all vertices of $\gamma$ have valence $\geq 3$. But then the inner integral describes a form of degree $\geq \frac 3 2 l(D-1)-lD + D+1 = \frac 12 l(D-3)+D+1>D$ on $M$, and $M$ is of dimension $D$. Hence this integral must vanish. 

If $D=2$, because of property \textit{(ii)} of Proposition \ref{prop:angular form} we can use the Kontsevich vanishing lemma from \cite[Section 6.6]{K1} to ensure the vanishing of the integral.

\end{proof}

\subsection{The full Graph Complex and $\Graphs_M$}
The map constructed in Lemma \ref{lemma: graphs to forms} is not (in general) a quasi-isomorphism and the fundamental obstruction is the existence of graphs containing connected components of only internal vertices in $^*\!\Tw \Gra_M$. The desired complex $^*\Graphs_M$ will be a quotient of $^*\!\Tw \Gra_M$ through which the map $\omega_\bullet$ factors. A formal construction can be done making use of the full graph complex that we define as follows.

\begin{defi}
The full graph complex of $M$, $^*\!\fGC_{H^\bullet(M)}$ is defined to be the complex $^*\!\Tw \Gra_M(0)$ consisting of graphs with no external vertices. 
This vector space forms a differential graded commutative $\mathbb R$-algebra with product defined to be the disjoint union of graphs.
We reserve the symbol $\fGC_{H^\bullet(M)}=(^*\!\fGC_{H^\bullet(M)})^*$ for the dual complex and the symbol $\GC_{H^\bullet(M)}\subset \fGC_{H^\bullet(M)}$ for the subcomplex of connected graphs.
\end{defi}


The vector space $^*\!\Tw \Gra_M$ can be naturally regarded as a left module over the algebra $^*\!\fGC_{H^\bullet(M)}$, where the action is given by taking the disjoint union of graphs. 
Furthermore, we define the partition function 
\begin{equation}\label{equ:partition ZM}
Z_M \colon  ^*\!\fGC_{H^\bullet(M)} \to \R
\end{equation}
to be the map of dg commutative algebras obtained by restriction of the map $\omega_\bullet$ from Lemma \ref{lemma: graphs to forms}.

There is a commutative diagram of dg commutative algebras and modules

$$
\begin{tikzcd}[column sep=0.5em]				
 ^*\!\fGC_{H^\bullet(M)} \arrow{d}{Z_M} & \aol & ^*\!\Tw \Gra_M  \arrow{d}{\omega_\bullet} \\
\mathbb R &  \aol  &  \Omega_{PA} (\FM_M).     \\
\end{tikzcd}
$$

\begin{defi}
The right $^*\Graphs_D$ cooperadic comodule $^*\Graphs_M$ is defined by
$$^*\Graphs_M = \mathbb R \otimes_{Z_M} {}^*\Tw \Gra_M.$$
\end{defi}

\begin{rem}
We pick as representatives for a basis of  $^*\Graphs_M$ the set of graphs that contain no connected components without external vertices. With this convention it still makes sense to talk about the total number of vertices of a graph in $^*\Graphs_M$. 

Notice that as a consequence, part of the differential of $^*\Graphs_M$ might reduce the number of vertices by more than $1$ by ``cutting away" a part of the graph that contains only internal vertices, which did not happen with $^*\!\Tw \Gra_M$.
\end{rem}

\begin{cor}\label{cor: graphs to forms}
 The map $^*\!\Tw \Gra_M \to \Omega_{PA}(\FM_M)$ defined in Lemma \ref{lemma: graphs to forms} induces a map of cooperadic comodules $\pdGraphs_M \to \Omega_{PA}(\FM_M)$, still denoted by $\omega_\bullet$.
\end{cor}

\begin{rem}\label{rem:maurercartan}
One may also endow $\fGC_{H^\bullet(M)}$ with the (free) product being given by union of graphs.
 The differential is not a derivation with respect to this product, but it is a coderivation and it splits into a first order and a second order part, say $\delta_1+\delta_2$. Concretely, the second order part $\delta_2$ replaces a pair of $H^\bullet(M)$-decorations in different connected components by an edge, while the first order piece splits vertices and joins decorations in the same connected component.
 By Koszul duality, the commutator of the product and the operator $\delta_2$ defines a Lie bracket of degree 1 on $\fGC_{H^\bullet(M)}$, which reduces to a Lie bracket on the connected piece $\GC_{H^\bullet(M)}$.

Now the partition function $Z_M \in \fGC_{H^\bullet(M)}$ is a map from the free graded commutative algebra $^*\!\fGC_{H^\bullet(M)}$ and hence completely characterized by the restriction to the generators, i.e., to the connected graphs, say $z_M\in \GC_{H^\bullet(M)}$.
The closedness of $Z_M$ then translates to the statement that the connected part $z_M$ satisfies the Maurer-Cartan equation. See Section \ref{sec:GC_M} for details.
\end{rem}


To summarize, we constructed a cooperadic right Hopf comodule $^*\Graphs_M$. As a vector space, $^*\Graphs_M(n)$ is spanned by graphs with $n$ labeled external vertices and an unspecified number of indistinguishable internal vertices that can be decorated by (possibly multiple) cohomology classes of degree at least $1$, under the condition that there are no connected components without external vertices. 
\[
\begin{tikzpicture}[scale=1.4]
\node[ext] (v1) at (0,0) {$\scriptstyle 1$};
\node[ext] (v2) at (.5,0) {$\scriptstyle 2$};
\node[ext] (v3) at (1,0) {$\scriptstyle 3$};
\node[ext] (v4) at (1.5,0) {$\scriptstyle 4$};
\node[int] (w1) at (.25,.5) {};
\node[int] (w2) at (1,1) {};
\node[int] (w3) at (1,.5) {};
\node[ext] (i1) at (.7,1.3) {$\scriptstyle \omega_1$};
\node[ext] (i2) at (1.3,1.3) {$\scriptstyle \omega_1$};
\node[ext] (i3) at (-.4,.5) {$\scriptstyle \omega_2$};
\node[ext] (i4) at (1.9,.4) {$\scriptstyle \omega_3$};
\draw (v1) edge (v2) edge (w1) (w1) edge (w2) edge (w3) edge (v2) (v3) edge (w3) (v4) edge (w3) edge (w2) (w2) edge (w3);
\draw[dotted] (v1) edge (i3) (w2) edge (i2) edge (i1) (v4) edge (i4);

\node at (3,.5) {$\in {}^*\Graphs_M(4),$};
\end{tikzpicture}
\]

There is a graded commutative algebra structure given by superposition of external vertices
$$\begin{tikzpicture}

\node at (0,0) {\begin{tikzpicture}[scale=1.4]
	\node (a) at (0,0) [ext] {$1$};
	\node (b) at (1,0) [ext]{$2$};
	
	\node (int) at (0.5,0) [int] {.};
	
	\node[ext] (i1) at (.5,-0.4) {$\omega$};
	
	\draw (int) edge (a) edge (b);
	\draw[dotted] (i1) edge (int);
	\end{tikzpicture}};
\node at (1.5,0) {$\cdot$};
\node at (3,0){\begin{tikzpicture}[scale=1.4]
	\node (c) at (0,0) [ext] {$1$};
	\node (d) at (1,0) [ext]{$2$};
	
	\node (int2) at (0.5,0.5) [int] {.};
	
	\node[ext] (i2) at (1,0.4) {$\nu$};
	
	\draw (int2) edge (c) edge (d);
	\draw[dotted] (i2) edge (d);
	\end{tikzpicture}}; 
\node at (4.5,0) {$=$};
\node at (6,0){\begin{tikzpicture}[scale=1.4]
	\node (a) at (0,0) [ext] {$1$};
	\node (b) at (1,0) [ext]{$2$};
	
	\node (int) at (0.5,0) [int] {.};
	
	\node[ext] (i1) at (.5,-0.4) {$\omega$};
	
	\draw (int) edge (a) edge (b);
	\draw[dotted] (i1) edge (int);

	\node (int2) at (0.5,0.5) [int] {.};
	
	\node[ext] (i2) at (1,0.4) {$\nu$};
	
	\draw (int2) edge (a) edge (b);
	\draw[dotted] (i2) edge (b);
	
	\end{tikzpicture}};
\end{tikzpicture}.
$$

The differential $\delta$ splits as $\delta=\delta_{contr} + \delta_{cut}$, where $\delta_{contr}$ contracts edges adjacent to at least one internal vertex and $\delta_{cut}$ splits any edge into two decorations given by the diagonal class of $M$. Due to the constraint of not allowing connected components without external vertices, $\delta_{cut} = \Delta^* + \delta_{Z_M}$ splits again into two pieces, $\Delta^*$ which does not create \textit{forbidden} graphs and $\delta_{Z_M}$ that when creating such forbidden connected components transforms them into a scalar as prescribed by the partition function $Z_M$.

\begin{center}
\begin{tikzpicture}

\node at (-10.5,0) {$ \delta_{contr}$};

\node at (-9.4,0) {\begin{tikzpicture}[scale=1]
	\node (ro) at (1,1.7) {};
	
	\node[fill=gray!20] (e1) at (1,0.2)[ext] {$a$};
	\node (black) at (1,0.8) [int] {.};
	
	\node (e2) at (0,-1) {};
	\node (e3) at (.75,-1) {};
	\node (ed) at (1.3,-1){} ;
	\node (en) at (2,-1) {};
	\node (dec) at (0.8,1.5) [ext] {$\omega$};

	\draw (black) -- (e1);

	\draw (0.4,-0.3) -- (e1);
	\draw (1.5,-0.2) -- (e1);
	
	\draw (0.4,1.1) -- (black);
	\draw (1.4,1.35) -- (black);
	\draw[dotted] (dec)--(black);	
	\end{tikzpicture}};

\node at (-8.5,0) {$=$};

\node at (-7.6,0.15) {
	\begin{tikzpicture}[scale=0.8]
	\node[fill=gray!20] (a) at (-1,-0.5) [ext] {$a$};
	\node (dec) at  (-1.3,0.4) [ext] {$\omega$};

	\draw[dotted] (a) -- (dec);
	\draw (a) -- (-1.8,-0.1);
	\draw (a) -- (-1.7,-1.1);
	
	\draw (a) -- (-0.5,-1);
	\draw (a) -- (-0.4,0.3);
	\end{tikzpicture}};

\node at (-6,0) {$,\quad \quad \Delta^*$};

\node at (-4,0) {
	\begin{tikzpicture}[scale=0.8]
	\node[fill=gray!20] (a) at (-1,-0.5) [ext] {$a$};
	
	\node[fill=gray!20] (b) at (0,-0.25) [ext] {$b$};
	
	\draw (a)--(b);
	
	\draw (a) -- (-1.5,0.5);
	\draw (a) -- (-2,0);
	\draw (a) -- (-2,-1);
	\draw (a) -- (-1.5,-1.5);
	\draw (b) -- (0.5,0.8);
	\draw (b) -- (0.8,0.3);
	\draw (b) -- (1,-0.5);
	\draw (b) -- (0.5,-1);
	\end{tikzpicture}};

\node at (-1.8,0){$= \displaystyle\sum_{e_i \text{ basis of } H^\bullet(M)}$};

\node at (0.8,0) {
	\begin{tikzpicture}[scale=0.8]
	\node (w1) at (-0.11,-0.7) [ext]{$e_i$};
	\node (w2) at (0.1,0.2) [ext]{$\scriptstyle e_i^*$};
	\node[fill=gray!20] (a) at (-1,-0.5) [ext] {$a$};
	
	\node[fill=gray!20] (b) at (1,-0.25) [ext] {$b$};
	
	\draw[dotted]  (a)--(w1);
	\draw[dotted] (b)--(w2);
	
	\draw (a) -- (-1.5,0.5);
	\draw (a) -- (-2,0);
	\draw (a) -- (-2,-1);
	\draw (a) -- (-1.5,-1.5);
	\draw (b) -- (1.5,0.8);
	\draw (b) -- (1.8,0.3);
	\draw (b) -- (2,-0.5);
	\draw (b) -- (1.5,-1);
	\end{tikzpicture}};

\end{tikzpicture}

\begin{tikzpicture}

\node at (-1.7,0) {$\delta_{Z_M}$};
\node at (0,0) {
	\begin{tikzpicture}[scale=0.8]
	\node[fill=gray!20] (a) at (-1,-0.5) [ext] {\phantom{$a$}};
	
	\node (b) at (0,-0.25) [int] {$.$};
	
	\node (c) at (1,0) [int] {.};
	
	\node (dec b) at (-0.4,0.3) [ext] {$\omega$};
	
	\node (dec c) at (1,-0.7) [ext] {$\nu$};
	
	\draw (b)--(a)--(c);
	
	\draw [dotted] (b)--(dec b) (c)--(dec c);
	
	\draw (a) -- (-1.5,0.5);
	\draw (a) -- (-2,0);
	\draw (a) -- (-2,-1);
	\draw (a) -- (-1.5,-1.5);
	\end{tikzpicture}};
\node at (3,0) {$= \displaystyle\sum_{e_i \text{ basis of } H^\bullet(M)} Z_M($};

\node at (5.4,0){\begin{tikzpicture}[scale=0.8]
	\node (a) at (-0.7,-0.6) [ext] {$e_i$};
	
	\node (b) at (0,-0.25) [int] {$.$};
	
	\node (c) at (1,0) [int] {.};
	
	\node (dec b) at (-0.4,0.3) [ext] {$\omega$};
	
	\node (dec c) at (1,-0.7) [ext] {$\nu$};
	
	\draw (b)--(c);
	
	\draw [dotted] (a)--(b) (b)--(dec b) (c)--(dec c);
	
	\node at (1.5,0) {$)$};
	\end{tikzpicture}
};

\node at (7.5,0) {
\begin{tikzpicture}[scale=0.8]
	\node[fill=gray!20] (a) at (-1,-0.5) [ext] {\phantom{$a$}};
	
	\node (b) at (0,-0.25) [ext] {$\scriptstyle e_i^*$};

	\draw [dotted](b)--(a);

	\draw (a) -- (-1.5,0.5);
	\draw (a) -- (-2,0);
	\draw (a) -- (-2,-1);
	\draw (a) -- (-1.5,-1.5);
	\end{tikzpicture}};

\node at (9,0) {$+ \dots$};

\end{tikzpicture}
\end{center}
The cooperadic right comodule structure is obtained by collapsing subgraphs containing at least one external vertex into a single external vertex. 

\subsection{Historic Remark}
The above graph complexes can be seen as a version of the non-vacuum Feynman diagrams appearing in the perturbative expansion of topological field theories of AKSZ type, in the presence of zero modes. In this setting the field theories have been studied by Cattaneo-Felder \cite{CF} and Cattaneo-Mnev \cite{CM}, whose names we hence attach to the above complexes of diagrams, though the above construction of $\stG_M$ does not appear in these works directly. 
Furthermore, it has been pointed out to us by A. Goncharov that similar complexes have been known by experts before the works of the aforementioned authors.
Finally, in the local case the construction is due to M. Kontsevich \cite{K2}.

\section{Cohomology of the CFM (co)operad}\label{sec:main proof}

The following theorem relates the right $\Graphs_D$-module $\Graphs_M$ with the right $\FM_D$-module $\FM_M$.

\begin{thm}
\label{thm:GraphsM}
The map $\omega_\bullet \colon \pdGraphs_M \to \Omega_{PA}(\FM_M)$ established in Corollary \ref{cor: graphs to forms} is a quasi-isomorphism. Similarly,  the composition map $Chains(\FM_M)\to \Omega_{PA}(\FM_M)^* \stackrel{\omega_\bullet^*}{\to} \Graphs_M$ is a quasi-isomorphism of right modules.
\end{thm}

Note that there is in general no known explicit formula for the cohomology of the configuration spaces $\FM_M(n)$ on a manifold. However, two spectral sequences converging to the (co)homology are known, one by Cohen and Taylor \cite{CT} and one by Bendersky and Gitler \cite{BG}.
Both spectral sequences have been shown to be isomorphic (via Poincar\'e duality) from the $E^2$ term on by Felix and Thomas \cite{FT}.
The $E^2$ term is the cohomology of a relatively simple complex described below. It was shown by B. Totaro \cite{To} and I. Kriz \cite{Kr} that the spectral sequence abuts at the $E^2$ term for smooth projective varieties. However, it does not in general abut at the $E^2$ term, a counterexample was given in \cite{FT}.

The strategy to prove Theorem \ref{thm:GraphsM} will be as follows. We will compare the double complex $\BG$ giving rise to the Bendersky-Gitler spectral sequence (its definition will be recalled below) to $\pdGraphs_M$. There is a complex $\tBG$ quasi-isomorphic to $\BG$ that comes with a natural map $f: \tBG\to \Graphs_M$. Our goal is to show that $f$ is a quasi-isomorphism, and for that we set up another spectral sequence. The detailed proof is contained in section \ref{sec: map is a quasi-iso}.

\subsection{The Bendersky-Gitler spectral sequence}
Let us recall the definition of the Bendersky-Gitler spectral sequence. See also the exposition in \cite{FT}.

Recall that the configuration space of $n$ points in $M$ is $\Conf_n(M)\coloneqq M^n \setminus \Delta$, where $\Delta=\{(p_1,\dots, p_n)\mid \exists i\neq j: p_i = p_j \}$. By Poincar\'e - Lefschetz duality 
\[
H_{-d}(\Conf_n(M)) \cong H^{n\dim (M) -d}(M^n, \Delta).
\]
The relative cohomology $H^{\bullet}(M^n, \Delta)$ on the right is the cohomology of the complex
\[
H^\bullet(M^n) \to H^\bullet(\Delta).
\]
The left hand side is the cohomology of $\Omega_{PA}(M)^{\otimes n}$. The right hand side may be computed as the cohomology of the \v Cech-de Rham complex corresponding to any covering of $\Delta$.  
To obtain the Bendersky-Gitler double complex one takes the cover of the diagonal by the sets
\[
U_{i,j} = \{p_i=p_j\} \subset \Delta.
\]
The Bendersky-Gitler complex is the total complex of the double complex obtained using the natural quasi-isomorphism $\Omega_{PA}(M)^{\otimes n} \to \Omega_{PA}(M^{n})$ , i.e.,
\[
\BG(n) \coloneqq Total( \Omega_{PA}(M)^{\otimes n} \to \text{\v Cech-de-Rham}(\Delta)  ).
\]
By the statements above and a simple spectral sequence argument it follows that $H^\bullet(\BG(n))\cong H(M^n, \Delta)$.

For what we will say below it is important to describe $\BG(n)$ in a more concise way. Elements of $\BG(n)$ can be seen as linear combinations of decorated graphs on $n$ vertices, the decoration being one element of $\Omega_{PA}(M)$ for each connected component of the graph. The degree of such a graph is computed as
\[
(\text{degree}) =
\# (\text{edges})
+\# (\text{total degree of decorations})
-n\cdot \dim(M).
\]

The differential is composed of two parts, one of which comes from the de Rham differential and one of which comes from the \v Cech differential:
\[
d_{total} = d_{\text{dR}} + \delta.
\]

Concretely, $\delta$ adds an edge in all possible ways, and multiplies the decorations of the connected components the edge joins.

\begin{rem}
 The original construction of the Bendersky-Gitler spectral sequence uses the de Rham complex of $M$, but since there is only semi-algebraic data involved, namely intersections of sets $U_{i,j} \cong M^{n-1}$, we are allowed to replace differential forms by piecewise algebraic (PA) forms.
\end{rem}

\subsection{A general construction}
Recall that the monoidal product of symmetric sequences $\circ$ is given by 

$$(\op S \circ \op S')(n)  =\bigoplus_{k= k_1 + \dots +k_n} \op S(k) \otimes \op S'(k_1) \otimes \dots \otimes \op S'(k_n)\otimes \mathbb R [\text{Sh}(k_1 , \dots , k_n)],$$ 
where $\text{Sh}(k_1 , \dots , k_n)$ are the $k_1,\dots,k_n$ shuffles.
 Let $\op C$ be a cooperad, $\op M$ be a cooperadic right $\op C$-comodule with coaction $\Delta_{\op M}\colon \op M \to \op M \circ \op C$, and let $A$ be some dg commutative algebra, which can be seen as a symmetric sequence concentrated in arity $1$.  Then the spaces
\[
\op M(n) \otimes A^{\otimes n} = (\op M \circ A)(n) 
\] 
assemble into another cooperadic right $\op C$-comodule.

More formally, since $A$ is a dg commutative algebra, for every symmetric sequence $\op S$ there is a morphism 
\[
s\colon \op S \circ A \to A \circ \op S 
\]
given by the multiplication in $A$.

The coaction of $\op C$ on $\op M \circ A$ is given by the composition of the following maps:
\[
\op M \circ A
\stackrel{\Delta_{\op M}\circ \id_A}{\to}
(\op M \circ \op C) \circ A
\cong 
\op M \circ (\op C \circ A)
\stackrel{\id_{\op M} \circ s}{\to}
\op M \circ (A \circ \op C)
\cong 
(\op M \circ A) \circ \op C.
\]

It is a straightforward verification to check that the axioms for cooperadic comodules hold.

\subsection{The definition of $\tBG$}
Let $\op C$ be a coaugmented cooperad and $\op M$ be a right $\op C$ comodule. Applying the cobar construction to the cooperad $\op C$ we obtain an operad
$\Omega(\op C)$. Applying the cobar construction to the comodule $\op M$ we obtain a right $\Omega(\op C)$-module $\Omega_{\Omega(\op C)}(\op M)$, also denoted just by $\Omega(\op M)$. As a symmetric sequence $\Omega(\op M) = \op M \circ \Omega(\op C)$ and the differential splits as $d = d_1 + d_2 + d_3$, where $d_1$ comes from the differential in $\op M$, $d_2$ comes from the differential in $\Omega(\op C)$ and $d_3$ is induced by the comodule structure. Of course, if $A$ is a dg commutative algebra, then replacing $\op M$ by $\op M\circ A$ we obtain a right $\Omega(\op C)$-module $\Omega(\op M\circ A)$.
We can now define $\tBG$.
\[
\tBG := \Omega_{\La^{D-1}L_\infty}( s^{-D}\La^{D} \coCom \circ \Omega_{PA}(M))
\]
where on the right hand side we consider $s^{-D}\La^{D} \coCom$ as a right comodule over $\La^{D} \coCom$ and then we use the construction from the previous section that gives us a $\La^{D} \coCom$-right comodule structure on $s^{-D}\La^{D} \coCom \circ \Omega(\op M)$. Notice that the operadic cobar construction is given by $\Omega(\La^{D} \coCom) = \Omega((\La^{D-1}\Lie)^\vee) = \La^{D-1}L_\infty$.

Up to degrees, one can picture $\tBG$ as multiple (``commuting'') $L_\infty$ words, each labeled by a $PA$-form on $M$. Besides the de Rham and the $L_\infty$ differential, the cobar differential acts by merging two $L_\infty$ words while multiplying the associated forms.

\subsection{Some other general remarks and the definition of $\sBG$} 
Let $\op P$ be a Koszul operad, $\op P^\vee$ the Koszul dual cooperad and $\op P_\infty=\Omega(\op P^\vee)$ the minimal cofibrant model for $\op P$. There are bar and cobar construction functors between the categories 
of right $\op P$ modules and right $\op P^\vee$ comodules
\[
B_{\op P^\vee}  : Mod-\op P \leftrightarrow coMod-\op P^\vee : \Omega_{\op P}.
\]

Given a right $\op P^\vee$ comodule $\op M$ there are two ways to construct a right $\op P_\infty$ module:
\begin{enumerate}
\item Take the right $\op P_\infty$ module $\Omega_{\op P_\infty}(\op M)$.
\item Take $\Omega_{\op P}(\op M)$ and consider it as a right $\op P_\infty$ module via the morphism of operads $p\colon \op P_\infty\to \op P$.
\end{enumerate}

\begin{lemma}\label{lemma:quasi-iso of quasi-free}
Let $\op P$ be a Koszul operad with zero differential such that $\op P(0) = 0$ and $\op P(1) = \mathbb R$ and let $\op M$ be a right $\op P^\vee$ comodule.
There is a canonical (surjective) quasi-isomorphism 
\[
\pi \colon \Omega_{\op P_\infty}(\op M) \to \Omega_{\op P}(\op M).
\]
\end{lemma}
\begin{proof}
As symmetric sequences, $\Omega_{\op P_\infty}(\op M) = \op M \circ \op P_\infty$ and $\Omega_{\op P}(\op M) = \op M \circ \op P$. We define $\pi = \id_{\op M} \circ p$. It is clear that each piece of the differential commutes with $\pi$. The remaining claim that $\pi$ is a quasi-isomorphism follows from a spectral sequence argument.

Concretely, we consider a filtration $\mF^p \Omega_{\op P_\infty}(\op M)$ spanned by elements for which the sum of the degree in $\op M$ with the weight in $\op P_\infty$ (the amount of elements from $\op P^\vee$ used) does not exceed $p$. On the first page of the spectral sequence given by this filtration we recover $\Omega_{\op P}(\op M)$ and thus the result follows.
\end{proof}

Now let us give the definition of $\sBG$:
\[
\sBG = \Omega_{\La^{D-1} \Lie}( s^{-D}\La^{D} \coCom \circ \Omega_{PA}(M))
\]
where on the right we consider $\La^{D} \coCom = (\La^{D-1} \Lie)^\vee$ as a right comodule over itself and the algebra of differential forms $\Omega_{PA} (M)$. Then, by the Lemma above, we see that there is a canonical quasi-isomorphism 
\[
\tBG \to \sBG.
\]

Similar to $\tBG$, one can picture $\sBG$ as connected components of Lie words, each labeled by a $PA$-form on $M$. 
One can consider a basis of $\Lie(n)$ consisting of graphs on $n$ vertices, with $n-1$ edges, such that there no two edges  $(i,r)$ and $(j,r)$ with $r$ bigger than both $i$ and $j$.
Taking the degrees and differentials into account, we see that $\sBG(n)$ is precisely what in \cite{FT} is denoted by $\overline{E}(n,A)$, for $A=\Omega_{PA}(M)$.

Furthermore it was shown in \cite[Proposition 2.4]{FT} that there is a canonical quasi-isomorphism 
\[
\BG \to \sBG.
\]

In particular one obtains:
\begin{cor}The following symmetric sequences are isomorphic:
\[
H_\bullet(\Conf_\bullet(M)) \cong H(\BG) \cong H(\sBG) \cong H(\tBG).
\]
\end{cor}

\subsection{The map $\tBG\to \Graphs_M$}\label{sec:BG to graphs}
The goal of this subsection is to construct the map of right $\La^{D-1} L_\infty$ modules $\Phi: \tBG\to \Graphs_M$. 
Since $\tBG:= \Omega_{\La^{D-1}L_\infty}(s^{-D}\La^{D} \coCom \circ \Omega_{PA}(M))$ is quasi-free as right $\La^{D-1} L_\infty$ module, it suffices to define our map $\Phi$ on the module generators and verify that this map is compatible with the differential.
Note that $s^{-D}\La^{D}\coCom(n) = \R[nD]\mu_n$ is one dimensional, generated by the $n$-fold coproduct $\mu_n$. 
We define the map $\Phi$ on generators by setting, for $\alpha_1, \dots, \alpha_n\in \Omega_{PA}(M)$ and $\Gamma\in \stG_M$
\begin{equation}\label{equ:Phidef}
(\Phi(\mu_n\otimes \alpha_1 \otimes \cdots\otimes \alpha_n)) (\Gamma) := \int_{\FM_M(n)} (\pi_1^* \alpha_1) \cdots (\pi_n^* \alpha_n)\omega_\Gamma.
\end{equation}
Here $\pi_j : \FM_M(n)\to \FM_M(1) = M$ is the map that forgets the position of all points in the configuration except for the $j$-th point. Notice that the element $\mu_n\otimes \alpha_1 \otimes \cdots \otimes\alpha_n$ has degree $-nD + |\alpha_1| + \dots + |\alpha_n| = -(\dim (\FM_M(n)) - |\pi_1^*\alpha_1| - \dots - |\pi_n^*\alpha_n|)$, therefore $F$ preserves degrees.

A general element of $\tBG$ is a linear combination of elements obtained by acting with elements of the operad $\ell_j\in \La^{D-1}L_\infty$ on generators,
\[
x := (\mu_n\otimes \alpha_1 \otimes \cdots\otimes \alpha_n)
\circ (\ell_1,\dots,\ell_n).	
\]
For such elements $x$ we have that $\Phi(x)=\Phi(\mu_n\otimes \alpha_1 \otimes \cdots\otimes \alpha_n)\circ (\ell_1,\dots,\ell_n)$, using the right action of $\La^{D-1}L_\infty$ on $\Graphs_M$.
This latter action factors through the right action of $\Graphs_D$ on $\Graphs_M$ via the maps 
\[
	\La^{D-1}L_\infty \xrightarrow{f} \mathit{Chains}(\FM_D) 
	\xrightarrow{\omega^*} \Graphs_{D}. 
\]
Denoting the cooperadic coaction on $\Gamma\in \stG_M$ by $\sum \Gamma' \otimes \gamma_1\otimes \cdots \otimes \gamma_k$ , with $\gamma_j\in\stG_D$, this implies  that 
\begin{equation}\label{equ:Phi explicit}
\begin{aligned}
\Phi(x)(\Gamma)
&=(\Phi(\mu_n\otimes \alpha_1 \otimes \cdots\otimes \alpha_n)\circ (\ell_1,\dots,\ell_n) )(\Gamma)
\\&= 
\sum \pm \Phi(\mu_n\otimes \alpha_1 \otimes \cdots\otimes \alpha_n)(\Gamma') \cdot 
\prod_j \int_{f(\ell_j)}\omega_{\gamma_j}
\\&=
\sum \pm
\int_{\FM_M(n)} (\pi_1^* \alpha_1) \cdots (\pi_n^* \alpha_n)\omega_{\Gamma'}
\prod_j \int_{f(\ell_j)}\omega_{\gamma_j}
\\&=
\int_{\mathrm{Fund}(\FM_M(n))\circ (f(\ell_1),\cdots,f(\ell_n))} (\pi_{i_1}^* \alpha_1) \cdots (\pi_{i_n}^* \alpha_n)\omega_{\Gamma}\, .
\end{aligned}
\end{equation}
In the last line we are integrating over the fundamental chain of a boundary stratum of $\FM_M$ in which groups of points are infinitesimally close together. The indices $i_1,\dots,i_n$ shall be those of one (arbitrary) point in each such group.
Furthermore, we used the compatibility of the map $\omega$ with the operadic $\FM_D$-action on $\FM_M$.
Using the formula above we can show the following result.
\begin{lemma}\label{lem:Phi well defined}
The map $\Phi:\tBG\to \Graphs_M$ defined above is compatible with the differentials and is hence a map of right $\La^{D-1} L_\infty$ modules. 
It furthermore factorizes through the adjoint $\omega^*$ of the map $\omega:\stG_M\to \Omega_{PA}(\FM_M)$ of Corollary \ref{cor: graphs to forms} as 
\[
	\tBG\xrightarrow{F} \Omega_{PA}(\FM_M)^* \xrightarrow{\omega^*}\Graphs_M
\]
with 
\[
F((\mu_n\otimes \alpha_1 \otimes \cdots\otimes \alpha_n)\circ (\ell_1,\dots,\ell_n)  )(\omega)
=
\int_{\mathrm{Fund}(\FM_M(n))\circ (f(\ell_1),\cdots,f(\ell_n))} (\pi_{i_1}^* \alpha_1) \cdots (\pi_{i_n}^* \alpha_n)\omega\, .
\]
\end{lemma}
\begin{proof}
The factorization through $\omega^*$ is clear by \eqref{equ:Phi explicit}.

It remains to check that the differentials are preserved by $\Phi$.
Note that the diffential on $\tBG$ decomposes into three terms, $d=d_{\Omega_{PA}(M)}+d_{\La^{D-1}\L_\infty}+d_{\rm cobar}$, stemming from the internal differentials on $\Omega_{PA}(M)$ and $\La^{D-1}\L_\infty$, and the cobar construction respectively. Note that the second summand is zero on generators.

On the other hand we compute, applying Stokes' Theorem:
\begin{align*}
&(\Phi(\mu_n\otimes \alpha_1 \otimes \cdots \otimes \alpha_n)) (d \Gamma) \\
&= \int_{\FM_M(n)} (\pi_1^*\alpha_1) \cdots (\pi_n^*\alpha_n)\omega_{d\Gamma} \\
&= \int_{\FM_M(n)} (\pi_1^*\alpha_1) \cdots (\pi_n^*\alpha_n)d \omega_{\Gamma} \\
&= \sum_{j=1}^n \pm \int_{\FM_M(n)} (\pi_1^*\alpha_1) \cdots (\pi_j^* d\alpha_j) \cdots (\pi_n^*\alpha_n) \omega_\Gamma
\pm
\int_{\p \FM_M(n)} (\pi_1^*\alpha_1) \cdots (\pi_n^*\alpha_n)\omega_\Gamma.
\end{align*}
The two terms exactly reproduce the differential on $\tBG$. The first term corresponds to the part from the internal differential on $\Omega_{PA}(M)$. 
The second term (the boundary integral) produces the part of the differential from the cobar construction.
More precise it is the sum over codimension 1 boundary strata corresponding to some subset of the $n$ points colliding.
But each such term is, using the computation \eqref{equ:Phi explicit} again, identified with an action of a generator of $\La^{D-1}L_\infty$, so that all these terms together assemble to $\pm d_{\rm cobar}\Phi(\mu_n\otimes \alpha_1 \otimes \cdots \otimes \alpha_n)$.
\end{proof}

\subsection{The map $\tBG\to \Graphs_M$ is a quasi-isomorphism}\label{sec: map is a quasi-iso}
In this section we will show the following proposition.
\begin{prop}
\label{prop:tBGGraphsM}
The map $\Phi \colon \tBG\to \Graphs_M$ constructed above is a quasi-isomorphism.
\end{prop}


There is a filtration on $\Graphs_M$ by the number of connected components in graphs. Concretely, let $\mF^p \Graphs_M$ be the set of elements of $\Graphs_M$ which contain only graphs with $p$ or fewer connected components. There is a similar filtration on $\tBG$ coming from the arity of elements of the generating symmetric sequence $s^{-D}\La^{D} \coCom$. Concretely, elements of $\mF^p \tBG$ are those elements of $\tBG$ that can be built without using any generators $\mu_{p+1}, \mu_{p+2},\dots$ in $\La^{D} \coCom$.
The filtration is arity-wise bounded, since the number of connected components in arity $r$ is necessarily between $1$ and $r$.

\begin{lemma}
The map $\Phi$ from above is compatible with the filtration, i.e., 
\[
\Phi(\mF^p \tBG) \subset \mF^p \Graphs_M.
\]
\end{lemma} 
\begin{proof}
The result is clear for generators of $\tBG$, since graphs with $n$ vertices cannot have more than $n$ connected components. In general $\Phi$ is compatible with the filtration since is a  morphism of $\La^{D-1} L_\infty$ right modules and the right action of $\La^{D-1} L_\infty$ on $\Graphs_M$ is given by the insertion of connected graphs which cannot increase the number of connected components.
\end{proof}

It follows that $\Phi$ induces a morphism of the respective spectral sequences.
We will show the following lemma:
\begin{lemma}
\label{lem:HMCPhi}
The map $\Phi$ induces an isomorphism at the first pages of the associated spectral sequences.
\end{lemma}
The statement of the Lemma is equivalent to saying that the graded version of $\Phi$
\[
\gr \Phi \colon \gr\tBG\to \gr\Graphs_M
\]
is a quasi-isomorphism.

One can compute the cohomology of $\gr\tBG$ explicitly. 
\begin{lemma}
\label{lem:HgrtBG}
$H(\gr\tBG) =  (s^{-D}\La^{D} \coCom \circ H^\bullet(M)) \circ \La^{D-1} \Lie \eqqcolon \sBG_{H(M)}$.
\end{lemma}
\begin{proof}
The differential on $\gr \tBG$ is precisely the one induced by the de Rham differential and the differential on $\La^{D-1}L_\infty$. Therefore, by the K\"unneth formula, $H(\gr \tBG) = H(s^{-D}\La^{D} \coCom) \circ H (\Omega_{PA}(M))\circ H(\La^{D-1}L_\infty) = (s^{-D}\La^{D}\coCom \circ H^\bullet(M))\circ \La^{D-1}\Lie$.
\end{proof}

Having fixed the embedding $H^\bullet(M)\hookrightarrow \Omega_{PA}(M)$  and fixing any arity-wise right inverse (as cochain complexes) of the projection  $\La^{D-1}L_\infty\to \La^{D-1}\Lie$,  from now on we interpret the space $\sBG_{H(M)}$ (with zero differential) as a subcomplex of $\gr \tBG$. 
\begin{prop}
\label{prop:grGraphsM}
The map $\gr \Phi$ restricts to an injective map $\sBG_{H(M)} \to \gr \Graphs_M$ and the inclusion morphism $ \Phi(\sBG_{H(M)}) \hookrightarrow \gr \Graphs_M$ is a quasi-isomorphism.
\end{prop}
The proof is by an argument similar to the one used by P. Lambrechts and I. Volic in \cite[Lemma 8.3]{LV}.
If we believe Proposition \ref{prop:grGraphsM} for now, Lemma \ref{lem:HMCPhi} follows as a Corollary.

\begin{proof}[Proof of Proposition \ref{prop:tBGGraphsM}]
As a consequence of Lemma \ref{lem:HMCPhi}, the map $\Phi$ induces a quasi-isomorphism at the level of the associated graded with respect to a(n arity-wise) bounded filtration, and therefore is a quasi-isomorphism itself.
\end{proof}

\subsection{Proof of Proposition \ref{prop:grGraphsM}}
\label{sec:prop grGraphsM proof}


\begin{prop}
The vector spaces $\sBG_{H(M)}(n)$ satisfy the following recursion
\begin{equation}\label{eq:sBG recursion}
\sBG_{H(M)}(n) = \sBG_{H(M)} (n-1) \otimes H^\bullet(M) \oplus \sBG_{H(M)} (n-1)[D-1]^{\oplus n-1}.
\end{equation}
\end{prop}

\begin{proof}
We have $$\sBG_{H(M)}(n) = \bigoplus_{i_1+\dots+i_k =n} H^\bullet(M)^{\otimes k}[kD] \otimes \La^{D-1}\Lie(i_1)\otimes \dots \otimes  \La^{D-1}\Lie(i_k)\otimes \R[\text{Sh}(i_1,\dots,i_k)].$$ 

Let us take an element of $\sBG_{H(M)}(n)$ and consider two different cases. If the input labeled by $1$ corresponds to the unit $\mathsf 1\in \La^{D-1} \Lie(1)$ it is associated to an element of $H^\bullet(M)$ and by ignoring these we are left with a generic element of $\sBG_{H(M)}(n-1)$, thus giving us the first summand of \eqref{eq:sBG recursion}.

If, on the other hand, the vertex labeled by $1$ corresponds to some Lie word in $\La^{D-1} \Lie(i_j)$ with $j>1$, the only possibility is that it came from the insertion of the generator $\mu_2\in \La^{D-1} \Lie(2)$
in some other Lie word. Since there are $n-1$ such choices and $\mu_2$ has degree 
has degree $1-D$, we obtain the summand $\sBG_{H(M)} (n-1)[D-1]^{\oplus n-1}$.
\end{proof}

\begin{lemma}\label{lem:isomorphic image}
The map $\gr \Phi$ restricts to an isomorphism from $\sBG_{H(M)}(n)$ onto its image  $\Phi (\sBG_{H(M)}(n))\subset \gr \Graphs_M(n)$.
\end{lemma}

\begin{proof}
%
%

It is enough to show the injectivity of the map $\gr \Phi$ when restricted to $\sBG_{H(M)}(n)$. 

Recall that  $$\sBG_{H(M)}(n) = \bigoplus_{i_1+\dots+i_k =n} H^\bullet(M)^{\otimes k}[kD] \otimes \La^{D-1}\Lie(i_1)\otimes \dots \otimes  \La^{D-1}\Lie(i_k)\otimes \text{Sh}(i_1,\dots,i_k).$$ 

Let us start by considering the case in which the numbers $i_1,\dots, i_n$ are all equal to $1$. Let $\omega_1\otimes \dots \otimes \omega_n \in H^\bullet(M)^{\otimes n}[nD] \otimes \La^{D-1}\Lie(1)\otimes \dots \otimes  \La^{D-1}\Lie(1)$. The element $\Phi(\omega_1\otimes \dots \otimes \omega_n) \in \Graphs_M(n)$ is in principle a sum of many terms, but its projection into the subspace of $\Graphs_M(n)$ made only of graphs with no internal vertices, no more than one decoration per vertex, and precisely $n$ connected components is simply the graph
$$\pm
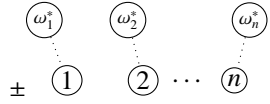
\begin{tikzpicture}
\node (1) at (0,0) [ext] {$1$};

\node (a) at (-0.3,0.8) [ext] {$\scriptstyle{\omega_1^*}$};

\node (2) at (1,0) [ext] {$2$};

\node (b) at (0.8,0.8) [ext] {$\scriptstyle{\omega_2^*}$};

\draw[dotted]  (a)--(1);
\draw[dotted]  (b)--(2);

\node at (1.6,0) {$\dots$};

\node (3) at (2.2,0) [ext] {$n$};

\node (c) at (2.4,0.8) [ext] {$\scriptstyle{\omega_n^*}$};

\draw[dotted]  (c)--(3);

\end{tikzpicture}
$$
where $\omega_i^*$ is dual to $\omega_i$ under the pairing on $H^\bullet(M)$. This implies in particular that $\Phi$ is injective when restricted to $ H^\bullet(M)^{\otimes n}[nD] \otimes \La^{D-1}\Lie(1)\otimes \dots \otimes  \La^{D-1}\Lie(1)$.

The same idea can be adapted for the case of arbitrary $i_j$. The image of the elements of $\sBG_{H(M)}$ might be very complicated, but to conclude injectivity it is enough to see that the components on a ``disconnected enough" subspace are different and by compatibility with the $L_\infty$ action these components are just given by insertion of graphs representing $L_\infty$ words.

 Let $p\subset  2^{\{1,\dots,n\}}$ denote a partition of the numbers $1,\dots,n$. To every such $p$ we can associate a subspace $V_p \subset \Graphs_M(n)$ spanned by graphs with no internal vertices and such that the vertices labeled by $a$ and $b$ are on the same connected component if and only if $a$ and $b$ are in the same element of the partition $p$.

 Every partition $p$ is determined the number of elements of the partition, which is a number $k\leq n$, the sizes of the partitions, $i_1,\dots,i_k$ such that $i_1+\dots +i_k=n$ and an element of $\Sh(i_1,\dots,i_k)$ specifying which numbers are included in each element of the partition. This data defines a subspace $W_p$ of $\sBG_{H(M)}(n)$ and the map $\Phi$ induces maps $\Phi_p \colon \overline{W_p} \to \overline{V_p}$, where $\overline{V_p}= \displaystyle\bigoplus_{p' \text{ coarser than }p } V_{p'}$ and similarly for $W_p$. It can shown by induction on the size of the partition $p$ that the maps $\Phi_p$ are injective for every partition $p$, so in particular for $p$ the discrete partition we obtain the injectivity of full map.

This follows from the fact that a linear map $f \colon A\oplus B \to V$ is injective  if its restriction to both $A$ and $B$ is injective and $f(A) \cap f(B) = 0$ and in our case these two conditions can be verified just by looking at the component of $V_p \subset \overline{V_p}$.
\end{proof}

\begin{cor}\label{cor:recursion}
The family of graded vector spaces $\Phi (\sBG_{H(M)}) \subset \gr \Graphs_M$ satisfies the following recursion:
$$\Phi (\sBG_{H(M)}(0))= \mathbb R,$$
$$\Phi (\sBG_{H(M)}(n)) = \Phi (\sBG_{H(M)} (n-1)) \otimes H^\bullet(M) \oplus \Phi (\sBG_{H(M)} (n-1))[D-1]^{\oplus n-1}.$$

\end{cor}

Proposition \ref{prop:grGraphsM} will follow from showing that the inclusion $\Phi (\sBG_{H(M)}) \hookrightarrow \gr \Graphs_M$ is a quasi-isomorphism and for this we will use some additional filtrations.

The differential on $\gr\Graphs_M$ splits into the following terms:
\[
\delta = \delta_s + \Delta + \Delta_1
\]
where $\delta_s$ is obtained by splitting vertices, $\Delta$ (the BV part of the differential) removes two decorations and creates an edge instead and $\Delta_1$ connects a connected component of (possibly decorated) internal vertices to the given graph.
Let us call the emv-degree (edges minus vertices) of  a graph the number 
\[
\#(edges)-\#(vertices).
\]
The differential can only increase or leave constant the emv degree. Hence we can put a filtration on $\gr\Graphs_M$ by emv degree. We will denote the associated graded by
\[
\gr'\gr\Graphs_M.
\]

The induced differential on the associated graded ignores the $\Delta$ part of the differential.

\begin{lemma}
\label{lem:grgr}
$H(\gr'\gr\Graphs_M)=\Phi(\sBG_{H(M)})$.
\end{lemma}

Since in $\gr'\gr\Graphs_M$ the $\Delta$ part of the differential is zero, all pieces of the differential increase the number of internal vertices by at least one.
To show this Lemma, we will put yet another filtration on $\gr'\gr\Graphs_M$ by  $\#(internal \ vertices) - degree$.  Let us call the associated graded 
\[
\gr''\gr'\gr\Graphs_M
\]

Notice that in $\gr''\gr'\gr\Graphs_M$ we have $\Delta=0$ and the only ``surviving" pieces of $\Delta_1$ replace any decoration by an internal vertex with the same decoration or connect a single internal vertex to another vertex of the graph. These pieces also appear in $\delta_s$ and it can be checked that they appear with opposite signs thus canceling out.

\begin{lemma}\label{lem:[LV] argument}
$H(\gr''\gr'\gr\Graphs_M)= \Phi(\sBG_{H(M)})$.
\end{lemma}
\begin{proof}
Let us write $V(n)=\gr''\gr'\gr\Graphs_M(n)$ for brevity. We will show that 
$H(V(n))\cong \Phi(\sBG_{H(M)}(n))$ by induction on $n$.
We can split
\[
V(n) = 
\begin{tikzpicture} [every edge/.style={-triangle 60, draw}, baseline=-.65ex]
\matrix (m) [matrix of math nodes, column sep=1em] {
V_0 & \oplus & V_1 & \oplus & V_{\geq 2} \\};
\draw (m-1-1) edge[loop above] (m-1-1);
\draw (m-1-3) edge[loop above] (m-1-3);
\draw (m-1-5) edge[loop above] (m-1-5)
              edge[bend right] (m-1-3);
\end{tikzpicture}
\]
according to the valence of the external vertex $1$ (where decorations are considered to increase the valence of the vertices). The arrows indicate how the differential maps the individual parts to each other.
The complex $V_0$ is isomorphic to $V(n-1)$ and we can invoke the induction hypothesis. For the remainder we consider a spectral sequence whose first differential is $V_{\geq 2} \to V_1$. Concretely, we consider $(\mathcal F_k)_{k\in \mathbb Z}$, a descending filtration $V(n)\supset \dots \supset \mathcal F_k\supset \mathcal F_{k+1} \supset \dots\supset 0 $, such that $\mathcal F_k$ is spanned by graphs of degree at least $k$ in which the vertex $1$ is not $1$-valent and by graphs of degree at least $k+1$ in which the vertex $1$ has valence $1$. The map $V_{\geq 2} \to V_1$ is injective and its cokernel is generated by graphs of one of the following types:
\begin{enumerate}
\item Vertex $1$ has a decoration and no incoming edges.
\item Vertex $1$ has no decoration and is connected to some other external vertex.
\end{enumerate}

In the first case we obtain a complex isomorphic to $V(n-1)$ for every choice of decoration, with a degree shift given by the decoration. In the second case, each choice of connecting external vertex yields a complex isomorphic to $V(n-1)$ with a degree shift given by the additional edge. This gives us the following expression of the first page of the spectral sequence:

$$E_1(V(n)) = H(\gr V(n)) = V_0 \oplus V(n-1) \otimes \overline{H^\bullet}(M) \oplus V(n-1) [D-1]^{\oplus n-1}$$
$$= V(n-1) \otimes H^\bullet(M) \oplus V(n-1) [D-1]^{\oplus n-1}.$$

Under this identification, on the this page of the spectral sequence we obtain precisely the differential of $V(n-1)$. Notice that $V_1\oplus V_{\geq 2}$ is a double complex concentrated on a double column and therefore the spectral sequence collapses at the second page $E_2$. From this observation we obtain the following recursion

$$H(V(n))= H(V(n-1)) \otimes H^\bullet(M) \oplus H(V(n-1)) [D-1]^{\oplus n-1}.$$
which is the same as the recursion for $\Phi(\sBG_{H(M)}(n))$, as show in Corollary \ref{cor:recursion}.  To see that the inclusion $\Phi(\sBG_{H(M)}(n)) \to V(n)$ induces a quasi-isomorphism on the second page of the spectral sequence, we start by noticing that the result holds trivially on the $1$-dimensional initial terms $\Phi(\sBG_{H(M)}(0))$ and $H(V((0))$ and therefore $\Phi(\sBG_{H(M)}(n))$ and $H(V((n))$ have the same dimension.

The second page of the inclusion map
$$ \Phi (\sBG_{H(M)} (n-1)) \otimes H^\bullet(M) \oplus \Phi (\sBG_{H(M)} (n-1))[D-1]^{\oplus n-1} \to  H(V(n-1)) \otimes H^\bullet(M) \oplus H(V(n-1)) [D-1]^{\oplus n-1}$$

can be written as 

\[ \left( \begin{array}{cc}
f_{11} & f_{12}\\
f_{21} & f_{22}
\end{array} \right),\]
where $f_{12}\colon \Phi (\sBG_{H(M)} (n-1))[D-1]^{\oplus n-1} \to H(V(n-1)) \otimes H^\bullet(M)$ is actually the $0$ map, since  $\Phi (\sBG_{H(M)} (n-1))[D-1]^{\oplus n-1}$ corresponds to the image of elements in $ H^\bullet(M)^{\otimes k}[kD] \otimes  \La^{D-1}\Lie(i_1)\otimes \dots \otimes  \La^{D-1}\Lie(i_k)$ with $i_1\geq 2$ and the vertex $1$ cannot be the only labeled vertex in its connected component. The maps $f_{11}$ and $f_{22}$ are isomorphisms by induction and therefore the second page of the inclusion map is an isomorphism, from where the result follows.
\end{proof}

\begin{proof}[Proof of Lemma \ref{lem:grgr}]
The $E^1$ term of the spectral sequence is a quotient complex, hence it abuts at that point.
\end{proof}

\begin{proof}[Proof of Theorem \ref{thm:GraphsM}]
We have shown that the composition $\tBG \stackrel{F}{\to} \Omega_{PA}(\FM_M)^* \stackrel{\omega^*}{\to} \Graphs_M$ is a quasi-isomorphism, but since the homology of $\Omega_{PA}(\FM_M)^*$ is also isomorphic to the other two homologies which are finite dimensional in each arity and degree, it follows that $F$ and $\omega^*$ are quasi-isomorphisms themselves.

Consequentially, the map $Chains(\FM_M) \to \Omega_{PA}(\FM_M)^* \to \Graphs_M$ is a composition of quasi-isomorphisms, therefore is a quasi-isomorphism as well.

This concludes the  proof of Theorem \ref{thm:GraphsM}.
\end{proof}

\begin{rem}
For the proof of Theorem \ref{thm:GraphsM} we consider the functor $\Omega_{PA}$ of semi-algebraic forms, but it could equally be used any contravariant functor $\Omega$ landing in dgca's satisfying the following properties:

\begin{itemize}
\item $\Omega$ is quasi-isomorphic to the Sullivan functor $A_{PL}$ of piecewise-linear de Rham forms.

\item $\Omega$ admits pushforwards of the forgetful maps $\FM_M(n) \to \FM_M(n-k)$ satisfying the usual properties of fiber integrals, in particular Stokes Theorem.

\item $\Omega$ is ``almost" comonoidal, as in Remark \ref{rem:almost cooperad}.

\end{itemize}

\end{rem}

\section{The non-parallelizable case}
\label{sec:non-framed}

Let $M$ be a closed oriented connected manifold. In this section we show that even in absence of the parallelizability hypothesis a slight variant of the collection of commutative algebra $\pdGraphs_M$ is still a model of $\FM_M$. 

In this respect it is not natural to consider graphs with tadpoles as the compatibility of the differential of the map from Lemma \ref{Lemma:map of comodules} depended on the vanishing of the Euler characteristic for those graphs. More precisely, the problem is that in the map of Lemma \ref{Lemma:map of comodules} a tadpole edge is sent to a form whose coboundary is the Euler class.

We define $\pdGra_M^{\cutaol}\subset \pdGra_M$ to be the dg Hopf sub-collection spanned by graphs without tadpoles.

Note that this subcollection is indeed closed under the product and differential.
It furthermore retains a $\La^{D-1}\Lie^*$-comodule structure from $\pdGra_M$, but not the full $\pdGra_D$ comodule structure, as the proof of Proposition \ref{prop:GraV comodule} fails in the absence of tadpoles.
Furthermore, the map \eqref{eq:Gra to forms} naturally restricts to a map of dg Hopf collections
\[
	\pdGra_M^{\cutaol}\to \Omega_{triv}(\FM_M)\subset \Omega_{PA}(\FM_M),
\]
that is well-defined even if $M$ has a non-trivial Euler class.
The twisting construction of section \ref{sec: twisting} and in particular the construction of the map $\omega$ of Corollary \ref{cor: graphs to forms} also naturally yields a map 
\begin{gather*}
 \omega \colon	\pdGraphs_M^{\cutaol} \to \Omega_{PA}(\FM_M)
 \\
 \Gamma\mapsto\omega_\Gamma,
\end{gather*}
where we denote by $\pdGraphs_M^{\cutaol}\subset \pdGraphs_M$ the sub-collection spanned by graphs without tadpoles.

To be clear, if $M$ has non-vanishing Euler class then the map $\omega$ of Corollary \ref{cor: graphs to forms} is not a priori not well-defined on $\Graphs_M$ because we would need to send a tadpole edge to a form whose coboundary is the Euler class. Furthermore, the partition function \eqref{equ:partition ZM} is only well-defined on the tadpole-free part $^*\!\fGC_{H^\bullet(M)}^\cutaol\subset ^*\!\fGC_{H^\bullet(M)}$.
Hence one does not even get a well-defined (square-zero) differential on the graded collection $\pdGraphs_M$ from the partition function, one only has this on the tadpole-free part $\pdGraphs_M^{\cutaol}$. 

In particular, we note that the differential on $\pdGraphs_M^{\cutaol}$ can indeed not produce tadpoles.
The only term in the differential that is able to produce a tadpole is the edge contraction in the presence of a multiple edge.
However, multiple edges are zero by symmetry reasons for even $D$ while tadpoles are not present by symmetry reasons for odd $D$, hence no problem arises.


Also, if $M$ is not parallelized, there is no consistent way of defining a right $\FM_D$ action on $\FM_M$. Nonetheless, disregarding the cooperadic coactions, the map  $\pdGraphs_M^{\cutaol} \to \Omega_{PA}(\FM_M)$ is well defined as a map of dgcas since the proof of Lemma \ref{lemma: graphs to forms} uses parallelizability condition only for the tadpoles and the coaction, see the remarks within that proof on using the trivialization of the tangent bundle.

Before proceeding, let us furthermore show that the exclusion of tadpoles has no effect on the homotopy type, provided $\pdGraphs_M$ is well-defined. 
(See \cite[Proposition 3.4]{W1} for simiar results and arguments.)

\begin{prop}
	Suppose that $M$ is parallelizable (or at least has vanishing Euler class), so that the dg Hopf collection $\pdGraphs_M$ is well-defined.
Then the inclusion $\pdGraphs_M^{\cutaol}\to \pdGraphs_M$ is a quasi-isomorphism of collections of dg commutative algebras.
\end{prop}

\begin{proof}[Proof sketch]
We consider a spectral sequence on $\pdGraphs_M$ whose associated graded has a differential contracting internal vertices with only an adjacent edge and a tadpole along the non-tadpole edge

$$
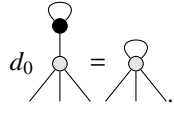
\begin{tikzpicture}[scale=0.5, every loop/.style={}]
\node at (-1,-1) {$d_0$};
\node (0) at (0,0) [int]{$.$};
\node[fill=gray!20] (1) at (0,-1) [int]{$\phantom{.}$};

 \path (0) edge[loop] (0);
 \draw (1)--(0);
 
 \draw (0,-2)--(1);
 \draw (0.8,-2)--(1);
 \draw (-0.8,-2)--(1);

 \node at (1,-1){$=$};

\node[fill=gray!20] (3) at (2,-1) [int]{$\phantom{.}$};

 \path (3) edge[loop] (3);
 
 \draw (2,-2)--(3);
 \draw (2.8,-2)--(3);
 \draw (1.2,-2)--(3); 
 
\end{tikzpicture} .$$

Such a spectral sequence can be obtained by filtering first by the number of tadpoles and then by $l+degree$, where $l$ is the sum of lengths of maximal connected subgraphs consisting of $2$-valent internal vertices and one internal vertex with just a tadpole at the end.

We can then set up a homotopy $h$ that splits out an internal vertex with a tadpole
$$
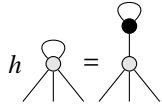
\begin{tikzpicture}[scale=0.5, every loop/.style={}]
\node at (-1,-1) {$h$};
\node[fill=gray!20] (1) at (0,-1) [int]{$\phantom{.}$};

 \path (1) edge[loop] (1);

 \draw (0,-2)--(1);
 \draw (0.8,-2)--(1);
 \draw (-0.8,-2)--(1);

 \node at (1,-1){$=$};
 
\node (4) at (2,0) [int]{$\phantom{.}$};

\node[fill=gray!20] (3) at (2,-1) [int]{$\phantom{.}$};

 \path (4) edge[loop] (4);
  \draw (3)--(4);
 \draw (2,-2)--(3);
 \draw (2.8,-2)--(3);
 \draw (1.2,-2)--(3); 
 
\end{tikzpicture} .$$

We have $d_0h+hd_0 = T\ \id$, where $T$ is the number of tadpoles, from where it follows that $H(\pdGraphs_M,d_0) = \pdGraphs_M^{\cutaol}$.\end{proof}

Finally, one has the following version of Theorem \ref{thm:GraphsM} for non-parallelizable $M$.

\begin{thm}\label{thm:Main non-parallelized}
Let $M$ be a closed oriented manifold.
The map $\omega_\bullet \colon \pdGraphs_M^{\cutaol} \to \Omega_{PA}(\FM_M)$ is a quasi-isomorphism of symmetric sequences of dg commutative algebras. Similarly,  the composition map $Chains(\FM_M)\to \Omega_{PA}(\FM_M)^* \stackrel{\omega_\bullet^*}{\to} \Graphs_M^{\cutaol}:=(\pdGraphs_M^{\cutaol})^*$ is a quasi-isomorphism.
\end{thm}

\begin{proof}
We follow the proof of Theorem \ref{thm:GraphsM}.
First we note that while in general one does not have a right $\FM_D$-module structure on $\FM_M$ if $M$ is not framed, the insertion of fundamental chains of $\FM_D$ at points in $\FM_M$ is independent of the framing so in fact it gives us a well defined operadic action $Chains(\FM_M) \circ\Lambda^{D-1} L_\infty \to Chains(\FM_M)$.
Similarly, as mentioned above $\Graphs_M^{\cutaol}$ inherits a right $\Lambda^{D-1} L_\infty$ module structure from the one on $\Gra_M^{\cutaol}:=(\pdGra_M^{\cutaol})^*$.
These structures suffice to define the map of right $\Lambda^{D-1} L_\infty$ modules
\[
\Phi \colon \tBG\to \Graphs_M^{\cutaol}
\]
as in section \ref{sec:BG to graphs} by formula \eqref{equ:Phidef} (respectively \eqref{equ:Phi explicit}).
Furthermore, Lemma \ref{lem:Phi well defined} does not make use of tadpoles and holds in this case as well.

Furthermore the remaining arguments of sections \ref{sec: map is a quasi-iso} and \ref{sec:prop grGraphsM proof} above leading to Theorem \ref{thm:GraphsM} are agnostic to the presence or absence of parallelizability of $M$ or tadpoles in graphs, and hence also show Theorem \ref{thm:Main non-parallelized}. 
\end{proof}

\section{A simplification of $^*\Graphs_M$ and relations to the literature}\label{sec:Lambrechts and Stanley}

\subsection{An alternative construction of $\Graphs_M$.}\label{sec:GC_M}
Recall that in Section \ref{sec: twisting} the space $\pdGraphs_M$ was constructed by identifying connected components without external vertices with real numbers via a ``partition function'', which is a map of commutative algebras $Z_M\colon ^*\!\fGC_{H^\bullet(M)} \to \mathbb R$.

In this subsection and the next we present an alternative construction of $\Graphs_M$ that will allow us to understand better the relevance of the partition function $Z_M$ in the homotopy type of $\Graphs_M$.

Notice that $^*\!\fGC_{H^\bullet(M)}$ is a quasi-free commutative algebra generated by its subspace of connected graphs $^*\!\GC_{H(M)}$.
 The differential $d$ on $^*\!\fGC_{H^\bullet(M)}$
defines then a $\La L_\infty$ coalgebra structure on $^*\!\GC_{H^\bullet(M)}$. In fact, since the differential can increase the number of connected components by at most one, this is in fact a strict Lie coalgebra structure. 
 
 The dual Lie algebra structure is denoted by $\GC_{H^\bullet(M)}= ({}^*\!\GC_{H^\bullet(M)})^*$ and is represented by infinite sums of graphs decorated by $H_\bullet(M)$ (or dually by $H^\bullet(M)$, via the Poincar\'e  pairing). The Lie bracket $[\Gamma,\Gamma']$ is given by summing over all possible ways of selecting a decoration in $\Gamma$ and another decoration in $\Gamma'$ and connecting them into an edge, with a factor given by their pairing. The differential acts by vertex splitting and joining decorations.
  
It follows that maps of dg commutative algebras $^*\!\fGC_{H^\bullet(M)}\to \mathbb R$ are identified with maps in the Lie algebra  satisfying the Maurer-Cartan equation.
$$\mathsf{MC}(\GC_{H^\bullet(M)}) = \Hom_{\text{dgca}}( {}^*\fGC_{H^\bullet(M)},\mathbb R).$$ We denote by $z_M\in \GC_{H^\bullet(M)}$ the Maurer-Cartan element corresponding to the partition function $Z_M$. If we consider the subrepresentation $S\subset \Tw \Gra_M$ given by graphs with no connected components consisting only of internal vertices, then $\Graphs_M$ is obtained by twisting $S$ by the Maurer-Cartan element $z_M$, as recalled in the following section.

In analogy we denote by $\GC_M\coloneqq \GC_{H^\bullet(M)}^{z_M}$ the Lie algebra obtained by twisting with the Maurer-Cartan element $z_M$.

For later use let us also split the Maurer-Cartan element
\[
 z_M = \underbrace{
\sum_{i,j=1}^D g^{ij}\,
\begin{tikzpicture}[baseline=-.65ex]
 \node[int] (v) at (0,0) {};
\node[ext] (w1) at (-.4,0) {$\scriptstyle e_i$};
\node[ext] (w2) at (.4,0) {$\scriptstyle e_j$};
\draw[dotted] (v) edge (w1) edge (w2);
\end{tikzpicture}
}_{=:z_0} \,+ z_M'
\]
into a part $z_0$ given by graphs with exactly one vertex and $2$ or $1$ decorations and a remainder $z_M':=z_M-z_0$.
Note in particular that $z_0$ is determined solely by the non-degenerate pairing on $H(M)$.
The element $z_0$ is itself a Maurer-Cartan element, and below we will consider the twisted dg Lie algebra 
\[
 \GC_{H(M)}':= \GC_{H(M)}^{z_0},
\]
and consider $z_M'$ as a Maurer-Cartan element in $\GC_{H(M)}'$.

\newcommand{\Exp}{\mathrm{Exp}}
\newcommand{\grt}{\alg {grt}}

\subsection{Twisting of modules}\label{sec:twisting of modules}

While the differential of $\Graphs_M$ can be very non-explicit, expressing it as twist by a Maurer-Cartan element opens the door to simplifications of the model, as long as we have some control over the gauge equivalence class of the Maurer-Cartan element.

Indeed, let us pause for a moment to consider the following general situation.
Suppose $\alg g$ is a dg Lie algebra, acting on $M$, where $M$ can be just a dg vector space, or a (co)operad or a (co)operadic (co)module, or a pair of a (co)operad and a (co)operadic (co)module.
In any case we require the $\alg g$ action to respect the given algebraic structure, in the sense that the action is by (co)derivations.

Suppose now that $m\in \alg g$ is a Maurer-Cartan element, i.e., $dm+\frac 1 2[m,m]=0$.
Then we can form the twisted Lie algebra $\alg g^m$ with the same Lie bracket, but differential $d_m=d+[m,-]$. We can furthermore form the twisted ($\alg g^m$-)module $M^m$, which is the same space as $M$, carrying the same action and underlying algebraic structure (operad, operadic module etc.), but whose differential becomes
\[
d_m = d + m\cdot
\]
where $m\cdot$ shall denote the action of $m$ and we denote the original differential on $M$ by $d$. 
Next suppose that $m'\in \alg g$ is another Maurer-Cartan element.
We say that $m$ and $m'$ are gauge equivalent if there is a Maurer-Cartan element
$\hat m\in \alg g[t,dt]$ whose restriction to $t=0$ agrees with $m$, and whose restriction to $t=1$ agrees with $m'$.
More concretely, 
\[
\hat m = m_t + dt h_t
\]
where $m_t$ can be understood as a family of Maurer-Cartan elements in $\alg g$, connected by a family of infinitesimal homotopies (gauge transformations) $h_t$.
The Maurer-Cartan equation for $\hat m$ translates into the two equations
\begin{align*}
dm_t+\frac 1 2 [m_t,m_t]&=0
&
\frac{\partial m_t}{\partial t}+ dh_t +[h_t,m_t]&=0.
\end{align*}
Now suppose that $\alg g$ is pro-nilpotent. Then we may form the exponential group $\Exp(\alg g)$, which is identified with the degree $0$ subspace $\alg g_0\subset \alg g$, with group product given by the Baker-Campbell-Hausdorff formula. 
We can integrate the flow of $h_t$ into the element $H_t \in \Exp(\alg g)$, which acts on $x\in \alg{g}$ by

\[
H_t(x ) = \exp(h_t)\cdot x = \alpha + \sum_{n\geq 0}\frac{\ad^n(h_t)}{(n+1)!}\left([h_t,x] -dh_t\right)
\]

The action of $H_t$ is compatible with the Lie bracket and has the property that, for every $x\in \alg g$
\[
H_t (dx + [m,x]) = (d+[m_t,-])H_t( x).
\]
In particular, the action of $H_1$ induces an isomorphism of dg Lie algebras
\[
H_1 : \alg g^m\to \alg g^{m'}.
\]
Next suppose that also the action of $\alg g$ on $M$ is pro-nilpotent.
Then, by a similar argument, the action of $H_1$ yields an isomorphism
\begin{equation}\label{equ:H1onmodules}
H_1\cdot : M^m \to M^{m'}.
\end{equation}

Now let us relate these general statements to the objects of relevance in this paper.
First consider $\alg g=\GC_D$ to be the graph complex, but as a graded Lie algebra, i.e., considered with zero differential.
The correct differential on the graph complex is then obtained by twisting with the Maurer-Cartan element \cite{W1}
\[
m_0= \begin{tikzpicture}[baseline=-.65ex]
\node[int] (v) at (0,0){};
\node[int] (w) at (.5,0){};
\draw (v) edge (w);
\end{tikzpicture}
\]
Furthermore, consider $M=\pdGraphs_D$, again with zero differential. There is a natural action of $\alg g$ on $M$ \cite{W1,DW}. The differential on $\pdGraphs_D=M^{m_0}$ is then reproduced by twisting with $m_0$.

Secondly, the above picture can be extended to include the (co)operadic right modules.
First, $\GC_D$ acts on $\GC_{H(M)}$.
We take
\[
\alg g = \GC_D\ltimes \GC_{H(M)}
\]
where we consider again the first factor with trivial differential, and the second factor only with the part of the differential joining two decorations to an edge.
The element $m_0$ from above is then a Maurer-Cartan element, and twisting by this Maurer-Cartan element reproduces the differential on the factors of $\alg g$ considered above.
Similarly, we may consider the Maurer-Cartan elements
\[
m' := m_0 + z_0
\]
or 
\[
m_M := m_0 + z_M
\]
where $z_0$, $z_M$ are as above. Twisting then reproduces on the second factor either the differential on $\GC_{H(M)}'$, or that on $\GC_M$.

Next consider for $M$ the pair consisting of a cooperad and a comodule $(\pdGraphs_D, \pdGraphs_{M})$, where the first factor we consider with the zero differential, and in the second we consider only the part that connects two decorations to an edge.
Then twisting with the Maurer-Cartan element $m_M$ reproduces the full differential on the factors.

\begin{rem}
Note that an immediate consequence of the above way of constructing $\pdGraphs_{M}$ is that one has a large class of (co)derivations at hand. Namely, we have an action of $\alg g^{z_M}$ on $M^{z_M}$. In particular, it was shown in \cite{W1} that the 0-th cohomology of $\GC_2$ is the Grothendieck-Teichm\"uller algebra $\grt_1$. Hence, overstretching the analogy a bit, we may consider the dg Lie algebra $\alg g^{z_M}$, consisting of factors $\GC_D$ and $\GC_M$, as a version of the Grothendieck-Teichm\"uller dg Lie algebra associated to the manifold $M$.
Note however that this ``definition" is a little provisional: A more invariant definition would be to define the $M$-Grothendieck-Teichm\"uller Lie algebra as the homotopy derivations of a real model of the pair $(\FM_D,\FM_M)$. It is yet an open question in how far the homotopy derivations in $\alg g^{z_M}$ exhaust all homotopy derivations. For example, $\alg g^{z_M}$ itself does not readily capture the (non-nilpotent) action of the Lie algebra $\alg o(H(M))$ (of linear maps that preserve the pairing) on all objects involved.
\end{rem}

Next, let us note that the right comodule $\pdGraphs_M$ is unaltered (up to isomorphism) if one replaces the Maurer-Cartan element $z_M$ used in its definition by a gauge equivalent Maurer-Cartan element. Indeed, the action of $\GC_{H(M)}$ is nilpotent since the action of any element in $\GC_{H(M)}$ always kills at least on vertex.
Hence given two gauge-equivalent Maurer-Cartan elements an explicit isomorphism between the two version of $\pdGraphs_M$ produced is given by \eqref{equ:H1onmodules}.

Finally, let us note that the above construction works equally well for the tadpole free version $\pdGraphs_M^\notadp$ of $\pdGraphs_M$.
In this case, one needs to work with the tadpole-free version of the graph complex $\GC_M$.
Also, in this case one does not have a right $\pdGraphs_D$ coaction.

%
%
%
%
%


\subsection{Valence conditions}

In this section we show that the Hopf comodule $^*\Graphs_M$ is quasi-isomorphic to (essentially) a quotient that can be identified with graphs containing only $\geq3$-valent internal vertices. For this, we would like that the Maurer-Cartan element (partition function) $z_M'$ above vanished on the subspace spanned by graphs containing a $\leq 2$-valent internal vertex. While this might not be the case in general, we show that $z_M$ is gauge equivalent to a partition function satisfying this property.

%
%
%
%
%

\begin{lemma}\label{lem: trivallent GC_M}
The subspace $\GC_{H^\bullet(M)}^{\geq 3}\subset \GC_{H^\bullet(M)}'$ spanned by graphs having no $1$ or $2$-valent vertex is a dg Lie subalgebra.
\end{lemma}

\begin{proof}
$\GC_{H^\bullet(M)}^{\geq 3}$ is closed under the Lie bracket since it does not decrease the valence of vertices. It remains to check the stability under the differential.

Recall that the differential has three pieces, a first one that splits an internal vertex, a second one that joins decorations into an edge, and a third one arising from the twist by $z_0$.  Joining decorations into an edge cannot decrease the valency on vertices and therefore preserves $\GC_{H^\bullet(M)}^{\geq 3}$. Univalent or bivalent vertices can a priori be created both by the second and third term in the differential. However, one easily checks that these $\leq 2$-valent contributions cancel due to signs. For example, when computing the differential of the graph\begin{tikzpicture}[scale=0.5]
\node(a) at (-0.75,-0.5) [int] {.};

\node (b) at (0.5,-0.5) [int] {.};

\draw (a)--(b);

\draw (a) -- (-1.25,0.5);
\draw (a) -- (-1.75,0);
\draw (a) -- (-1.75,-1);
\draw (a) -- (-1.25,-1.5);
\draw (b) -- (1,0.8);
\draw (b) -- (1.3,0.3);
\draw (b) -- (1.5,-0.7);
\draw (b) -- (1,-1.2);
\end{tikzpicture} bivalent vertices are created by vertex splitting \begin{tikzpicture}[scale=0.5]
\node (a) at (-1,-0.5) [int] {.};

\node (black) at (0,-0.25) [int] {.};

\node (b) at (1,-0.5) [int] {.};

\draw (a)--(black)--(b);

\draw (a) -- (-1.5,0.5);
\draw (a) -- (-2,0);
\draw (a) -- (-2,-1);
\draw (a) -- (-1.5,-1.5);
\draw (b) -- (1.5,0.8);
\draw (b) -- (1.8,0.3);
\draw (b) -- (2,-0.7);
\draw (b) -- (1.5,-1.2);
\end{tikzpicture}. However,  since there are two contributions corresponding to each of the two vertices and they appear with opposite signs thus canceling out.
For bivalent vertices carrying a decoration, or for a univalent vertex, the argument is similar.
\end{proof}

Let $\GC_{H^\bullet(M)}''$ be the subspace of $\GC_{H^\bullet(M)}'$ spanned by graphs that (i) do not contain any univalent vertices, and (ii) that contain at least one $\geq 3$-valent vertex. 
Notice that $\GC_{H^\bullet(M)}''$ is a sub-Lie algebra of $\GC_{H^\bullet(M)}$ since the Lie bracket can not decrease any valences. 
Furthermore, we have the following easy result.
\begin{lemma}\label{lem:MC is geq 3 valent}
The Maurer-Cartan element $z_M'\in \GC_{H^\bullet(M)}'$ constructed above lives in the subspace $\GC_{H^\bullet(M)}''$.
\end{lemma}
\begin{proof}
First note that by definition $z_M'$ contains no graphs with a single $\leq 2$-valent vertex, as those graphs have been absorbed into $z_0$ above.
Hence the only instance of a (connected) graph with a univalent vertex is a graph with an ``antenna", i.e., an edge connected to a univalent vertex. However, to such graphs the configuration space integral formula associates weight 0, by property \textit{(iv)} of Proposition \ref{prop:angular form} (or alternatively by a degree argument, since there are not enough form degrees depending on the position of the antenna vertex).
Next, if the graph has no trivalent vertex, it is either a string, with some decorations at the ends, or a loop. In case of a string, the weight is zero again by \textit{(iv)} of Proposition \ref{prop:angular form}. Finally, the loops all have zero weight by degree reasons.
\end{proof}

The following Proposition is essentially proven in \cite[Prop. 3.4]{W1}. One uses essentially the dual argument of Theorem \ref{thm:trivalent}.
\begin{prop}\label{prop:3-vallent is quasi-isomorphic}
The inclusion map $\GC_{H^\bullet(M)}^{\geq 3}\hookrightarrow \GC_{H^\bullet(M)}''$ is a quasi-isomorphism of Lie algebras. Furthermore, endowing both sides with the descending complete filtrations by the number of non-bivalent vertices\footnote{On $\GC_{H^\bullet(M)}^{\geq 3}$ this filtration is albeit quite trivial.}, the map between the associated graded spaces is already a quasi-isomorphism. 
\end{prop}

Due to this Proposition we can apply the Goldman-Millson Theorem \cite{DR} to conclude that any Maurer-Cartan element in $\GC_{H^\bullet(M)}''$ is gauge equivalent to a Maurer-Cartan element in the subspace $\GC_{H^\bullet(M)}^{\geq 3}$. In particular:

\begin{cor}
The Maurer-Cartan element $z_M'$ is gauge equivalent to a Maurer-Cartan element in the subspace $\GC_{H^\bullet(M)}^{\geq 3}$.
\end{cor}

Next, we apply the remark of the previous subsection to conclude that we may use a $\geq$trivalent Maurer-Cartan element (say $z_M^3$) gauge equivalent to $z_M'$ to construct $\pdGraphs_M$.
For the sake of concreteness, let us temporarily (for this subsection) denote the version of $\pdGraphs_M$ constructed as before by $\Graphs_M^{z_M'}$, and the one constructed with $z_3$ instead by $\pdGraphs_M^{z_3}$, though this notation is abusive.

Let us consider a subspace $S$ of $\pdGraphs_M^{z_3}$ spanned by graphs having at least one internal  $1$- or $2$-valent vertex. Recall that decorations count as valence and there are no $0$-valent internal vertices in $^*\Graphs_M$.

\begin{lemma}\label{lem: trivallent}
The space $S$ described above is a subcomplex of $^*\Graphs_M^{z_3}$.
\end{lemma}

\begin{proof}
Recall that the differential has two pieces, a first one that contracts an edge connected to an internal vertex and a second one that either cuts an edge into the diagonal class or deletes a subgraph of internal vertices producing a factor given by the image of such subgraph under $Z_M$. Due to the Maurer-Cartan element $z_3$ containing only $\geq 3$-valent diagrams, the differential cannot cut out a subgraph containing a bivalent internal vertex.
Let us consider a graph with a $2$-valent internal vertex that is adjacent to two other vertices. There, the differential acts as follows:

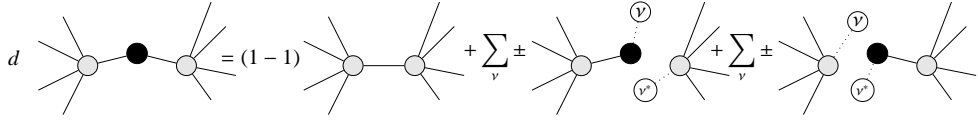
\resizebox{13cm}{!}{
\begin{tikzpicture}

\node at (0,0) {$d$};

\node at (2,0) {\begin{tikzpicture}[scale=0.8]
\node[fill=gray!20] (a) at (-1,-0.5) [ext] {\phantom{a}};

\node (black) at (0,-0.25) [int] {$a$};

\node[fill=gray!20] (b) at (1,-0.5) [ext] {\phantom{a}};

\draw (a)--(black)--(b);

\draw (a) -- (-1.5,0.5);
\draw (a) -- (-2,0);
\draw (a) -- (-2,-1);
\draw (a) -- (-1.5,-1.5);
\draw (b) -- (1.5,0.8);
\draw (b) -- (1.8,0.3);
\draw (b) -- (2,-0.7);
\draw (b) -- (1.5,-1.2);
\end{tikzpicture}};

\node at (4,0) {$= (1-1)$};

\node at (6,0) {\begin{tikzpicture}[scale=0.8]
\node[fill=gray!20] (a) at (-0.75,-0.5) [ext] {\phantom{a}};

\node[fill=gray!20] (b) at (0.5,-0.5) [ext] {\phantom{a}};

\draw (a)--(b);

\draw (a) -- (-1.25,0.5);
\draw (a) -- (-1.75,0);
\draw (a) -- (-1.75,-1);
\draw (a) -- (-1.25,-1.5);
\draw (b) -- (1,0.8);
\draw (b) -- (1.3,0.3);
\draw (b) -- (1.5,-0.7);
\draw (b) -- (1,-1.2);
\end{tikzpicture}};

\node at (7.8,0) {$+\displaystyle\sum_\nu \pm $};

\node at (10,0) {\begin{tikzpicture}[scale=0.8]
\node[fill=gray!20] (a) at (-1,-0.5) [ext] {\phantom{a}};

\node (black) at (0,-0.25) [int] {$a$};

\node (nu) at (0.2,0.6) [ext] {$\nu$};

\node (nustar) at (0.3,-1) [ext] {$\scriptstyle\nu^*$};

\node[fill=gray!20] (b) at (1,-0.5) [ext] {\phantom{a}};

\draw[dotted] (black)--(nu) (nustar)--(b);
\draw (a)--(black);

\draw (a) -- (-1.5,0.5);
\draw (a) -- (-2,0);
\draw (a) -- (-2,-1);
\draw (a) -- (-1.5,-1.5);
\draw (b) -- (1.5,0.8);
\draw (b) -- (1.8,0.3);
\draw (b) -- (2,-0.7);
\draw (b) -- (1.5,-1.2);
\end{tikzpicture}};

\node at (11.8,0) {$+\displaystyle\sum_\nu \pm $};

\node at (14,0) {\begin{tikzpicture}[scale=0.8]
\node[fill=gray!20] (a) at (-1,-0.5) [ext] {\phantom{a}};

\node (black) at (0,-0.25) [int] {$a$};

\node (nu) at (-0.4,0.4) [ext] {$\nu$};

\node (nustar) at (-0.3,-1) [ext] {$\scriptstyle\nu^*$};

\node[fill=gray!20] (b) at (1,-0.5) [ext] {\phantom{a}};

\draw[dotted] (a)--(nu) (nustar)--(black);
\draw (b)--(black);

\draw (a) -- (-1.5,0.5);
\draw (a) -- (-2,0);
\draw (a) -- (-2,-1);
\draw (a) -- (-1.5,-1.5);
\draw (b) -- (1.5,0.8);
\draw (b) -- (1.8,0.3);
\draw (b) -- (2,-0.7);
\draw (b) -- (1.5,-1.2);
\end{tikzpicture}};

\end{tikzpicture}}

The contributions of contracting both edges appear with opposite signs and therefore cancel. Notice that $1$-valent internal vertices are produced on the other summands when the decoration of the internal vertex takes the value $1$.

If there is a $2$-valent internal vertex that is adjacent to only one other vertex and has one decoration, the action of the differential there is:

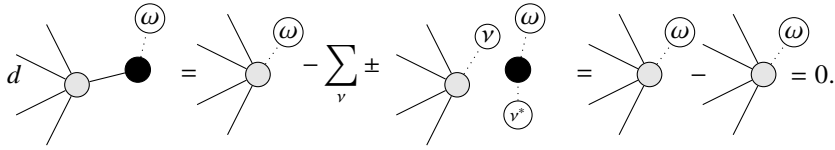
\begin{tikzpicture}
\node at (-1,0) {$d$};
\node at (0,0) {\begin{tikzpicture}[scale=0.8]
\node[fill=gray!20] (a) at (-1,-0.5) [ext] {\phantom{a}};

\node (black) at (0,-0.25) [int] {$a$};

\node (nu) at (0.2,0.6) [ext] {$\omega$};

\draw[dotted] (black)--(nu);
\draw (a)--(black);

\draw (a) -- (-1.5,0.5);
\draw (a) -- (-2,0);
\draw (a) -- (-2,-1);
\draw (a) -- (-1.5,-1.5);

\end{tikzpicture}};

\node at (1.3,0) {$=$};

\node at (2.1,0) {\begin{tikzpicture}[scale=0.8]
\node[fill=gray!20] (a) at (-1,-0.5) [ext] {\phantom{a}};
\node (nu) at (-0.5,0.2) [ext] {$\omega$};

\draw[dotted] (a)--(nu);

\draw (a) -- (-1.5,0.5);
\draw (a) -- (-2,0);
\draw (a) -- (-2,-1);
\draw (a) -- (-1.5,-1.5);
\end{tikzpicture}};

\node at (3.3,0) {$-\displaystyle \sum_\nu \pm $};

\node at (5,0) {\begin{tikzpicture}[scale=0.8]
\node[fill=gray!20] (a) at (-1,-0.5) [ext] {\phantom{a}};
\node (realnu) at (-0.5,0.3) [ext] {$\nu$};

\node (black) at (0,-0.25) [int] {$a$};

\node (nu) at (0.2,0.6) [ext] {$\omega$};
\node (nustar) at (0,-1) [ext] {$\scriptstyle\nu^*$};

\draw[dotted] (nustar)--(black)--(nu) (a)--(realnu);

\draw (a) -- (-1.5,0.5);
\draw (a) -- (-2,0);
\draw (a) -- (-2,-1);
\draw (a) -- (-1.5,-1.5);

\end{tikzpicture}};

\node at (6.5,0) {$=$};

\node at (7.25,0) {\begin{tikzpicture}[scale=0.8]
\node[fill=gray!20] (a) at (-1,-0.5) [ext] {\phantom{a}};
\node (nu) at (-0.5,0.2) [ext] {$\omega$};

\draw[dotted] (a)--(nu);

\draw (a) -- (-1.5,0.5);
\draw (a) -- (-2,0);
\draw (a) -- (-2,-1);
\draw (a) -- (-1.5,-1.5);
\end{tikzpicture}};

\node at (8,0) {$-$};

\node at (8.75,0)  {\begin{tikzpicture}[scale=0.8]
\node[fill=gray!20] (a) at (-1,-0.5) [ext] {\phantom{a}};
\node (nu) at (-0.5,0.2) [ext] {$\omega$};

\draw[dotted] (a)--(nu);

\draw (a) -- (-1.5,0.5);
\draw (a) -- (-2,0);
\draw (a) -- (-2,-1);
\draw (a) -- (-1.5,-1.5);
\end{tikzpicture}};

\node at (9.5,0) {$=0.$}; 

\end{tikzpicture}

It is easy to see that if there is one $1$-valent internal vertex the two pieces of the differential cancel each other, thus concluding the proof.
\end{proof}

The following proof is an adaptation of \cite[Prop. 3.4]{W1}
\begin{thm}\label{thm:trivalent}
The projection map $^*\Graphs_M^{z_3}\to { \tgraphs_M}\coloneqq  {^*\Graphs_M^{z_3}/S } $ is a quasi-isomorphism of dg Hopf right $\pdGraphs_D$-comodules. 
\end{thm}

\begin{proof}
It suffices to show that $H(S) = 0$.  If we set up a filtration on $S$ by the total number of decorations, on the zeroth page of the spectral sequence we recover $d_0$ as the contracting piece and a piece that cuts out a connected component of internal vertices with a factor given by an integral. We claim that the spectral sequence collapses already on the first page.

Notice that $d_0$ cannot produce $1$-valent internal vertices from $2$-valent internal vertices and it follows from the proof of Lemma \ref{lem: trivallent} that a $1$-valent internal vertex cannot be destroyed.

It follows that on the zeroth page $S$ decomposes as a sum of complexes $S=S_1 \oplus S_2$, where $S_1$ is spanned by graphs with at least one $1$-valent internal vertex and $S_2$ is spanned by graphs whose internal vertices are at least $2$-valent. 

To see that $S_1$ is acyclic one can look at ``antennas" of the graphs, i.e. maximal connected subgraphs consisting of one $1$-valent and some $2$-valent internal vertices. By setting a spectral sequence whose differential decreases only the length of antennas one can construct a contracting homotopy that increases this length thus showing  $H(S_1)=0$.

As for $S_2$ the same idea can used by replacing every path on the graph consisting of 2-valent internal vertices by single edges labeled by their length, see Figure \ref{fig:labeled graphs}.

\begin{figure}[h] 

\begin{tikzpicture}

\node at (0,0)   {\begin{tikzpicture}[scale=0.8]
\node[fill=gray!20] (a) at (-1,-0.5) [ext] {\phantom{a}};

\node (black) at (0,-0.3) [int] {$a$};
\node (black2) at (1,-0.25) [int] {$a$};
\node (black3) at (2,-0.3) [int] {$a$};

\node[fill=gray!20] (b) at (3,-0.5) [ext] {\phantom{a}};

\draw (a)--(black)--(black2)--(black3)--(b);

\draw (a) -- (-1.5,0.5);
\draw (a) -- (-2,0);
\draw (a) -- (-2,-1);
\draw (a) -- (-1.5,-1.5);
\draw (b) -- (3.5,0.8);
\draw (b) -- (3.8,0.3);
\draw (b) -- (4,-0.7);
\draw (b) -- (3.5,-1.2);
\end{tikzpicture}};

\node at (3,-0.15) {$=$};

\node at (5,0){\begin{tikzpicture}[scale=0.8]
\node[fill=gray!20] (a) at (-1,-0.5) [ext] {\phantom{a}};

\node[fill=gray!20] (b) at (1,-0.5) [ext] {\phantom{a}};

\draw (b) to [out=160,in=20] node[above] {$\scriptstyle 3$} (a) ;

\draw (a) -- (-1.5,0.5);
\draw (a) -- (-2,0);
\draw (a) -- (-2,-1);
\draw (a) -- (-1.5,-1.5);
\draw (b) -- (1.5,0.8);
\draw (b) -- (1.8,0.3);
\draw (b) -- (2,-0.7);
\draw (b) -- (1.5,-1.2);
\end{tikzpicture}};

\end{tikzpicture}\caption{Replacing bivalent internal vertices by a single labeled edge.\label{fig:labeled graphs}}
\end{figure}
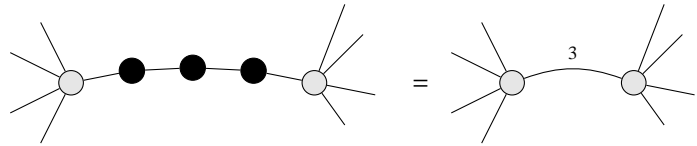

By considering a spectral sequence whose differential on the zeroth page only reduces the numbers on the labels, being careful with the signs one can construct a contracting homotopy which gives $H(S_2)=0$.

\end{proof}

Overall, we conclude that $\tgraphs_M$ is a dgca model for $\FM_M$, by the following explicit zigzag:
\[
\tgraphs_M \xleftarrow{\sim} \pdGraphs_M^{z_3} \xleftarrow{\cong} \pdGraphs_M^{z_M'}\xrightarrow{\sim} \Omega_{PA}(\FM_M).
\]
Moreover, the above maps are morphisms of dg Hopf right comodules. 

If $M$ is not parallelizable, one can construct the space $\tgraphs_M^\cutaol$ as the analogous quotient of $\pdGraphs_M^{z_3}$. The same proof allows us to conclude that $\tgraphs_M^\cutaol$ is a dgca model for the collection of topological spaces $\FM_M$ by a similar zigzag.

\begin{rem}
 The smaller model $\tgraphs_M$ (as well as $\tgraphs_M^\cutaol$) has the advantage that for $D\geq 3$ it is connected in the sense that each dgca $\tgraphs_M(r)$ is concentrated in non-negative cohomological degrees, and one-dimensional in degree 0. 
 This can be shown by a degree counting argument similar to Lemma \ref{lem:degree counting}, using the trivalence condition and the existence of at least one external vertex per connected component.
 Similarly, one sees that if in addition $H^1(M)=0$, then $\tgraphs_M(r)$ is finite dimensional in each cohomological degree.
\end{rem}

\begin{rem}
The propagator $\phi_{12}$ established in Proposition \ref{prop:angular form} can be chosen such that 
$\int_{2} \phi_{12}\alpha=0$, where the integration is conducted along the fiber of the forgetful map $p_2 \colon \FM_M(2)\to M$, and where $\alpha$ is any of the chosen representative forms for the cohomology, see Convention \ref{convention} (see \cite{CM}).
It would be desirable to show that $\phi_{12}$ may be chosen such that in addition $\int_{3} \phi_{13}\phi_{32}=0$, where the integration is performed along the fiber of the forgetful map $p_3 \colon \FM_M(3)\to \FM_M(2)$. In that case the above discussion could be considerably simplified, since the extra condition immediately renders the integral weights of all graphs with bivalent vertices zero. 
A propagator with this desired property has been constructed in the smooth setting in \cite[Lemma 4]{CM}. We expect that the proof carries over to the semi-algebraic setting. However, there is a technical difficulty due to our use of PA instead of smooth forms, whose resolution we leave to future work. Roughly speaking, the technical problem is that for a PA form $\beta\in \Omega(M\times N)$ one has to define a good notion of ``de Rham differential in the first slot'' $d_N \beta$.
\end{rem}

\newcommand{\mce}{z_M}

\subsection{Computing the cohomology and loop orders}

Above we construct real dgca models $\stG_M$ and $\tgraphs_M$ for configuration spaces of points on a manifold $M$, which depend on $M$ only through the Maurer-Cartan element $\mce \in \GC_{H^\bullet(M)}$.
Note that $\GC_{H^\bullet(M)}$ is naturally filtered by the loop order of graphs. We can decompose the Maurer-Cartan element 
\[
\mce = \mce^0 +\mce^1+\cdots
\]
accordingly into pieces of various loop orders.

The differential on $\tgraphs_M(n)$ can only maintain or decrease the number of loops (genus) of the graphs. It follows that the subspace $\tgraphs^{\text{for}}_M(n) \subset \tgraphs_M(n)$ spanned by graphs of genus zero, i.e. forests, is a subcomplex and a dg subalgebra for $n=1$. Notice that however it is not a subalgebra if $n>1$.
In any case the object $\tgraphs^{\text{for}}_M$ depends on $M$ only through the tree-level piece $\mce^0$ of our Maurer-Cartan element $\mce$.

\begin{lemma}
The inclusion of $\tgraphs_M^{\text{for}}$ in $\tgraphs_M$ is a quasi-isomorphism (of symmetric sequences of complexes).
\end{lemma}

\begin{proof}
The proof follows essentially from the spectral sequence argument given in Lemma \ref{lem:[LV] argument}.

 The differential in $\tgraphs_M$ cannot decrease the number of connected components of a graph, so by considering a filtration by the number of connected components of the graphs we obtain the respective associated graded complexes $\gr\tgraphs^{\text{for}}$ and $\gr\tgraphs_M$. Then we notice that the number $\#edges - \#vertices $ cannot increase so we take the respective filtration obtaining the associated graded complexes $\gr'\gr\tgraphs^{\text{for}}$ and $\gr'\gr\tgraphs_M$ (notice that this filtration is bounded below since there are no connected components of only internal vertices). 
 After this, the only piece of the differential remaining is the one cutting out a (decorated) tree of internal vertices and evaluating the partition function on it.
 
 At last, filtering by $\#internal \ vertices - degree$, we obtain in the associated graded complexes $\gr''\gr'\gr\tgraphs^{\text{for}}$ and $\gr''\gr'\gr\tgraphs_M$ a the piece of the differential that reduces the number of internal vertices exactly by $1$, i.e., the differential contracts one edge connected to one or two internal vertices or cuts out a tree consisting only of a single decorated internal vertex.
 
 We claim that the induced inclusion map is a quasi-isomorphism at this level. As in Lemma \ref{lem:[LV] argument}, by induction on $n$ one can show that the homology of $V(n) = \gr''\gr'\gr\tgraphs_M (n)$ satisfies 
 $$H(V(n)) =H(V(n-1)) \otimes H^\bullet(M) \oplus H(V(n-1)) [1-D]^{\oplus n-1},$$
but the same proof gives the same result for the homology of $\gr''\gr'\gr\tgraphs_M^{\text{for}}$, so the result follows.\end{proof}

In particular we see the following.
\begin{enumerate}
\item The dgca $\tgraphs_M^{\text{for}}(1)$ is a real model for $M$, so that the tree-level piece of $\mce$ encodes the real homotopy type of $M$.  
\item Knowledge of the tree-level piece of $\mce$ suffices to compute the real cohomology of $\FM_M(n)$, as a graded vector space, for all $n$.
\end{enumerate}

\section{The real homotopy type of $M$ and $\FM_M$}
\label{sec:proof of conjecture}

%
%


The goal of this section is to compare the information contained in the partition function $z_M$ from above to the real homotopy type of $M$. By the latter, we mean the isomorphism type
 of a homotopy commutative ($\mathsf C_\infty$) algebra structure on the cohomology $H(M)$. The end result will be that the knowledge of the real homotopy type of $M$ suffices to recover $z_M$ (up to gauge equivalence) in the case that $D\geq 4$ and $H^1(M)=0$.

Let us first see how the $\mathsf C_\infty$ algebra structure on $H(M)$ can be obtained from our graphical models.
For every closed oriented connected manifold $M$ we fix the following \textit{homotopy data} of chain complexes

\begin{equation*}
\begin{tikzcd}[cells={nodes={}}]
H^\bullet(M) \arrow[yshift=-.5ex,swap]{r}{i}  
        & \tgraphs_M^{\text{for}}(1) \arrow[yshift=.5ex,swap]{l}{p} 
\arrow[loop right, distance=2em, 
]{}{h} 
    \end{tikzcd}
\end{equation*}
$$    pi= \id , \ \id -ip = dh+hd.$$

Where the map $i$ is defined such that $i(\omega)=$ \begin{tikzpicture}
\node[ext] (v1) at (0,0) {$\scriptstyle 1$};
\node[ext] (i1) at (.5,.3) {$\scriptstyle \omega$};

\draw[dotted] (v1) edge (i1);
\end{tikzpicture},  the map $h$ is defined such that it vanishes on graphs with a $\leq 1$-valent external vertex and
\begin{tikzpicture}
\node at (-0.4,0) {$h$};
\node[ext] (v1) at (0,0) {$\scriptstyle 1$};
\node[ext,dashed] (g1) at (0,0.7) {$\Gamma$};

\node at (0.5,0) {$=$};

\node[ext] (v2) at (1,0) {$\scriptstyle 1$};
\node[ext,dashed] (g2) at (1,1) {$\Gamma$};
\node[int] (int) at (1,0.4) {};

\draw (v2) edge (int);
 
\draw (v1) edge (g1);
\draw[bend right]  (v1) edge (g1);
\draw[bend left]  (v1) edge (g1);

\draw (int) edge (g2);
\draw[bend right]  (int) edge (g2);
\draw[bend left]  (int) edge (g2);

\end{tikzpicture}.

Finally, $p$ is defined such that for every $\Gamma \in \tgraphs_M^{\text{for}}$, $p(\Gamma)=\displaystyle\sum_{i} e_i\int_M e_i^* \wedge f(\Gamma)$, where the $\{e_i\}$ form  a basis of $H^\bullet(M)$ and $\{e_i^*\}$ the respective dual basis and $f\colon \tgraphs_M^{\text{for}}(1)\to \Omega_{PA}(M)$ is the map induced by the one constructed in Section \ref{sec:graphs}.

By the Homotopy Transfer Theorem \cite[Section 10.3]{LV} such homotopy data defines a $\Com_\infty$ structure on $H^\bullet(M)$ and such structure retains the real homotopy type of $M$.

Notice that $\Com_\infty$ structures on $H^\bullet(M)$ are identified with Maurer-Cartan elements in the Harrison complex $$\Harr(H^\bullet(M),H^\bullet(M)) = \Hom_{\mathbb S}(\Lie^c\{1\}[-1]\circ H^\bullet(M), H^\bullet(M)) = \prod_{n\in \mathbb N} \Lie(n) \otimes_{\mathbb S_{n}} H_\bullet(M)^{\otimes n} \otimes H^\bullet(M) [ n].$$


\begin{prop}[\cite{loday}, Proposition 1.6.5]
The projection map $\Harr(H^\bullet(M),H^\bullet(M))\to \Harr(\overline{H^\bullet(M)},H^\bullet(M))$ is a quasi-isomorphism of Lie algebras.
\end{prop}

\begin{lemma}\label{lem:degree counting}
If $M$ is a connected manifold of dimension at least $D\geq 4$ such that $H^1(M):=H^1(M,\R)=0$, then all the degree $0$ graphs in $^*\GC^{\geq 3}_M$ are trees.
\end{lemma}
\begin{proof}
The proof is a simple combinatorial argument. Let $\Gamma\in {}^*\GC^{\geq 3}_M$ be a non-tree graph with $E$ edges and $V$ vertices. We denote the sum of degrees of the decorations of a vertex $v_i$ by $\deg dec(v_i)$ and the number of incident vertices at $v_i$ by $\text{edges}(v_i)$.

From the relation $\sum_{i=1}^V \text{edges}(v_i) = 2E$, it follows
\begin{align*}
\deg(\Gamma) &= (D-1)E-DV+ \sum_{i=1}^{V} \deg dec(v_i)\\
&=(D-3)(E-V)+\sum^V_{i=1}\left( \deg dec(v_i) + \text{edges}(v_i)-3\right).
\end{align*}

Because of the $\geq 3$-valence condition, each term $\deg dec(v_i) + \text{edges}(v_i)-3$ must be greater than or equal te zero. In fact, since decorations have degree at least $2$ if there is at least one decoration in $\Gamma$, the sum $\sum^V_{i=1}\left( \deg dec(v_i) + \text{edges}(v_i)- 3\right)$ is strictly positive.

Now notice that since $\Gamma$ is a not a tree, we have $E\geq V$ and in case of equality there must be at least one decoration. In any of those cases it follows that $\deg \Gamma >0$.
\end{proof}

\begin{rem}
From the proof we also observe the following:
\begin{itemize}
 \item If $D=3$ and $H^1(M)=0$ the only non-tree graphs of degree $0$ have no decorations and every vertex is exactly trivalent. These graphs are also called simple cubic graphs.
\item For $D\geq 4$ but $H^1(M)\ne 0$ there are non-tree graphs of degree zero but they take on a very simple form: Besides trees, there are only graphs of genus $1$ that are trivalent and decorated only by $1$-forms. Such graphs are given by a ``fundamental loop" such that every vertex has a decorated trivalent tree attached. Here is an example:
\[
\begin{tikzpicture} 
\node[int] (v1) at (0:.5){};
\node[int] (v2) at (120:.5){};
\node[int] (v3) at (-120:.5){};
\node[ext] (x2) at (120:1){$\scriptstyle\alpha_2$};
\node[ext] (x3) at (-120:1){$\scriptstyle\alpha_1$};
\node[int] (w) at (1,0) {};
\node[int] (ww) at (1.5,0){};
\node[ext] (xw) at (1,.5){$\scriptstyle\alpha_3$};
\node[ext] (xww) at (2,.5){$\scriptstyle\alpha_4$};
\node[ext] (xwww) at (2,-.5){$\scriptstyle\alpha_5$};
\draw (v1) edge (v2) edge (v3) edge (w) (w) edge (ww) (v2) edge (v3);
\draw[dotted] (v2) edge (x2) (v3) edge (x3) (w) edge (xw) (ww) edge (xww) edge (xwww);
\end{tikzpicture}
\]

\end{itemize}
\end{rem}

From now on, let us suppose $M$ to be simply connected and of dimension $D\geq 4$. 

\begin{defprop}\label{defprop:GCtree}

The dgla $\GC_M^{\geq 3,tree}$ is the quotient of $\GC_M^{\geq 3}$ by the dg Lie ideal spanned by graphs with at least one loop. 

\end{defprop}

\begin{proof}

First notice that the Lie bracket of two graphs $\Gamma,\Gamma'\in \GC_M^{\geq 3}$ will be a sum of graphs with loop order given by the sum of the loop orders of $\Gamma$ and $\Gamma'$. It follows that the subspace spanned by graphs with at least one loop is a Lie ideal. 

The splitting part of the differential preserves the loop order and the part of the differential that connects decorations increases the loop order by one and the twisted piece of the differential does not reduce loops. It follows that the differential preserves the ideal.

\end{proof}

\begin{defprop}\label{defprop:GC tree = GC Lie}
The dgla $\GC_M^{\Lie}$ is defined as the quotient of $\GC_M^{\geq 3,tree}$ by the ideal generated by trees with vertices $\geq 4$-valent and the IHX (or Jacobi) relations that originate from the splitting differential of a $4$-valent vertex.

The quotient map $\GC_M^{\geq 3,tree}\to \GC_M^{\Lie}$ is a quasi-isomorphism.
\end{defprop}

\begin{proof}
It is clear that the differential preserves the ideal.

To see that the quotient map is a quasi-isomorphism, consider first a filtration by $deg-\#edges$ such that on the associated graded the differential cannot increase the number of vertices by more than one. Then, take a second filtration by the number of decorations and notice that on the associated graded we obtain (the cyclic version of) the quasi-isomorphism $\La^{-D-1}\hoLie\to \La^{-D-1}\Lie$.
\end{proof}

The dgla $\GC_M^{\Lie}$ is a cyclic variant of the Harrison complex of $H^\bullet(M)$.
Indeed, let us consider more generally a graded vector space $A=\overline A \oplus \R$, with a degree $-D$ pairing. 
A $\Com_\infty$ structure on $A$ is given by a Maurer-Cartan element in $\Hom(\Lie^c\{1\}[-1]\circ A, A)$ which, via the pairing, can be identified with the space 
\[
\Hom\left(A^{\bullet-D}\otimes \left(\Lie^c\{1\}[-1]\circ A^\bullet\right), \mathbb R\right).
\]

There is a map $A\otimes \left(\Lie^c\{1\}[-1]\circ A\right)[-D] \to {}^*\GC_{\overline A}^{\Lie}$ determined in the following way: A basis of the cooperad $\Lie^c$ can be identified with rooted planar trivalent trees modulo the Jacobi (co)relations. Forgetting about the position of the root and considering it as any other leave, and replacing every leaf with a decoration by $A$ we obtain an element in $^*\GC_{\overline A}^{\Lie}$.

\begin{defi}\label{def:cyclic structure}
Let $A=\overline A \oplus \R$ be a graded vector space with a non-degenerate pairing of degree $-D$. A cyclic $\mathsf C_\infty$ algebra structure on $A$ is a Maurer-Cartan element in $\GC_{\overline A}^{\Lie}$.
\end{defi}

If such a cyclic $\mathsf C_\infty$ algebra structure $z$ maps into a $\Com_\infty$ structure $\mu$ via the dual of the map described before Definition \ref{def:cyclic structure}, we say that $z$ \emph{extends} $\mu$.

\begin{rem}\label{rem:D cyclic algebra}
	Due to the implicit usage of the degree $-D$ pairing, such structure would be more appropriately called a ``$D$-cyclic $\Com_\infty$-algebra''.
\end{rem}

\begin{prop}\label{prop:cyclic structure}
An orientable closed manifold $M$ determines a cyclic $\mathsf C_\infty$ algebra structure on its cohomology $H^{\bullet}(M)$ extending the one arising from the Homotopy Transfer Theorem.
\end{prop}

\begin{proof}

The $\Com_\infty$ structure on $H^\bullet(M)$ is given by a map in $\Hom(\Lie^c\{1\}[-1]\circ H^\bullet(M), H^\bullet(M))$ which, by the Poincar\'e duality pairing is equivalent to an element  $$f\in \Hom\left(H^\bullet(M)\otimes \left(\Lie^c\{1\}[-1]\circ H^\bullet(M)\right), \mathbb R\right).$$


We claim that there is a factorization of $f$ by

\[
\begin{tikzcd}
H^\bullet(M)\otimes \left(\Lie^c\{1\}[-1]\circ H^\bullet(M)\right) \arrow{r}{f} \arrow{d}[swap]{g}& \mathbb R\\
^*\GC_M^{\Lie}\arrow[dashed]{ur}[swap]{Z}
\end{tikzcd}
\]

and the dashed arrow corresponds to a Maurer-Cartan $Z\in \GC_M^{\Lie}$ which is gauge equivalent to the image of $Z^3_M\in \GC_M^{\geq 3,tree}$. 

To show that $f$ factors through $g$ it is sufficient to show that for every $\mu \in \Lie^c\{1\}[-1](n)$ and $\omega_0,\dots,\omega_n\in H^\bullet(M)$, we have $f(\omega_0 \otimes \mu \otimes \omega_1\otimes \dots \otimes \omega_n) =f(\omega_n \otimes \mu \otimes \omega_0\otimes\dots\otimes\omega_{n-1})$, but this follows from the explicit formula the $\Com_\infty$ action given by the Homotopy Transfer Theorem.
This corresponds to computing the partition function on the trivalent graph given by the $\mathsf C_\infty$ operation $\mu$ where the root is replaced by a decoration by the element $\omega_0$, which is clearly cyclically invariant.

 As an example, suppose that $\mu$ corresponds to $\mu_2\circ_1\mu_2 \in \Lie^c(3)$, then

 $$\mu(\omega_1,\omega_2,\omega_3) = p(h(i(\omega_1)i(\omega_2))i(\omega_3))=p\  \begin{tikzpicture}

\node[ext] (v2) at (1,0) {$1$};
\node[ext] (1) at (1,1) {$\scriptstyle\omega_1$};
\node[ext] (2) at (1.4,0.8) {$\scriptstyle\omega_2$};
\node[int] (int) at (1,0.4) {};

\draw (v2) edge (int);
 
 \node[ext] (3) at (1.8,0) {$\scriptstyle\omega_3$};

\draw[dashed]  (int) edge (1);
\draw[dashed]  (int) edge (2);
\draw[dashed] (v2) edge (3);
\end{tikzpicture}=\displaystyle\sum_i e_i \int_{1,2}\pi^*_1(e_i^*)\pi_1^*(\omega_3)\phi_{1,2}\pi_2^*(\omega_1)\pi_2^*(\omega_2).$$
Therefore, $f(\omega_0,\mu(\omega_1,\omega_2,\omega_3))=\int_{1,2}\pi^*_1(\omega_0)\pi_1^*(\omega_3)\phi_{1,2}\pi_2^*(\omega_1)\pi_2^*(\omega_2)=Z \  \begin{tikzpicture}

\node[int] (v2) at (1,0) {};
\node[ext] (1) at (1,1) {$\scriptstyle\omega_1$};
\node[ext] (2) at (1.4,0.8) {$\scriptstyle\omega_2$};
\node[int] (int) at (1,0.4) {};

\draw (v2) edge (int);
 
 \node[ext] (3) at (1.8,0) {$\scriptstyle\omega_3$};

\draw[dashed]  (int) edge (1);
\draw[dashed]  (int) edge (2);
\draw[dashed] (v2) edge (3);

\node[ext] (1) at (0.4,0) {$\scriptstyle\omega_0$};
\draw[dashed] (v2) edge (1);

\end{tikzpicture}$.

\end{proof}

\begin{rem}
For simply connected $\geq 4$-dimensional $M$, the cyclic $\mathsf C_\infty$ structure on $H^\bullet(M)$ determines the spaces $\tgraphs^\cutaol_M(n)$, which encode the real homotopy type of $\FM_M(n)$. Moreover, if $M$ is parallelized, the cyclic $\mathsf C_\infty$ structure determines the Hopf comodule structure of $\tgraphs_M$, that encodes the real homotopy type of $\FM_M$ seen as a right $\FM_D$-module.
\end{rem}

Finally, one can check that the isomorphism type of the (non-cyclic) $\mathsf C_\infty$ algebra structure on $H(M)$ already determines the one of the cyclic $\mathsf C_\infty$ algebra structure. In other words, the cyclicity is not to be seen as extra data on, but rather a property of the real homotopy type, reflecting Poincar\'e duality.
More concretely, the following result has been shown in \cite[Theorems 5.5, 5.8]{HL}. 
We also sketch a short proof here for completeness.

\begin{prop}
The real homotopy type of a closed orientable manifold determines its cyclic homotopy type.
More precisely, given two cyclic $\mathsf C_\infty$ algebra structures on $H(M)$ that are $\mathsf C_\infty$ isomorphic as non-cyclic $\mathsf C_\infty$ structures, they are also isomorphic as cyclic $\mathsf C_\infty$ structures.
\end{prop}

\begin{proof}[Proof sketch]
We are given 2 cyclic $\mathsf C_\infty$ structures on $H(M)$ and a $\mathsf C_\infty$ isomorphism between them. We may assume that the linear part of the $\mathsf C_\infty$ isomorphism is the identity, otherwise we just pull back one cyclic $\mathsf C_\infty$ structure along this linear part.
Note also that the implicit underlying non-degenerate pairing on $H(M)$ is determined by the product up to an unimportant scale factor, so we may assume it is the same for both our cyclic $\mathsf C_\infty$ structures.

We denote by $\mu_1,\mu_2$ the two Maurer-Cartan elements in $\GC_{H(M)}^{\Lie}$ encoding our cyclic $\mathsf C_\infty$ structures. The underlying (non-cyclic) $\mathsf C_\infty$ structure is encoded by the images of $\mu_1,\mu_2$ under the natural inclusion of dg Lie algebras into the reduced Harrison complex
$$\mathbf{root}\colon \GC_{H(M)}^{\Lie}\to \Harr(\overline{H^\bullet(M)},H^\bullet(M)).$$
Graphically, elements on the left-hand side can be interpreted as linear combinations of non-rooted Lie trees, and elements of the right-hand sides can be seen as rooted Lie trees as above, and the map $\mathbf{root}$ is defined by summing over all possible ways of making one leaf into the root. 

The $\mathsf C_\infty$ morphism between our two $\mathsf C_\infty$ structures (with linear term being the identity) then yields a gauge equivalence between the MC elements $\mathbf{root}(\mu_1)$ and $\mathbf{root}(\mu_2)$ in $\Harr(\overline{H^\bullet(M)},H^\bullet(M))$. We desire to check that this implies that $\mu_1$ and $\mu_2$ are already gauge equivalent in $\GC_{H(M)}^{\Lie}$. To this end we can employ the Goldman-Millson Theorem \cite{DR}.
To check the conditions of this Theorem we consider a filtration such that $\mF^p\Harr(\overline{H^\bullet(M)},H^\bullet(M))$ is spanned by rooted Lie trees with $\geq p$ leafs that are decorated by classes of nonzero degree.

On the associated graded the only piece of the differential that survives replaces the root (say decorated by some $\alpha\in H_k(M)$) by two leafs, with one decorated $\alpha$ and the new root decorated with $1\in H_0(M)$.
$$
\begin{tikzpicture}
\node at (0,0) {
\begin{tikzpicture}[scale=0.8]
\node (a) at (-1,-0.5) [int] {.};

\node (b) at (0,-0.5) [ext, label=0:{\tiny root}] {$\alpha$};

\draw (a) -- (-2,0);
\draw (a) -- (-2,-1);
\draw[dashed] (a)--(b);

\end{tikzpicture}};

\node at (1.5,0.1) {$\stackrel{d_0}{\mapsto}$};

\node at (3.5,0) {
\begin{tikzpicture}[scale=0.8]
\node (a) at (-1,-0.5) [int] {.};
\node (a2) at (0,-0.5) [int] {.};

\node (b2) at (1,-1) [ext, label=0:{\tiny root}] {$\scriptstyle 1$};
\node (b1) at (1,0) [ext] {$\scriptstyle \alpha$};

\draw (a) -- (-2,0);
\draw (a) -- (-2,-1);
\draw (a) -- (a2);
\draw[dashed] (a2)--(b1);
\draw[dashed] (a2)--(b2);

\end{tikzpicture}};

\end{tikzpicture}$$
It is an easy exercise to check that the cohomology of the $p$-th graded piece of the Harrison complex is identified for $p\geq 3$ precisely with non-rooted trees all of whose leaves are decorated by elements of $\bar H_\bullet(M)$. But this is precisely the image of 
$\GC_{H(M)}^{\Lie}$ under the map $\mathbf{root}$.

Hence the Goldman-Millson Theorem is applicable to the inclusion of dg Lie algebras $\mathbf{root}\colon \GC_{H(M)}^{\Lie}\to \mF^2\Harr(\overline{H^\bullet(M)},H^\bullet(M))$. 
To conclude the desired result we then just need to remark that our gauge equivalence between 
$\mathbf{root}(\mu_1)$ and $\mathbf{root}(\mu_2)$ in $\Harr(\overline{H^\bullet(M)},H^\bullet(M))$ may actually be taken in $\mF^2\Harr(\overline{H^\bullet(M)},H^\bullet(M))$. To see this in turn one also computes the $p$-th graded piece of the Harrison complex for $p=2$, and sees that there is no cohomology in the at least quadratic part.
But since the underlying $\mathsf C_\infty$ morphism has trivial linear part, we may always remove the parts in the $2$-graded piece by adding an exact terms, to yield the required gauge equivalence in $\mF^2$.
\end{proof}

The real homotopy type of a manifold determines its cyclic homotopy type by the previous proposition. This in turn determines the (gauge equivalence class of) the Maurer--Cartan element $z_M$ by Propositions \ref{defprop:GC tree = GC Lie}, \ref{prop:3-vallent is quasi-isomorphic} and Lemma \ref{lem:degree counting} which itself determines the quasi-isomorphism type of the graph complex by the discussion in Section \ref{sec:twisting of modules}. 
We obtain thus the following Theorem as a corollary:

\begin{thm}\label{thm:ht only dep on htM} 
 Let $M$ be an orientable compact manifold without boundary of dimension $D\geq 4$, such that $H^1(M,\R)=0$. Then the real homotopy type of $\FM_M$ depends only on the real homotopy type of $M$. 
By this statement we mean that there is a zigzag of quasi-isomorphisms of symmetric sequences of dgcas over $\R$ 
$$
\Omega_{PA}(\FM_M) \to \cdot \leftarrow X
$$
with $X$ being a sequence of dgcas defined using only knowledge of the quasi-isomorphism class of $\Omega_{PA}(M)$ as a real dgca.
\end{thm}

\begin{rem}
We remark that we generally work with unbounded cochain complexes, and a priori in the zigzag as constructed above there will occur dgcas which have unbounded degrees.
However, the concrete $X$ we use is (cf. above) $X=\pdGraphs_M^{\geq 3}$, which is concentrated in non-negative degrees.
Furthermore, $X$ is cofibrant in the category of sequences of (unbounded) dgcas, and by homotopy lifting of the zigzag we may in fact construct a quasi-isomorphism of dgcas $X\to \Omega(\FM_M)$.
For the statement above it is hence inessential whether we work over non-negatively graded cochain complexes or cochain complexes of unbounded degrees.
\end{rem}

Moreover, if we suppose $M$ to be parallelized, the action of the Lie algebra $\GC_M$ on $\Graphs_M$ is compatible with the right $\Graphs_D$ module structure. 
In this case, the (real homotopy type) of $\Graphs_M$ as a  right $\Graphs_D$ module is determined by (the gauge equivalence class of) the Maurer-Cartan element $z_M$. 
In that case, by the same argument we obtain a stronger version of the previous Theorem.

\begin{thm}
 Let $M$ be a parallelizable compact manifold without boundary of dimension $D\geq 4$, such that $H^1(M,\R)=0$. Then the real homotopy type of the operadic right module $\FM_M \aor \FM_D$ depends only on the real homotopy type of $M$, in the sense that there is a zigzag of quasi-isomorphisms of right dg Hopf comodules connecting $\Omega_{PA}(\FM_M)$ and some $X$, with $X$ depending only on the quasi-isomorphism type of the dgca $\Omega_{PA}(M)$.
\end{thm}

We note again that we abuse slightly the notation since $\Omega_{PA}(\FM_D)$ is not (strictly speaking) a dg Hopf cooperad and $\Omega_{PA}(\FM_M)$ is not a right comodule, see Remark \ref{rem:almost cooperad}.
The cleaner variant of stating the above Theorem is to work in a category of homotopy cooperads and homotopy comodules, whose construction we however leave to future work, cf. \cite[section 3]{LV}.

\section{The framed case in dimension $D=2$}\label{sec:framed case}

In Section \ref{sec:graphs} we considered parallelized manifolds since a trivialization of the tangent bundle is needed to define the right operadic $\FM_D$-module structure. Informally, to define the action one needs to know in which direction to insert, and the parallelization provides us the direction of the insertion.

 In this section we wish to focus on the 2-dimensional case where unfortunately the only parallelizable (connected closed) manifold is the torus. 

 To go around the problem of not having a consistent choice of direction of insertion, instead of working with configuration spaces of points, we consider the framed configuration spaces. In other words, at every point of the configuration there is the additional datum of a direction, i.e. an element of the Lie group $\SO = S^1$. 

In this section $\Sigma$ shall denote a connected oriented closed surface with a smooth and semi-algebraic manifold structure. Most results will be an adaptation of the arguments in the previous sections to the framed case.

\subsection{Definitions}
In this section we introduce the compactification of the configuration space of framed points on $\Sigma$. A more detailed introduction to the subject can be found in \cite{MSS}.

\subsubsection{The operad of configurations of framed points}

The construction of the operad of the framed version of $\FM_2$ is a special case of the notion of the semi-direct product of an operad and a group, as described below.

\begin{defi}
Let $\cP$ be a topological operad such that there is an action of a topological group $G$ on every space $\cP(n)$ and the operadic compositions are $G$-equivariant. The semi-direct product $ \cP\rtimes G$ is a topological operad with $n$-spaces

 $$( \cP\rtimes G)(n) = G^n \times \cP(n),$$ 
and composition given by $$(\overline{g},p) \circ_i (\overline{g'},p') = \left(g_1,\dots, g_{i-1}, g_ig_1',\dots,g_ig_m',g_{i+1}, \dots,g_n , p\circ_i (g_i \cdot p')\right),$$ where $\overline g = (g_1,\dots,g_n)$ and $\overline {g'}=(g'_1,\dots,g'_m) $.

\end{defi}

The group $\SO$ has a well defined action on $\FM_2$ given by rotation.

\begin{defi}
The Framed Fulton-MacPherson topological operad $\FFM_2$ to be the semi-direct product $\FM_2\rtimes \SO$.
\end{defi}

 When the operadic composition is performed, the configuration inserted rotates according to the frame on the point of insertion as depicted in Figure \ref{fig:composition FFM}, where at every point we draw a small line indicating the associated element of $\SO$.

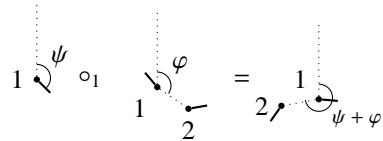
\begin{figure}[h]\label{fig:composition FFM}

\begin{tikzpicture}[scale=1]

\node[label=left:$1$][fill,circle,scale=0.25] (1) at (1.8,1) {};

\draw[thick] (1) --  +(-45:0.25cm);
\draw [dotted] (1)-- +(90:1cm);
\draw (1)+(0,0.2) arc (90:-45:0.18cm);

\node  at (2.1,1.3)  {$\psi$};

\node at (2.5,1) {$\circ_1$};

\node at (3.5,1) {  \begin{tikzpicture}[scale=1]

\node[label=-115:$1$][fill,circle,scale=0.25] (1) at (1.8,1)  {};
\path (1) -- +(-35:0.5cm) node[label=below:$2$][fill,circle,scale=0.25] (2) {};

\draw[thick] (1) --  +(130:0.25cm);
\draw[thick] (2) --  +(10:0.25cm);
\draw [dotted] (1)-- +(90:1cm);
\draw (1)+(0,0.2) arc (90:-35:0.18cm);
\draw [dotted] (1)--(2);

\node  at (2.1,1.3)  {$\varphi$};
\end{tikzpicture}

};

\node at (4.5,1) {$=$};

\node at (5.5,1) {\begin{tikzpicture}[scale=1]

\node[label=above left:$\small{1}$][fill,circle,scale=0.25] (1) at (1.8,1)  {};
\draw [dotted] (1)--+(-170:0.5cm) node[label=left:$\small{2}$][fill,circle,scale=0.25] (2)  {};

\draw [dotted] (1)-- +(90:1cm);
\draw (1)+(0,0.2) arc (90:-170:0.18cm);

\draw[thick] (1) --  +(-5:0.25cm);
\draw[thick] (2) --  +(-125:0.25cm);

\node  at (2.3,0.75)  {\footnotesize{$\psi+\varphi$}};

\end{tikzpicture}};

\end{tikzpicture}

\caption{Operadic composition in $\FFM_2$. }
\end{figure}

\subsubsection{Configurations of framed points on a surface}

\begin{defi}
The Fulton-MacPherson compactification of the configuration spaces of points on the surface $\Sigma$, $\FFM_\Sigma$ is a symmetric sequence in semi-algebraic smooth manifolds which is given as the pullback of the following diagram

\begin{center}
	
\begin{tikzcd}
 \phantom{.}& \mathrm{SO}(\Sigma)^{\times n} \arrow{d}{\pi^n} \\
\FM_\Sigma(n) \arrow{r} & \Sigma^{\times n} 
\end{tikzcd} 

\end{center}
where $\pi\colon \mathrm{SO}(\Sigma) \to \Sigma$ is the frame bundle over $\Sigma$ (assuming some Riemannian metric).
\end{defi}

As in the non-framed case, the space $\FFM_\Sigma(n)$ is a manifold with corners. The interior of this manifold is the framed configuration space of points and is denoted by $\mathsf F \Conf_n(\Sigma)$.

\begin{prop}
The insertion of points at the boundary of $\FFM_\Sigma$ according to the direction of the frame defines a right $\FFM_2$ operadic module structure on  $\FFM_\Sigma$.
\end{prop}

The associativity of the operadic module structure is clear.

\subsection{Graphs}
In this subsection we work with the operadic module $\BVGraphs_\Sigma \aor \BVGraphs_2$ which is the version of $\Graphs_\Sigma \aor \Graphs_{2}$ adapted to the framed case.

Informally, the difference between $\Graphs_\Sigma$ (resp. $\Graphs_2$) and $\BVGraphs_\Sigma$ (resp $\BVGraphs$) is that we now allow tadpoles (edges connecting a vertex to itself) at external vertices but graphs with tadpoles at internal vertices are considered to be $0$. 

This can be done by considering the subalgebra  $^*\BVGraphs_\Sigma\subset ^*\Graphs_\Sigma$ of graphs with no tadpoles on internal vertices or dually defining $\BVGraphs_\Sigma$ as a quotient of $\Graphs_\Sigma$. A precise definition of $\BVGraphs_2$ can be found in \cite{Ca}. 

The non-twisted analog of $^*\BVGraphs(n)$ is $^*\BVGra(n)$, the symmetric algebra on symbols $s^{ij}=s^{ji}$, $1\leq i,j\leq n$. 
One can also consider the non-twisted analog $^*\BVGra_\Sigma$, but notice that this is just the same space as $^*\Gra_\Sigma$ as tadpoles are not forbidden in $^*\Gra_\Sigma$ and there are no internal vertices upon which we can impose any condition. \\

Let $\phi \in \Omega_{triv}^1(\FFM_\Sigma(1))$ be a global angular form of the $S^1$-bundle  $\pi \colon \FFM_\Sigma(1) = \mathrm{SO}(\Sigma) \to \Sigma$. Such form satisfies $d\phi = \pi^*(e)$, where $e\in \Omega_{triv}^2(\Sigma)$ is the Euler class of the circle bundle.

Let $1\leq i \leq n$. We denote by $\phi_{ii}\in \Omega_{triv}^1(\FFM_\Sigma(n))$ the form $\pi_i^*(\phi)$, where $\pi_i \colon \FFM_\Sigma(n) \to \FFM_\Sigma(1)$ is the map that remembers only the point labeled by $1$

We define a map $^*\BVGra_\Sigma(n) \to \Omega_{triv}(\FFM_\Sigma(n))$ as a morphism of algebras sending $s^{ij}$  to  $\phi_{ij}$, where if $i\ne j$, $\phi_{ij}$ is the form constructed in section \ref{sec:FM_M} and sends $[\omega]^j\in ^*\BVGra_\Sigma$ to $p_j^*(\iota([\omega]))$, where $p_j\colon \FFM(n)  \to M$ is the map that remembers only the point labeled by $j$.

Similarly one defines a map  $^*\BVGra_2(n) \to \Omega_{triv}(\FFM_2(n)) = \Omega_{triv}(\FM_2(n) \times \SO^{\times n})$ as a morphism of algebras sending a tadpole at the vertex $i$ to the volume form of the $i$-th $\SO$.

\begin{lemma}\label{lem:map of cooperadic comodules}
This defines a morphism of  cooperadic comodules $^*\BVGra_\Sigma \aor {}^*\BVGra_2 \to \Omega_{triv}(\FFM_\Sigma) \aor \Omega_{triv}(\FFM_2)$.
\end{lemma}
\begin{proof}
Regarding the compatibility with the differentials, the only case not covered in Lemma \ref{Lemma:map of comodules} is $\phi_{ii}$, but this follows from the fact that the Euler form can be expressed as $\sum_{i,j} g^{ij}e_i \wedge e_j$.

For the compatibility with the cooperadic comodule structure it remains to check it for the elements $s^{ii}\in \BVGraphs_\Sigma(n)$. For simplicity of notation we consider the element $s^{11}\in \BVGraphs_\Sigma(1)$ which is sent to $\phi_{11}\in \Omega_{PA}^1(\FFM_\Sigma(1))$ whose coaction gives $\phi_{11}\otimes 1 + 1 \otimes \text{vol}_{S^1} \in  \Omega_{PA}(\FFM_\Sigma(1))\otimes \Omega_{PA}(\FFM_2(1))$.

On the other hand, the coaction on $s^{11}\in \BVGraphs_\Sigma(1)$ gives us $s^{11} \otimes 1 + 1 \otimes s^{11} \in \BVGraphs_\Sigma(1) \otimes \BVGraphs_2(1)$, from where the compatibility follows.
\end{proof}

Similarly to what was done in Section \ref{sec: twisting}, one can prove the following Proposition

\begin{prop}
There is a morphism of cooperadic modules $^*\BVGraphs_\Sigma \aor ^*\BVGraphs_2 \to \Omega_{PA}(\FFM_\Sigma) \aor \Omega_{PA}(\FFM_2)$ extending the morphism from Lemma \ref{lem:map of cooperadic comodules}.
\end{prop}

The only difference relatively to the non-framed case is that the map $ ^*\BVGraphs_\Sigma(n) \to \Omega_{PA}(\FFM_\Sigma(n))$ evaluated at a graph $\Gamma\in\BVGraphs_\Sigma$ with $k$ internal vertices is given by an integral over the fiber of $\FFM_\Sigma(n,k)\to \FFM_\Sigma(n)$, where the space $\FFM_\Sigma(n,k)$ is the (compactification of the) configuration space of $n$ framed points and $k$ unframed points corresponding respectively to the external vertices and the internal vertices of $\Gamma$. 

A similar procedure is done for the map $^*\BVGraphs_2(n) \to \Omega_{PA}(\FFM_2(n))$.

The goal of this section is to prove the following Theorem.

\begin{thm}\label{thm:main framed}
The map $^*\BVGraphs_\Sigma \aor ^*\BVGraphs_2 \to \Omega_{PA}(\FFM_\Sigma) \aor \Omega_{PA}(\FFM_2)$ is a quasi-isomorphism of Hopf cooperadic comodules.
\end{thm}

\begin{prop}
The map $^*\BVGraphs_2\to \Omega_{PA}(\FFM_2)$ is a quasi-isomorphism. 
\end{prop}

\begin{proof}
On the one hand we have $$H^\bullet(\FFM_2(n)) = H^\bullet(\FM_2(n) \times \SO^{\times n}) = H^\bullet(\FM_2(n)) \otimes H^\bullet(\SO)^{\otimes n} = H^\bullet(\FM_2(n)) \otimes (\mathbb R \oplus \mathbb R[- 1])^{\otimes n}$$ by the K\"unneth formula.
On the other hand, notice that as dg symmetric sequences $\BVGraphs_2= \Graphs_2\circ (\mathbb R[-1] \oplus \mathbb R)$, therefore $$H(^*\BVGraphs_2(n)) = H\left(^*\Graphs_2(n) \otimes (\mathbb R \oplus \mathbb R[- 1])^{\otimes n}\right) =   H(^*\Graphs_2(n)) \otimes (\mathbb R \oplus \mathbb R[-1])^{\otimes n}.$$

Since a tadpole at the vertex labeled by $i$ is sent to the volume form of $i$-th $\SO$, which is the generator of $H^1(\SO)$, we have that at the cohomology level the map $$H(^*\BVGraphs_2) = H(^*\Graphs_2(n)) \otimes (\mathbb R \oplus \mathbb R[ -1])^{\otimes n} \to H^\bullet(\FFM_2(n)) =H^\bullet(\FM_2(n)) \otimes (\mathbb R \oplus \mathbb R[- 1])^{\otimes n} $$
is just the map $f_* \otimes \id$, where $f \colon ^*\Graphs_2 \to \Omega_{PA}(\FM_2)$ is the quasi-isomorphism from Definition/Proposition \ref{defprop:Graphs_D}, from where the result follows.
 \end{proof}

\subsection{Proof of Theorem \ref{thm:main framed}}
Let $n,k\geq 0$ and let us consider an auxiliary differential graded vector space $G(n,k)$ that is the subcomplex of $^*\BVGraphs_\Sigma(n+k)$ in which the points labeled $n+1,\dots n+k$ cannot have tadpoles. 
 This should be seen as the algebraic analog of the space $\FFM_\Sigma(n,k)$, the compactification of the configuration space of $n$ framed points and $k$ unframed points in $\Sigma$. 

The map $^*\BVGraphs_\Sigma(n+k) \to \Omega_{PA}(\FFM_\Sigma(n+k))$ restricts naturally to a map $G(n,k) \to \Omega_{PA}(\FFM_\Sigma(n,k))$. We will show that this map is a quasi-isomorphism, thus proving Theorem \ref{thm:main framed} which corresponds to the cases with $k=0$. The proof will be done by induction on $n$. The case $n=0$ was already proven in Theorem \ref{thm:Main non-parallelized}.

\subsubsection{A long exact sequence of graphs}
Let us prove the following auxiliary result.

\begin{prop}
There is a long exact sequence of graded vector spaces

$$\dots\longrightarrow H^d(G(n+1,k-1))\stackrel{f}{\longrightarrow}H^{d-1}(G(n,k))\stackrel{ \wedge e }{\longrightarrow}H^{d+1}(G(n,k))\stackrel{i_*}{\longrightarrow}H^{d+1}(G(n+1,k-1))\longrightarrow \dots,$$
where the map $i_*$ induced by the inclusion of $G(n,k)$ in $G(n+1,k-1)$.
\end{prop}

\begin{proof}
Let us clarify the undescribed maps. The map $f$ removes a tadpole on the vertex labeled by $n+1$ if there exists one, otherwise it sends a graph to zero. The map $\wedge e$ decorates the vertex $n+1$ with the ``Euler form".

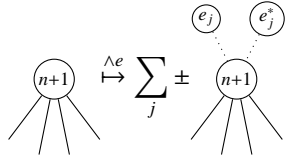
\begin{figure}[h]
\begin{tikzpicture}

\node (1) at (0,0) [ext] {$\scriptstyle{n+1}$};

\draw (-0.6,-0.8)--(1);
\draw (-0.2,-1)--(1);
\draw (0.2,-1)--(1);
\draw (0.6,-0.8)--(1);

\node at (1.2,-0.1) {$\stackrel{\wedge e}{\mapsto} \displaystyle\sum_j \pm$};

\node (2) at (2.4,0) [ext] {$\scriptstyle{n+1}$};

\draw (-0.6,-0.8)+(2.4,0)--(2);
\draw (-0.2,-1)+(2.4,0)--(2);
\draw (0.2,-1)+(2.4,0)--(2);
\draw (0.6,-0.8)+(2.4,0)--(2);

\node (e1) at (2,0.8) [ext] {$\scriptstyle{e_j}$};
\node (e2) at (2.8,0.8) [ext] {$\scriptstyle{e_j^*}$};

\draw[dotted] (e1)--(2);
\draw[dotted] (e2)--(2);
\end{tikzpicture}

\end{figure}

 It is not clear that these maps are well defined at the cohomology level, but this will become clear by the construction of the sequence.

Let us consider the following decomposition of $G(n+1,k-1)$:

$$
\begin{tikzpicture}
\node at (0,0) {$G(n+1,k-1)=$};
\node (a) at (2.1,0) {$G(n,k)[-1]$};
\node at (3.23,0) {$\oplus$};
\node (b) at (4,0) {$G(n,k)$};

\draw (a) edge[loop above] (a);
\draw (b) edge[loop above] (b);

\draw[->] (a) edge[bend left] (b);

\node at (2.1,0.8) {$\scriptstyle{d_0}$};
\node at (4,0.8) {$\scriptstyle{d_0}$};

\node at (3.05,0.53) {$\scriptstyle{d_1}$};
\end{tikzpicture},
$$
where the first summand corresponds to graphs in which the vertex labeled by $n+1$ has a tadpole and the second summand corresponds to graphs in which the vertex labeled by $n+1$ does not have a tadpole. The differential splits into two terms $d_0$ and $d_1$, as in the picture. Let us consider a two-level filtration on the number of tadpoles at the vertex $n+1$. On the zeroth page of the spectral sequence the differential is $d_0$, which acts as the ordinary differential of $G(n,k)$. 

The differential on the second page is induced by $d_1$ and is the map that was denoted by $\wedge e$,   $$\wedge e \colon H^\bullet(G(n,k)[-1]) = H^{\bullet-1}(G(n,k)) \to H^{\bullet +1}( G(n,k)).$$

The spectral sequence converges at the second page since we considered a two-level filtration, therefore 
$$H^\bullet(G(n+1,k-1))= \ker(\wedge e) \oplus \coker(\wedge e).$$

The map $f$ is defined to be the composition 
$H^\bullet(G(n+1,k-1)) \twoheadrightarrow \ker(\wedge e) \hookrightarrow H^{\bullet-1}(G(n,k)).$
It is then clear that $\text{Im} (f) = \ker(\wedge e)$, which gives us exactness at $H^{d-1}(G(n,k))$. 

The map $i_*$ is given by the composition $H^\bullet(G(n,k)) \twoheadrightarrow \coker(\wedge e) \hookrightarrow H^{\bullet-1}(G(n+1,k-1))$, therefore its image coincides with the kernel of $f$, which shows exactness at $H^{d+1}(G(n+1,k-1))$.

Since $i_*$ is the projection to the cokernel of $\wedge e$, its kernel is precisely the image of $\wedge e$, which shows the remaining exactness.
\end{proof}

\subsubsection{The Gysin sequence}

The map $\pi \colon \FFM_\Sigma(n+1,k-1) \to \FFM_\Sigma(n,k)$ that forgets the frame at the point $n+1$ is a circle bundle. We denote by $e\in \Omega_{PA}(\FFM_\Sigma(n,k)$ the Euler form of the circle bundle. The Gysin sequence of this circle bundle is the following long exact sequence:

\begin{equation}\resizebox{13cm}{!}{$
 H^d(\FFM_\Sigma(n+1,k-1))\stackrel{\int_\pi}{\longrightarrow}H^{d-1}(\FFM_\Sigma(n,k))\stackrel{ \wedge e }{\longrightarrow}H^{d+1}(\FFM_\Sigma(n,k))\stackrel{\pi^*}{\longrightarrow}H^{d+1}(\FFM_\Sigma(n+1,k-1))\longrightarrow \dots$}
\end{equation}

Using the maps $G(a,b)\to \Omega_{PA}(\FFM_\Sigma(a,b))$ we obtain the following morphism of exact sequences:

\begin{tikzpicture}[baseline= (a).base]
\node[scale=.85] (a) at (0,0) {\begin{tikzcd}
H^d(\FFM_\Sigma(n+1,k-1)) \arrow{r}{\int_\pi}& H^{d-1}(\FFM_\Sigma(n,k))\arrow{r}{\wedge e}&H^{d+1}(\FFM_\Sigma(n,k))\arrow{r}{\pi^*}&H^{d+1}(\FFM_\Sigma(n+1,k-1))\\
H^d(G(n+1,k-1))\arrow{u}\arrow{r}{f}&H^{d-1}(G(n,k))\arrow{u}\arrow{r}{\wedge e}&H^{d+1}(G(n,k))\arrow{u}\arrow{r}{i_*}&H^{d+1}(G(n+1,k-1))\arrow{u}.\\
\end{tikzcd}};
\end{tikzpicture}

Since by induction $G(n,k)\to \Omega_{PA}(\FFM_\Sigma(n,k))$ is a quasi-isomorphism, the Five Lemma implies that $G(n+1,k-1)\to \Omega_{PA}(\FFM_\Sigma(n+1,k-1))$ is a quasi-isomorphism as well, thus concluding the proof of Theorem \ref{thm:main framed}.

\appendix

\section{Comparison to the Lambrechts--Stanley model through cyclic $\Com_\infty$ algebras}
\label{sec:LS model and ours}
\newcommand{\K}{\R}
\newcommand{\stGra}{{}^*\Gra}
\newcommand{\beq}[1]{\begin{equation}\label{#1}}
\newcommand{\eeq}{\end{equation}}

In this appendix we show how to obtain from the $\pdGraphs_{M}$ model a proof that the Lambrechts--Stanley algebra is a dgca model for the $\FM_M$ (Conjecture \ref{conj:LS}). 

\begin{defi}[\cite{LS2}]
 A Poincar\'e duality algebra of dimension $D$ is a non-negatively graded connected dgca $A$ together with a linear map
\[
 \epsilon : A^D \to \K
\]
such that $\epsilon\circ d=0$ and such that the bilinear maps
\begin{gather*}
A\otimes A\to \K[-D] \\
a\otimes b\mapsto \epsilon(a,b) 
\end{gather*}
are non-degenerate.
\end{defi}

Note that by the connectivity assumption necessarily $A^D=\K$ and hence $\epsilon$ is unique up to scale, if it exists.
Note that a Poincar\'e duality algebra is a particular case of a cyclic $\Com_\infty$-algebra.

A Poincar\'e duality model for a manifold $M$ is a Poincar\'e duality algebra weakly equivalent (as a dgca) to $\Omega(M)$.
It is shown in \cite{LS2} that such a Poincar\'e duality model always exists for simply connected compact orientable manifolds.

Lambrechts and Stanley furthermore define the following family of dgcas from a Poincar\'e duality algebra $A$, generalizing earlier work by Kriz \cite{Kr} and Totaro \cite{To}.
Consider the algebra
\[
A^{\otimes n}[\omega_{ij}; 1\leq i\neq j\leq n]. 
\]
For $a \in A$ let $p_j^*(a)$ be the element $1\otimes \cdots \otimes a \otimes \cdots \otimes 1$, with $a$ in the $j$-th slot.
Then one imposes on the above algebra the following relations
\begin{enumerate}
 \item $\omega_{ij}=(-1)^D\omega_{ji}$
\item $\omega_{ij}^2=0$
\item $\omega_{ij}\omega_{ik}+\omega_{jk}\omega_{ji}+\omega_{ki}\omega_{kj}=0$ for distinct $i,j,k$
\item $(p_i^*(a) - p_j^*(a))\omega_{ij}=0$.
\end{enumerate}

Let us define for $A$ a Poincar\'e duality algebra as above the diagonal $\Delta\in A\otimes A$ to be the inverse of the non-degenerate bilinear pairing.
Let us further denote by $\Delta_{ij}$ the corresponding element in $A^{\otimes n}$, the two ``non-trivial'' factors of $A$ situated in positions $i$ and $j$.
Then one defines 
\[
 (A^{\otimes n}[\omega_{ij}; 1\leq i\neq j\leq n]/\sim, d_A+\nabla)
\]
where the differential $d_A$ is that induced by the differential on $A$ and $\nabla$ is defined as 
\[
 \nabla \omega_{ij} = \Delta_{ij}.
\]
One readily checks that the ideal generated by the relation is closed under this differential.
Furthermore, if the Euler class of $A$, i.e., the image $\Delta$ under the multiplication, vanishes, then the $F(A,-)$ naturally assemble into a right $\Pois_D^*$ cooperadic comodule.

Lambrechts and Stanley \cite{LS} show that for $A$ a Poincar\'e duality model for $M$ we have that $H(F(A,n))=H(\FM_M(n))$, and furthermore raise the following conjecture.

\begin{conj}[\cite{LS}]
\label{conj:LS}
If $A$ is a Poincar\'e duality model for the simply connected compact orientable manifold $M$ then $F(A,n)$ is a dgca model for $\Conf(M,n)$.
\end{conj}

A proof of (a slightly weaker form of) this statement is given in \cite{I}, using methods similar to ours.
While in this paper we work with cyclic $\Com_{\infty}$ structures on $H(M)$, rather than Poincar\'e duality models to capture the real homotopy type ``with Poincar\'e duality'' for $M$, one can still deduce the conjecture of Lambrechts and Stanley from our methods, at least in the case that the dimension of $M$ is at least 4. (The case $M=S^2$ also follows from the computation in Appendix \ref{app:sphere}, leaving only the case $M=S^3$.)
Let us sketch this reduction.

First let $V$ be a finite dimensional differential non-negatively graded vector space with the subspace of degree 0 elements $V_0=\K$ and a non-degenerate symmetric bilinear pairing of degree $D$
\[
 V\otimes V\to \K[-D].
\]
We denote by $\Delta\in V\otimes V$ the corresponding dual degree $D$ element (the "diagonal") as above.
Then we may define a graph complex (and dg Lie algebra) $\GC_V$ akin to $\GC_M$ above, just replacing each occurrence of $H^*(M)$ by $V$ and with an additional piece of the differential coming from $d_V$. Concretely, this means that vertices in graphs of $\GC_V$ may be decorated by copies of $\bar V^*$.
Furthermore, suppose a cyclic $\Com_\infty$ structure is given on $V$, for the above bilinear form. We may see this structure as a Maurer-Cartan element $Z\in \GC_V$, all of whose coefficient in front of non-tree graphs vanish. We may furthermore use it to define a Graph complex $\stG_V$ analogously to $\stG_M$ above, replacing each occurrence of $H(M)$ by $V$, and using the given $Z$ in place of the partition function.

Next, fix representatives of the cohomology of $V$ by providing a map
\beq{equ:HVtoV}
 H(V)\hookrightarrow V.
\eeq
The pairing on $V$ induces a pairing on $H(V)$, independent of the representatives chosen. We denote the corresponding diagonal by $\Delta_{H}\in H(V)\otimes H(V)$.
Via the chosen embedding we may as well consider $\Delta_{H}$ as an element in $V\otimes V$, in which case it becomes cohomologous to $\Delta$. We may hence choose $\eta\in A\otimes A$ (of the same symmetry under exchange of the two $A$'s as $\Delta$) such that 
\beq{equ:Deltaeta}
 \Delta_H = \Delta - d_V \eta.
\eeq

We may then define a natural map of dg cooperadic comodules
\beq{equ:stGraVH}
 \stGra_{H(V)} \to \stGra_V 
\eeq
by sending the decorations in $H(V)$ to $V$ using our map \eqref{equ:HVtoV}, and by sending an edge between vertices $i$ and $j$ to the same edge, minus the element $\eta$, considered as decoration at vertices $i$ and $j$.
In pictures
\[
 \begin{tikzpicture}[baseline=-.65ex]
  \node[ext] (v) at (0,0) {};
\node[ext] (w) at (0.5,0) {};
\draw (v) edge (w);
 \end{tikzpicture}
\mapsto
\begin{tikzpicture}[baseline=-.65ex]
  \node[ext] (v) at (0,0) {};
\node[ext] (w) at (0.5,0) {};
\draw (v) edge (w);
 \end{tikzpicture}
-
\begin{tikzpicture}[baseline=-.65ex]
  \node[ext] (v) at (0,0) {};
\node[ext] (w) at (0.5,0) {};
\node (x) at (.25,.5) {$\eta$};
\draw[dotted] (x) edge (v) edge (w);
 \end{tikzpicture}
\]
Equation \eqref{equ:Deltaeta} implies that the map \eqref{equ:stGraVH} is indeed compatible with the differentials.

Following the construction of $\GC_V$, this map \eqref{equ:stGraVH} induces an $L_\infty$-morphism of dg Lie algebras
\[
\GC_V\to \GC_{H(V)},
\]
and we can hence transfer the Maurer-Cartan element $Z\in \GC_V$ inducing the cyclic $\Com_\infty$-structure on $V$ to a Maurer-Cartan element $Z_H\in \GC_H$.
(The MC element $Z_H$ is still supported on trees, and encodes the cyclic $\Com_\infty$ structure on $H(V)$ induced by homotopy transfer.)
Furthermore, we obtain from \eqref{equ:stGraVH} a map 
\[
 \stG_{H(V)} \to \stG_V,
\]
that one can check to be a quasi-isomorphism by an easy spectral sequence argument.

In particular, let us take for $V$ a Poincar\'e duality model for the simply connected manifold $M$. Then if the dimension $D$ of $M$ is at least $4$ the Maurer-Cartan element $Z_H$ is gauge equivalent to the partition function $Z_M$ constructed above. This is because by degree reasons there cannot be loop order $\geq 1$ contributions to this partition function, and the tree part of $Z_M$ encodes the real homotopy type of $M$ (in the form of a cyclic $\Com_\infty$ structure on $H(V)=H(M)$), and hence must be gauge equivalent to $Z_H$, whch also encodes the real homotopy type by construction.
Hence we can conclude that $\stG_V$ is quasi-isomorphic to $\stG_M$ and is hence a dgca model for $\FM_M$, with the partition function concentrated on trees with one vertex.
Furthermore, in this case we have a direct map 
\beq{equ:stgVF}
 \stG_V \to F(V,-)
\eeq
to the Lambrechts-Stanley algebra, by sending all graphs with internal vertices to zero, and imposing the defining relations.
Again, by a spectral sequence sequence argument the map \eqref{equ:stgVF} can be seen to be a quasi-isomorphism.
Furthermore, it is evidently compatible with the right $\Pois_D^*$ cooperadic comodule structures, in the case the Euler class vanishes.
This shows that $F(V,-)$ is quasi-isomorphic to $\stG_M$, i.e., to a dgca model for $\FM_n$.
Hence the conjecture \ref{conj:LS} follows, in dimension $D\geq 4$.

\section{Example computation: The partition function of the 2-sphere}\label{app:sphere}
As an illustration, let us show that the partition function of the two-sphere is essentially trivial.
We cover $S^2$ by two coordinate charts $\C$ via stereographic projection as usual.
The coordinate transformation relating the two charts is then 
\begin{gather*}
\Phi\colon  \C\setminus \{0\} \to \C \\
\quad  z\mapsto \frac 1 z
\end{gather*}

We take a basis $1\in H^0(S^2)$, $\omega\in H^2(S^2)$ of the cohomology, with $\int\omega =1$.
Take as a representative for $\omega$ any compactly supported top form of volume $1$, which we also denote by $\omega$. In fact, to abuse the notation further, denote by $\omega\in \Omega^2(\C)$ also the coordinate expression in one of our charts. To achieve somewhat nicer formulas later, let us also assume that this $\omega$ is supported away from the origin and that 
\beq{equ:Phiomega}
\Phi^*\omega = \omega.
\eeq

Let $\phi_0$ be the propagator on $\C$, i.e., 
\[
\phi_0(z,w) = \frac 1 {2\pi} \Im d\log(z-w).
\]
Note that 
\beq{equ:Phiphi}
\phi_0(\frac 1 z,\frac 1 w) = \frac 1 {2\pi} \Im d\log(\frac{w-z}{wz}) = \phi_0(z,w)-\phi_0(z,0)-\phi_0(w,0) .
\eeq
Then we will take as propagator of the sphere\footnote{In Proposition \ref{prop:angular form} the propagator has been denoted $\phi_{12}$. Here we choose to drop the subscript $12$ for brevity.}
\[
\phi(z,w) = \phi_0(z,w) - \int_u \phi_0(z,u) \omega(u) -  \int_u \phi_0(w,u) \omega(u).
\]
Let us first verify that this two form extends from our coordinate chart to $\FM_2(S^2)$. To this end, apply the coordinate transformation $\Phi$ and compute:
\begin{align*}
\phi(\frac 1 z,\frac 1 w) &= 
\phi_0(\frac 1 z,\frac 1 w)
 - \int_u \phi_0(\frac 1 z,u) \omega(u) -  \int_u \phi_0(\frac 1 w,u) \omega(u).
\end{align*}
Changing the integration variable from $u$ to $\frac 1 u$, using \eqref{equ:Phiomega} and applying \eqref{equ:Phiphi} three times we obtain:
\begin{align*}
\phi(\frac 1 z,\frac 1 w) &= 
\phi_0(z,w)-\phi_0(z,0)-\phi_0(w,0)
 - \int_u (\phi_0(z,u)-\phi_0(z,0)-\phi_0(u,0)) \omega(u) 
 \\&\quad
 -  \int_u (\phi_0(w,u)-\phi_0(w,0)-\phi_0(w,0))) \omega(u)
 \\
 &=
\phi(z,w)
-\phi_0(z,0)-\phi_0(w,0)
 + \phi_0(z,0) \int_u \omega(u) + \phi_0(w,0)\int_u \omega(u)
\\
&=\phi(z,w).
\end{align*}
Hence the propagator has the same form in the other coordinate chart, and in particular it has no singularity at the coordinate origin, and hence readily extends to $\FM_2(S^2)$.

Furthermore one checks the following properties:
\begin{itemize}
\item Clearly $\phi(z,w)=\phi(w,z)$.
\item By Stokes' Theorem
\[
d\phi(z,w)= \omega(z) + \omega(w)
\]
as required.
\item By degree reasons 
\[
\int_v\phi(z,v)=0.
\]
Furthermore
\begin{align*}
\int_v\phi(z,v)\omega(v)&=
\int_v\phi_0(z,v)\omega(v)
-
\int_v\int_u \phi_0(z,u)\omega(u)\omega(v)
-
\int_v\int_u \phi_0(v,u)\omega(u)\omega(v)
\\
&=
\int_v\phi_0(z,v)\omega(v)
-
\int_u \phi_0(z,u)\omega(u)
-0
\\&
=0.
\end{align*}
Here the third term on the right-hand side vanishes by degree reasons. (One integrates a 5-form over a 4-dimensional space.)
\item We have 
\begin{align*}
\int_v\phi(z,v)\phi(u,w) &= 
\int_v\phi_0(z,v)\phi_0(v,w)
-
\int_v \int_{u_1}\phi_0(z,u_1)\omega(u_1)\phi_0(v,w)
-
\int_v \int_{u_2}\phi_0(v,w)\phi_0(w,u_2)\omega(u_2)
\\
&
-
\int_v \int_{u_1}\phi_0(v,u_1)\omega(u_1)\phi_0(v,w)
-
\int_v \int_{u_2}\phi_0(v,w)\phi_0(v,u_2)\omega(u_2)
\\
&
+
\int_v \int_{u_1} \int_{u_2}\phi_0(z,u_1)\omega(u_1)\phi_0(w,u_2)\omega(u_2)
+
\int_v \int_{u_1} \int_{u_2}\phi_0(v,u_1)\omega(u_1)\phi_0(w,u_2)\omega(u_2)
\\
&
+
\int_v \int_{u_1} \int_{u_2}\phi_0(z,u_1)\omega(u_1)\phi_0(v,u_2)\omega(u_2)
+
\int_v \int_{u_1} \int_{u_2}\phi_0(v,u_1)\omega(u_1)\phi_0(v,u_2)\omega(u_2).
\end{align*}
The first term on the right-hand side vanishes by a standard vanishing Lemma of Kontsevich.
For the same reason vanish the fourth, fifth, and last terms.
The remaining terms terms vanish by degree reasons: There forms with $v$-dependence are of degree $\leq 1$.
Hence we conclude that the whole expression is zero, and graph weights computed using our propagator will be zero for graphs with bivalent vertices.   
\item Identify the pullback of $\partial\FM_2(S^2)$ to our coordinate chart with $\C\times S^1$, and fix the standard coordinate $\varphi$ on the $S^1$-factor.
Then restricting $\phi$ to the boundary $\partial\FM_2(S^2)$, (i.e., we take the limit $w\to z$ in our coordinate chart) we obtain the form
\[
\frac 1 {2\pi} d\varphi + \eta(z),
\]
where 
\[
\eta = -2\int_u\phi_0(z,u)\omega(u) 
\] 
depends only on $z$ but not on $\varphi$ as desired.
\end{itemize}

\subsection{Vanishing of integrals}
\begin{prop}
Using the propagator $\phi$ and the top form $\omega$ as above, the partition function becomes 
\beq{equ:Z}
z_{S^2} = 
\begin{tikzpicture}
\node[int,label={$\omega$}] (v) {};
\end{tikzpicture}.
\eeq
In other words, the weights of all graphs with more than one vertex vanish.
\end{prop}

\begin{proof}
By the properties above, all graphs vanish if either some vertex has valence 2 or some vertex has more than one decoration by $\omega$ or some vertex has valence one, and there is one incident edge.
The only connected graph with a vertex of valence one is the one appearing in \eqref{equ:Z}.
All other graphs with potentially non-vanishing weight must hence be of the following kind: 
\begin{enumerate}
\item There are $\geq 2$ edges incident to any vertex, and at most one decoration $\omega$.
\item If there are exactly 2 edges incident on some vertex, it must come with a decoration $\omega$.
\end{enumerate}

From an admissible graph $\Gamma$, we can build another linear combination of admissible graphs $\Gamma_0$ by formally replacing each edge by the linear combination 
\[
\begin{tikzpicture}
\node[ext](v) at(0,0){};
\node[ext](w) at(.6,0){};
\draw (v) edge (w);
\end{tikzpicture}
\mapsto 
\begin{tikzpicture}
\node[ext](v) at(0,0){};
\node[ext](w) at(.6,0){};
\draw (v) edge (w);
\end{tikzpicture}
\,-\,
\begin{tikzpicture}
\node[ext](v) at(0,0){};
\node[ext](w) at(.6,0){};
\node[int,label={$\omega$}] (x) at (0.3,0){};
\draw (v) edge (x);
\end{tikzpicture}
\,-\,
\begin{tikzpicture}
\node[ext](v) at(0,0){};
\node[ext](w) at(.6,0){};
\node[int,label={$\omega$}] (x) at (0.3,0){};
\draw (w) edge (x);
\end{tikzpicture}
\] 
Clearly, we have that 
\[
\int_{\FM_{d}(|V\Gamma|)} \omega_\Gamma 
= 
\int_{\FM_{d}(|V\Gamma_0|)} \omega^0_{\Gamma_0}
\]
where now the weight form $\omega^0_{...}$ is defined just like $\omega_{...}$ above, but using the Euclidean propagator $\phi_0$ instead of $\phi$.

It hence suffices to show that for each admissible graph $\Gamma$ with more than one vertex we have 
\[
\int_{\FM_{d}(|V\Gamma|)} \omega^0_\Gamma \stackrel{?}{=}0.
\] 
We may assume that the vertices are numbered such that the vertices decorated by $\omega$ have indices $1,\dots,k$, for some $k\geq 0$.
Then the above integral factorises as
\[
\int_{\FM_{d}(|V\Gamma|)} \omega^0_\Gamma
= 
\int_{\FM_{d}(k)} \omega(x_1)\omega(x_2)\cdots \omega(x_k)
\underbrace{ \int_{\FM_{d}(|V\Gamma|-k)}  \omega^0_\Gamma }_{=:f(x_1,\dots,x_k)}.
\]
Note that here $f(x_1,\dots,x_k)$ is a function associated to a graph with decorations $\omega$.
(There can be no form piece in $f(\dots)$, because the remainder of the integrand is already a top form.)
Hence by the Kontsevich Vanishing Lemma \cite[Lemma 6.4]{K1} $f(x_1,\dots,x_k)\equiv 0$.
Hence the desired vanishing result follows.
\end{proof}

\section{Pushforward of PA forms}\label{app:PA}

Given an SA bundle $p\colon M\to N$ of rank $l$, the pushforward map of ``integration along the fiber'' defined in  \cite{HLTV} is a map $p_*\colon \Omega^\bullet_{min}(M)\to \Omega_{\text{PA}}^{\bullet-l}(N)$. This map is only defined on minimal forms as the natural extension to the full algebra of PA forms is not well defined due to the failure of the relevant semi-algebraic chain to be continuous (see the discussion on \cite[Section 9]{HLTV}).\footnote{We note that in the original sketch of the construction of PA forms by Kontsevich and Soibelman \cite{KS}, the pushforward was (claimed to be) defined for all PA forms, for a slightly laxer definition of PA forms compared to \cite{HLTV}.}

For our purposes we need to consider pushforwards of the propagator $\phi_{12}\in \Omega_{\text{PA}}(\FM_M(2))$ constructed in Proposition \ref{prop:angular form}. Since we cannot construct the propagator in such a way that $\phi_{12}\in \Omega_{min}(\FM_M(2))$, in this section we consider a different space of forms, $\Omega_{triv}$, such that $\Omega_{\text{PA}}\supset \Omega_{triv}\supset\Omega_{min}$ to which the pushforward map can be extended and still satisfies the Stokes theorem.

Recall that for $F$ a compact oriented semi-algebraic manifold 
and $M$ a semi-algebraic manifold, the \textit{constant continuous chain} $\hat F \in C^{str}(M\times F\to M)$ is defined by $\hat F(x)=\llbracket\{x\}\times F\rrbracket$.
\begin{defi}
	Let $M$ be a semi-algebraic manifold. The space $\Omega_{triv}(M)$ of \textit{trivial forms} is the subvector space of $\Omega_{\text{PA}}(M)$ spanned by forms of the type $\fint_{\hat F}\mu$, where $\mu\in \Omega_{min}(M\times F)$ and $\hat F$ is a constant continuous chain.
\end{defi}

\begin{lemma}
The subspace $\Omega_{triv}(M)\subset \Omega_{\text{PA}}(M)$ is a dg commutative subalgebra.
\end{lemma}
\begin{proof}
	$\Omega_{triv}(M)\subset \Omega_{\text{PA}}(M)$ is closed under the differential by the fiberwise Stokes' Theorem \cite[Proposition 8.12]{HLTV} and since the fiberwise boundary of a trivial bundle is again a trivial bundle.
	Furthermore, the subspace $\Omega_{triv}(M)$ is closed under addition and the commutative product on $\Omega_{\text{PA}}(M)$ because the union and product of trivial bundles is again trivial, see the construction of these operations in \cite[section 5]{HLTV}.
\end{proof}

Let us consider a strongly continuous chain $\Phi\in C_l^{str}(E\stackrel{f}{\to} B)$ 
along a semi-algebraic map $f\colon E\to B$. Let $E\times F$ be the trivial bundle over $E$ with fiber $F$, a compact oriented semi-algebraic $k$ manifold.

\begin{prop}
	Under the previous conditions, there is a strongly continuous chain $$\Phi\ltimes \hat F \in C_{k+l}^{str}(E\times F \stackrel{f\circ pr_2}{\longrightarrow}B)$$ defined by $(\Phi\ltimes \hat F)(b) \coloneqq \Phi(b) \times F$.
\end{prop}

\begin{proof}
	If we consider the family $\{(S_\alpha,F_\alpha,g_\alpha)_{\alpha\in I}\}$ that trivializes the continuous chain $\Phi$, it is easy to see that $\{(S_\alpha,F_\alpha\times F,g_\alpha\times \id_F)_{\alpha\in I}\}$ trivializes $\Phi\ltimes \hat F$ since by hypothesis the following two squares commute.
	$$
	\begin{tikzcd}
	\bar{S}_\alpha \times F_\alpha \times F \arrow{r}{g_\alpha\times \id_F} \arrow{d}&  E\times F\arrow{d}{pr_2}\\		
	\bar{S}_\alpha \times F_\alpha \arrow{r}{g_\alpha} \arrow{d}{f}&  E \arrow{d}\\
	\bar{S}_\alpha \arrow{r} &B		
	\end{tikzcd}$$
\end{proof}

\begin{cor}
	Let $p\colon Y\to X$ be an oriented SA bundle and $\Phi\in C^{str}_l(Y\to X)$ the associated strongly continuous chain.
	There is a well defined map $p_*\colon \Omega_{triv}^\bullet (Y)\to \Omega^{\bullet-l}_{\text{PA}}(X)$ extending the one on minimal forms, given by $p_*(\omega) = \fint_{\Phi\ltimes \hat F}\omega$.
\end{cor}

\begin{rem}
	Recall that the proof of the fiberwise Stokes theorem relies essentially on the fact that for $\gamma\in C_k(X)$ and $\Psi\in C^{str}_l(Y\to X)$, we have $\partial (\gamma \ltimes \Psi) = \partial \gamma \ltimes \Psi + (-1)^{\deg \gamma} \gamma \ltimes \partial \Psi$.	With the same proof of \cite[Proposition 5.17]{HLTV} we see that this formula is still valid if we take $\Psi$ and $\gamma$ to be $\Phi$ and $\hat F$ as above and therefore Stokes theorem is also valid for pushforwards of trivial forms.
\end{rem}

We prove now the Poincar\'e lemma for the sheaf of complexes $\Omega_{triv}$.

\begin{prop}\label{prop:Poincare}
	If $U$ is a contractible semi-algebraic set, then $H(\Omega_{triv}(U))$ is one dimensional and concentrated in degree zero.
\end{prop}

\begin{proof}
	Let $h\colon [0,1]\times U \to U$ be a contraction of $U$, such that $h(1,x)=x$ and $h(0,x) = x_0$ for some fixed $x_0\in U$.
	Suppose $\omega\in \Omega_{triv}(U)$ is a closed form of degree at least $1$
	 . From the Stokes formula, we have 
	
	$$d\int_I h^*\omega = \int_I h^*d\omega \pm (\omega - \omega_{x_0}) = \pm \omega,$$
	from where it follows that $\omega$ is exact.
\end{proof}


We can now conclude more generally that the cohomology of a semi-algebraic manifold $M$ agrees with the homology of $\Omega_{triv}(M)$.

\begin{cor}\label{cor:omegatriv omegaPA}
	Let $M$ be a compact semi-algebraic manifold, possibly with corners. The inclusion $\Omega_{triv}(M) \to \Omega_{\text{PA}}(M)$ is a quasi-isomorphism of commutative algebras.
\end{cor}

\begin{proof}
Every compact semi-algebraic manifold admits a good cover: Indeed, every compact semi-algebraic set has a finite semi-algebraic triangulation \cite[Theorem 9.2.1]{BCR}, and can hence be indentified with a finite simplicial complex, see also the discussion in \cite[Section 2]{HLTV}. 
Given a semi-algebraic triangulation, one can construct a semi-algebraic good cover $\{U_\alpha\}$ by taking the open stars of the vertices of a refinement of the triangulation\footnote{The star of a vertex $v$ is the union of the interiors of faces that contain $v$.}. 

We also choose a subordinate semi-algebraic partition of unity $\{\rho_\alpha\}$. For convenience we shall also pick cutoff functions $\sigma_\alpha$ with support in $U_\alpha$, such that $\sigma_\alpha(x)=1$ on the support of $\rho_\alpha$. (We may slightly enlarge the $U_\alpha$ to this purpose or alter the partition of unity, see also the proof of \cite[Proposition 6.7]{HLTV}.)

This allows us to run the standard \v{C}ech-de Rham argument with respect to such a good cover to conclude by the Poincar\'e lemma that the homology of $\Omega_{triv}(M)$ coincides with the (\v{C}ech) cohomology of $M$, see for instance \cite[Example 14.16]{BT}.

To be concrete, we consider the \v Cech-de Rham complex 
\[
  C := \left(\prod \Omega_{triv}(U_{\alpha_0\dots \alpha_p})[-p], d+\delta\right)
\]
with $U_{\alpha_0\dots \alpha_p}=U_{\alpha_0}\cap\cdots \cap U_{\alpha_p}$, $d$ induced from the differential on the factors $\Omega_{triv}(U_{\alpha_0\dots \alpha_p})$, and $\delta$ the \v Cech part of the differential, defined on a cochain $\omega = (\omega_{\alpha_0\dots \alpha_p})$ with $\omega_{\alpha_0\dots \alpha_p}\in U_{\alpha_0\dots \alpha_p}$ by
\[
(\delta \omega)_{\alpha_0\dots \alpha_p} = 
\sum_{i=0}^p (-1)^i \omega_{\alpha_0\dots \hat\alpha_i \dots \alpha_p}.
\]

The \v Cech-de Rham complex $C$ is a first quadrant double complex, and one compares the two convergent spectral sequences associated to this complex.

The first (``column-wise'') spectral sequence has the complex $(C,d)$ as its $E^0$-page. By the Poincar\'e Lemma (Proposition \ref{prop:Poincare}) the $E^1$-page is then identified with the \v Cech complex associated to the constant sheaf $\R$. The $E^2$-page is hence the cohomology $H(M)$, and the spectral sequence abuts at this point by degree reasons.

The other (``row-wise'') spectral sequence has first page $(C, \delta)$.
We claim that the cohomology of this page is identified with $\Omega_{triv}(M)$.
This can in fact be shown identically to \cite[Proposition 8.5]{BT}.
Concretely, one may naturally extend $(C,\delta)$ to a complex 
\[
\tilde C := (\Omega_{triv}(M)\xrightarrow{\delta}
C),	
\]
with the map 
\[
	\delta: \Omega_{triv}(M) \to \prod_\alpha \Omega_{triv}(U_\alpha) \subset C
\]
given by the natural restriction.
One then checks that $(\tilde C,\delta)$ is acyclic by providing an explicit homotopy.
Concretely, for a $p$-cocycle $\omega = (\omega_{\alpha_0\dots \alpha_p}) \in \tilde C$ 
one defines the $p-1$-cochain $\tau$ such that
\[
  \tau_{\alpha_0\dots \alpha_{p-1}}
  = \sum_\alpha \rho_\alpha \omega_{\alpha\alpha_0\dots\alpha_{p-1}}.
\]
Note that here we extend $\rho_\alpha\omega_{\alpha\alpha_0\dots\alpha_{p-1}}\in \Omega_{triv}(U_{\alpha\alpha_0\dots\alpha_{p-1}})$ by zero to an element (abusively also denoted by) $\rho_\alpha \omega_{\alpha\alpha_0\dots\alpha_{p-1}}$ of $\Omega_{triv}(U_{\alpha_0\dots\alpha_{p-1}})$. To be precise, this extension by zero may be defined as follows.
Suppose $$\omega_{\alpha\alpha_0\dots\alpha_{p-1}}=\int_Y \beta$$ is given by a fiber integral associated to the trivial bundle $Y\times U_{\alpha\alpha_0\dots\alpha_{p-1}}\to U_{\alpha\alpha_0\dots\alpha_{p-1}}$, with $\beta\in \Omega_{min}(Y\times U_{\alpha\alpha_0\dots\alpha_{p-1}})$. 
Then we extend $\rho_\alpha\beta$ (by zero) to a minimal form on $U_{\alpha_0\dots \alpha_{p-1}}$, which we (abusively) also denote by $\rho_\alpha\beta$. For example, if 
$\beta=(f_0,\dots,f_k)$ in the notation of \cite[section 5.2]{HLTV} we may take
$\rho_\alpha\beta:=(\rho_\alpha f_0,\sigma_\alpha f_1,\dots,\sigma_\alpha f_k)$, with all appearing semi-algebraic functions extended by zero, using our cutoff functions $\sigma_\alpha$.
Then one sets 
\[
	\rho_\alpha\omega_{\alpha\alpha_0\dots\alpha_{p-1}}
	=
	\int_Y \rho_\alpha\beta,
\]
with the fiber integral now being the one associated to the trivial semi-algebraic bundle $Y\times U_{\alpha_0\dots\alpha_{p-1}}\to U_{\alpha_0\dots\alpha_{p-1}}$. 

Having defined the cochain $\tau$ above one then checks as in the proof \cite[Proposition 8.5]{BT} that $\delta\tau=\omega$, using that $\delta\omega=0$.
Overall, we have then shown that the second ($E^1$-)page of the ``row-wise'' spectral sequence is identified with $(\Omega_{triv}(M), d)$.

We also note that this step of the proof is closely analogous to that of \cite[Lemma 6.7]{HLTV}, but slightly simpler since trivial bundles can be extended trivially.

The next page of the ``row-wise'' spectral sequence is then $H(\Omega_{triv}(M), d)$, and the spectral sequence converges at this point by degree reasons.
Hence $H(\Omega_{triv}(M), d)\cong H^\bullet(M)$.
It is shown in \cite{HLTV} that $H(\Omega_{PA}(M), d)\cong H(M)$. 
To see that the inclusion $\Omega_{triv}(M)\subset \Omega_{\text{PA}}(M)$ induces the isomorphism on cohomology one may consider the PA-\v Cech-de Rham complex $C_{\text{PA}}$, defined by replacing $\Omega_{triv}$ by $\Omega_{\text{PA}}$ in the definition of $C$ above. Using the PA-Poincar\'e Lemma (\cite[Lemma 6.3]{HLTV}) it is then clear that the natural inclusion $C\to C_{\text{PA}}$ induces an isomorphism on the $E^2$-page of the ``column-wise'' spectral sequences on both sides, and hence is a quasi-isomorphism.
\end{proof}

We note that in fact in the definition of $\Omega_{triv}$ we do not need globally trivial bundles, local triviality suffices.

\begin{prop}\label{prop:integration of triv forms}
Let $M$ be a compact semi-algebraic manifold and let $p: E\to M$ be an oriented SA bundle (see \cite[Definition 8.1]{HLTV}).
Let $\omega\in \Omega_{triv}(E)$. Then the corresponding fiber integral $\int_{E\to M}\omega\in \Omega_{\rm PA}(M)$ is an element of $\Omega_{triv}(M)\subset \Omega_{\rm PA}(M)$.
\end{prop}
\begin{proof}
We may assume that $\omega\in \Omega_{min}(M)$ by replacing $E$ with a product of $E$ with some trivial bundle if needed.
We pick a finite trivializing cover $\{U_i\}$, a semi-algebraic partition of unity $\rho_i$, and cutoff functions $\sigma_i$ as in the proof of Corollary \ref{cor:omegatriv omegaPA}.

We then rewrite
$$
\int_{E\to M}\omega
=
\sum_i \rho_i \int_{E\to M}\omega
=
\sum_i \int_{E\to M}\rho_i \omega.
$$
For the last equality we abused notation and defined $\rho_i:=p^*\rho_i$, and we implicitly used \cite[Proposition 8.9]{HLTV}.
Let the local traivialization of the bundle on $U_i$ be denoted by $h_i: U_i\times F\xrightarrow{\cong} p^{-1}(U_i)$.
As in the previous proof we extend the minimal form $\rho_i h_i^*\omega\in \Omega_{min}(U_i\times F)$ to a minimal form 
$\rho_i h_i^*\omega\in \Omega_{min}(M\times F)$,
which we abusively denote by the same symbols.
We then claim that 
\begin{equation}\label{equ:prop sa triv}
\int_{E\to M}\rho_i \omega
=
\int_{M\times F\to M}\rho_i h_i^*\omega.
\end{equation}
Since the right-hand side is a fiber integral over a trivial bundle the Proposition then follows.

To check \eqref{equ:prop sa triv} we need to consider a trivializing stratification $\{S_\alpha\}$ for the strongly continuous chain $\Phi$ corresponding to the bundle $E\to M$.
The stratification can be taken such that the closure of each stratum is contained in one of the $U_j$ as in the proof of \cite[Proposition 8.2]{HLTV}. We can furthermore refine it so that each $\bar S_\alpha$ is either contained in $U_i$ or disjoint from the support of $\rho_i$. (For example, refine the stratification by intersecting the strata with $\{x\mid \sigma_i(x)\geq 0.9\}$ and the closure of its complement.)

Now consider some stratum $S_\alpha$, and the restriction of \eqref{equ:prop sa triv} to its closure. If $\bar S_\alpha$ is disjoint from the support of $\rho_i$ then trivially both sides of \eqref{equ:prop sa triv} vanish on it. Otherwise we may assume that $\bar S_\alpha\subset U_i$.
But the bundle isomorphism $h_i$ transforms one side of \eqref{equ:prop sa triv} into the other, cf. \cite[Proposition 8.10]{HLTV}. 

\end{proof}
%



\end{document}